\documentclass[12pt,reqno]{amsproc}
\usepackage[top=1in, bottom=1in, left=.7in, right=.7in]{geometry}
\usepackage{color}
\usepackage[usenames,dvipsnames,svgnames,table]{xcolor}
\usepackage[arrow, curve, frame, matrix, graph, tips]{xy}
\usepackage{amssymb}
\usepackage{amsthm}
\usepackage{amsmath}
\usepackage[colorlinks=true,urlcolor=blue,linkcolor=blue,citecolor=Maroon]{hyperref}
\usepackage{tabularx}

\newtheorem{thm}[subsubsection]{\bf Theorem}
\newtheorem{prop}[subsubsection]{\bf Proposition}
\newtheorem{lem}[subsubsection]{\bf Lemma}
\newtheorem{cor}[subsubsection]{\bf Corollary}

\theoremstyle{definition}
\newtheorem{defn}[subsubsection]{\bf Definition}
\newtheorem{rem}[subsubsection]{\bf Remark}

\newtheorem{exam}[subsubsection]{\bf Example}
\newtheorem{exams}[subsubsection]{\bf Examples}
\newtheorem{const}[subsubsection]{\bf Construction}
\newtheorem{qn}[subsubsection]{\bf Question}
\newtheorem{conv}[subsubsection]{\bf Convention}

\newtheorem{app-thm}[subsection]{\bf Theorem}
\newtheorem{app-prop}[subsection]{\bf Proposition}
\newtheorem{app-lem}[subsection]{\bf Lemma}
\newtheorem{app-cor}[subsection]{\bf Corollary}
\newtheorem{app-scholium}[subsection]{\bf Scholium}
\theoremstyle{definition}
\newtheorem{app-defn}[subsection]{\bf Definition}
\newtheorem{app-rem}[subsection]{\bf Remark}
\newtheorem{app-rems}[subsection]{\bf Remarks}
\newtheorem{app-exam}[subsection]{\bf Example}
\newtheorem{app-exams}[subsection]{\bf Examples}
\newtheorem{app-const}[subsection]{\bf Construction}
\newtheorem{app-qn}[subsection]{\bf Question}
\newtheorem{app-conv}[subsection]{\bf Convention}

\newcommand{\tn}[1]{\textnormal{#1}}
\newcommand{\tnb}[1]{\textnormal{\bf #1}}

\newcommand{\tensor}{\otimes}

\newcommand{\N}{\mathbb{N}}
\newcommand{\comp}{\circ}
\newcommand{\id}{\tn{id}}
\newcommand{\ca}[1]{\mathcal{#1}}
\newcommand{\ladj}{\dashv}
\newcommand{\iso}{\cong}
\newcommand{\catequiv}{\simeq}
\newcommand{\Set}{\tnb{Set}}

\newcommand{\Cat}{\tnb{Cat}}

\newcommand{\Prof}{\tnb{Prof}}
\newcommand{\CAT}{\tnb{CAT}}
\renewcommand{\implies}{\Rightarrow}
\renewcommand{\impliedby}{\Leftarrow}

\newcommand{\op}{\tn{op}}
\newcommand{\co}{\tn{co}}

\newcommand{\TwoCAT}{\mathbf{2}{\tnb{-CAT}}}
\newcommand{\TwoCat}{\mathbf{2}{\tnb{-Cat}}}
\newcommand{\TwoProf}{\mathbf{2}{\tnb{-Prof}}}

\newcommand{\Polyc}[1]{{\tnb{Poly}_{#1}}}
\newcommand{\PFun}[1]{{\tnb{P}_{#1}}}

\newcommand{\PsAlg}[1]{\tn{Ps-}{#1}\tn{-Alg}}
\newcommand{\PsAlgs}[1]{\tn{Ps-}{#1}\tn{-Alg}_{\tn{s}}}
\newcommand{\PsAlgl}[1]{\tn{Ps-}{#1}\tn{-Alg}_{\tn{l}}}
\newcommand{\PsAlgc}[1]{\tn{Ps-}{#1}\tn{-Alg}_{\tn{c}}}
\newcommand{\Algl}[1]{{#1}\tn{-Alg}_{\tn{l}}}
\newcommand{\Algc}[1]{{#1}\tn{-Alg}_{\tn{c}}}

\newcommand{\Alg}[1]{{#1}\tn{-Alg}}
\newcommand{\Algs}[1]{{#1}\tn{-Alg}_{\tn{s}}}
\newcommand{\MND}{\tn{MND}}

\newcommand{\Enrich}[1]{#1{{\tnb{-}}{\tnb{Cat}}}}
\newcommand{\Mod}[1]{#1{{\tnb{-}}{\tnb{Mod}}}}
\newcommand{\CoDesc}{\tn{CoDesc}}
\newcommand{\CrIntCat}[1]{\tn{CrIntCat}(#1)}

\newcommand{\Cnr}{\tn{Cnr}}
\newcommand{\dbslash}{\backslash\backslash}

\newcommand{\PolyMnd}[1]{\tnb{PolyMnd}_{#1}}

\renewcommand{\P}{\mathbb{P}}
\newcommand{\B}{\mathbb{B}}
\renewcommand{\S}{\mathbb{S}}

\newcommand{\SMCMnd}{\tnb{S}}
\newcommand{\BMCMnd}{\tnb{B}}
\newcommand{\MCMnd}{\tnb{M}}

\newcommand{\FPMnd}{\tnb{P}_{\tn{fin}}}

\newcommand{\Com}{\mathsf{Com}}

\begin{document}

\title{Algebraic Kan extensions along morphisms of internal algebra classifiers}

\author{Mark Weber}
\address{Department of Mathematics,
Macquarie University}
\email{mark.weber.math@gmail.com}

%
\begin{abstract}
An ``algebraic left Kan extension'' is a left Kan extension which interacts well with the algebraic structure present in the given situation, and these appear in various subjects such as the homotopy theory of operads and in the study of conformal field theories. In the most interesting examples, the functor along which we left Kan extend goes between categories that enjoy universal properties which express the meaning of the calculation we are trying to understand. These universal properties say that the categories in question are universal examples of some categorical structure possessing some kind of internal structure, and so fall within the theory of ``internal algebra classifiers'' described in earlier work of the author. In this article conditions of a monad-theoretic nature are identified which give rise to morphisms between such universal objects, which satisfy the key condition of Guitart-exactness, which guarantees the algebraicness of left Kan extending along them. The resulting setting explains the algebraicness of the left Kan extensions arising in operad theory, for instance from the theory of ``Feynman categories'' of Kaufmann and Ward, generalisations thereof, and also includes the situations considered by Batanin and Berger in their work on the homotopy theory of algebras of polynomial monads.
\end{abstract}
\maketitle

\section{Introduction}
\label{sec:intro}

Categorical issues from the homotopy theory of operads and the study of conformal field theories, can involve the interaction between certain types of possibly quite complicated colimit calculation on the one hand, and algebraic structure on the other, with the difficulties contained in the compatibility between these two aspects in given situations. Many such situations can be organised as the study of taking left Kan extensions along a certain given functor $f : A \to B$, and studying what properties and structure on $f$ ensures that the process of left Kan extending along $f$ is compatible with further structure. For example if $A$ and $B$ are symmetric monoidal categories, one might want to know when the left Kan extension of a symmetric strong monoidal functor along $f$ ends up inheriting the structure of a symmetric strong monoidal functor.

The general theory of ``algebraic Kan extensions'' is concerned with such questions. In the various ways of making this subject precise, one begins with a structure within which one has some formal notion of left Kan extension. Given an instance $\ca K$ of such a structure, and an appropriate type of monad $T$ on $\ca K$, one can consider also left Kan extensions within the structure $\ca A$ formed by the algebras of $T$. Then the basic result is that given morphisms
\[ \xygraph{
{A}="p0" [r(2)] {B}="p1" [dl] {X}="p2"
"p0" (:"p1"^-{f},:"p2"_-{g})} \]
algebras, there are natural conditions on $f$ and $X$ which ensure that left Kan extending $g$ along $f$ down in $\ca K$, becomes a left Kan extension in $\ca A$. The condition on $X$ is that such left Kan extensions exist in $\ca K$, and its monad algebra structure is compatible with them. For instance in the context of symmetric monoidal categories mentioned above, the condition of being cocomplete and having the tensor product commute with colimits in each variable, would be sufficient for all situations, but often one can get away with a lot less.

As for the condition on $f : A \to B$, this is formalised in terms of the notion, or the appropriate analogue of it within $\ca K$, of \emph{exact square} introduced by Ren\'e Guitart in \cite{Guitart-ExactSquares}. Since $f$ is some kind of monad algebra morphism, there is a square which expresses it, which could either commute on the nose, or be the boundary of some coherence cell which could be invertible or not, and it is this algebra square which one requires to be exact.

Versions of the theory of algebraic Kan extensions have been given, the first of these being the unpublished article \cite{MelliesTabareau-TAlgTheoriesKan} of Melli\`es and Tabareau, in which the formal setting is that of a proarrow equipment in the sense of Wood \cite{Wood-Proarrows-I, Wood-Proarrows-II}. More recently this subject was part of the PhD thesis \cite{Koudenburg-Thesis} of Roald Koudenburg, and the resulting publication \cite{Koudenburg-AlgKan}. In both cases the formal settings chosen involve quite a bit of metastructure, particularly from the point of view of the non-category theorist. Thus in our exposition of this theory in Section \ref{sec:AlgKan}, we choose the minimalistic setting of a 2-category $\ca K$ with comma objects together with a general 2-monad on it, and take Ross Street's notion of pointwise left Kan extension from \cite{Street-FibrationIn2cats}. While this is at a cost of some generality, it leads us most efficiently to the formulation of our main result in Section \ref{sec:formulating-main-result}, while being completely adequate for a wide range of interesting examples.

From this general theory, the main issue to understand when studying left Kan extensions along an algebra morphism $f : A \to B$, is whether or not $f$'s algebra square is exact. In particular examples this can be as difficult as $f$ and the given monad are complicated. Thus it is of interest to have general results which guarantee $f$'s exactness. As explained in the introduction to \cite{BataninKockWeber-Feynman}, there is a similar very related issue with Getzler's notion of ``regular pattern'', in that the corresponding key condition that ensures well-behaved left Kan extensions can be difficult to check in some examples. Kaufmann and Ward's context of a Feynman category \cite{KaufmannWard-FeynmanCats} is one in which such issues have been understood. The recent article \cite{BataninKockWeber-Feynman} makes a direct connection between the settings of Getzler and Kaufmann-Ward, with the notion of Guitart exactness.

These developments refer to the important case when the ambient structure required to be compatible with colimits is that of a symmetric monoidal category. However from the work of Batanin and Berger \cite{BataninBerger-HtyThyOfAlgOfPolyMnd}, there are many interesting situations in which the ambient structure can be otherwise. Moreover, in all the interesting examples of morphisms $f : A \to B$ along which we wish to left Kan extend, $A$ and $B$ enjoy fundamental universal properties, intricately linked to the meaning of the calculation one is trying to understand. Indeed, the underlying philosophy of \cite{BataninBerger-HtyThyOfAlgOfPolyMnd}, which goes back to the Batanin's seminal paper \cite{Batanin-EckmannHilton}, is that organising one's calculations conceptually via these universal properties is fundamentally useful.

In the language of \cite{Weber-CodescCrIntCat} $A$ and $B$ are ``internal algebra classifiers'', meaning that they are universal examples of some kind of categorical structure possessing some kind of internal structure. Each of these structures are expressed by 2-monads. The ability to speak of one type of structure being internal to the other, follows when one has an adjunction of 2-monads between them in the sense defined in \cite{Weber-CodescCrIntCat}. Unfortunately, since the setting of \cite{BataninBerger-HtyThyOfAlgOfPolyMnd} involves polynomial monads defined over $\Set$, one cannot apply this universal perspective directly to the symmetric monoidal category monad. This is because, as seen in \cite{Weber-PolynomialFunctors}, this monad comes from a polynomial defined over $\Cat$, and so the computation of the corresponding internal algebra classifiers is more involved. However in \cite{Weber-CodescCrIntCat} the computation of such internal algebra classifiers was understood.

In this article we provide a monad-theoretic setting which extends that of Batanin and Berger so that it does include the symmetric monoidal category monad. In the main result, Theorem \ref{thm:main}, one has a general situation involving two related types of internal structure expressable within a type of ambient structure. Following \cite{Weber-CodescCrIntCat} this is formalised in terms of three 2-monads and adjunctions of 2-monads between them, with the resulting setting giving rise to forgetful functors between the categories of different types of internal structure. In the example of the modular envelope construction discussed in Example \ref{exam:modular-envelope}, these are the forgetful functors, for each symmetric monoidal category $\ca V$, from the category of modular operads in $\ca V$ to the category of cyclic operads in $\ca V$. The left adjoint to these forgetful functors, when it exists, is the modular envelope construction. In this context $\ca V$'s symmetric monoidal category structure is expressed as a pseudo algebra structure for the symmetric monoidal category 2-monad.

Denoting by $R$ and $S$ the 2-monads which describe the two types of internal structure, and by $T$ the 2-monad which describes the ambient structure, one has the internal algebra classifiers $T^R$ and $T^S$ for internal $R$ and $S$ algebras respectively. From \cite{Weber-CodescCrIntCat} the meaning of these objects is that $R$-algebras internal to a pseudo $T$-algebra $A$ may be identified with pseudomorphisms $T^R \to A$, and similarly for internal $S$-algebras. Moreover our setting gives rise to a strict $T$-algebra morphism $T^R \to T^S$, and the forgetful functors described above correspond to the process of precomposition with it. Theorem \ref{thm:main} gives conditions on $R$, $S$ and $T$ ensuring that $T^R \to T^S$ is exact, so that the left adjoints to the above forgetful functors are obtained by algebraic left extension along $T^R \to T^S$. In \cite{Costello-AInftyOpAndModuliSpaceCurves} the modular envelope construction was defined by directly specifying a symmetric monoidal functor $f : A \to B$ along which to left Kan extend. Applied to this example, Theorem \ref{thm:main} together with the developments of \cite{BataninBerger-HtyThyOfAlgOfPolyMnd, Weber-CodescCrIntCat, Weber-OpPoly2Mnd}, clarify why the categories $A$ and $B$ and the strict monoidal functor $f$ are what they are (their direct definition being somewhat combinatorially involved), why $A$ and $B$ enjoy the expected universal properties, and why left Kan extending along $f$ produces the modular envelope construction.

{\bf Organisation of this article}. Section \ref{sec:AlgKan} gives a self-contained account of the theory of algebraic Kan extensions sufficient for our purposes. Most of Section \ref{sec:formulating-main-result} is devoted to giving the precise formulation of Theorem \ref{thm:main}, which in addition to the aspects discussed above, also involves notions from the theory of polynomial 2-functors defined over $\Cat$ \cite{Weber-PolynomialFunctors}. Thus if one was just interested in a precise formulation of the main result, then it would suffice to read just until the end of Section \ref{ssec:MainTheorem}.

In Section \ref{ssec:operadic-examples} we discuss applications to operad theory, and extend what is known about algebraic Kan extensions arising from coloured symmetric operads, to the non-symmetric and braided cases. Moreover the results of this paper apply to establishing all the algebraic left extensions that are used in \cite{BataninBerger-HtyThyOfAlgOfPolyMnd} for the homotopy theory of operads. In a subsequent article the technology of this article will be used to understand the construction of colimits in categories of internal algebras. In particular this will bring Batanin and Berger's insights on the calculation of semifree coproducts and semifree pushouts into a setting which includes the symmetric monoidal category monad. Further applications to homotopy theory are thus anticipated.

The correct notion of exact square in any 2-categorical context, is determined by what the notion of pointwise left Kan extension is in that context. In Section \ref{sec:ExactSquares} we describe the general theory of exact squares corresponding to Ross Street's notion of pointwise left Kan extension in \cite{Street-FibrationIn2cats}. The main result here is Proposition \ref{prop:exact-pullbacks}, which explains conditions under which pullbacks and bipullbacks are exact. In fact this result is sufficient to deal with all the algebraic Kan extensions of \cite{BataninBerger-HtyThyOfAlgOfPolyMnd}. Later on in this article, we isolate this special case in Theorem \ref{thm:exact-internalisations-easy}, which does not require any of the developments of \cite{Weber-CodescCrIntCat} and Sections \ref{ssec:codesc-crossed-dblcat}-\ref{ssec:exactness-of-internalisations} of this article.

In the remainder of Section \ref{sec:ExactSquares} we describe further results of general interest which do not bear directly on the proof of the central result. In Section \ref{ssec:general-exactness-in-colaxidem-case} we give a result which also appears in \cite{Koudenburg-AlgKan}, that all algebra morphisms are exact when $T$ is colax idempotent. In Section \ref{ssec:exact-colax-monoidal-functors} we obtain explicit characterisations of exact monoidal, braided monoidal and symmetric monoidal functors, and give some natural non-examples to contrast with the colax idempotent case. In Section \ref{ssec:exact-nat} we observe that for many of our examples, the unit and multiplication of the 2-monads we consider have naturality squares which exact in both possible senses, and that this gives rise to the ability to transfer algebraic cocompleteness across an adjunction of 2-monads.

Section \ref{sec:exact-via-codescent} is concerned with the deeper interactions between the codescent calculations of internal algebra classifiers understood in \cite{Weber-CodescCrIntCat}, and exact squares. The key technical result in this regard is Theorem \ref{thm:hard-exactness} whose proof occupies Sections \ref{ssec:codesc-crossed-dblcat}-\ref{ssec:exactness-of-internalisations}. This result says that applying the functor whose object map takes codescent objects of crossed double categories in the sense of \cite{Weber-CodescCrIntCat} to a square, which satisfies a double categorical mixture of the hypotheses of Proposition \ref{prop:exact-pullbacks}, produces an exact square in $\Cat$. Then in Sections \ref{ssec:int-alg-class-codesc} and \ref{ssec:alg-square-before-codescent} it is explained how to apply this result to our monad theoretic context giving the proof of Theorem \ref{thm:main}.

{\bf Acknowledgments.} My interest in this subject began with the very inspiring work of Batanin \cite{Batanin-EckmannHilton} in which internal algebras were used to shed light on configuration spaces. I am heavily indebted to Michael Batanin for so generously sharing his insights. While working in Paris I was introduced to exact squares and algebraic Kan extensions by Paul-Andr\'e Melli\`es, long before beginning to think seriously about this project. More recently, illuminating discussions with Ross Street helped me to navigate through the world of lax coends, which appear in Section \ref{sec:exact-via-codescent}. There are variants of the main result of this paper, and it was in discussions with Joachim Kock that it became clear to me that for expository purposes, the variant based on polynomial monads which appears here is probably the most illuminating. Finally, discussions with Roald Koudenburg helped me to optimise some parts of Section \ref{sec:AlgKan}. I am grateful also for the support of the Australian Research Council grant No. DP130101172.

\section{Algebraic Kan extensions}
\label{sec:AlgKan}

In this section we reformulate some of the theory of algebraic left Kan extensions so that it applies for a 2-monad $(\ca K,T)$. As mentioned above, the basic ideas and results of this section are not new. Indeed in the double categorical setting of \cite{Koudenburg-AlgKan} one has versions of our Theorem \ref{thm:AlgKan}, Corollary \ref{cor:AlgKan} and also of Proposition \ref{prop:exactness-coKZ} in Section \ref{sec:ExactSquares}.

\subsection{2-monads}
\label{ssec:2-monad-theory}

A \emph{2-monad} on a 2-category $\ca K$ is just the $\Cat$-enriched version of a monad. As such, a 2-monad consists of a 2-functor $T :\ca K \to \ca K$, and 2-natural transformations $\eta : 1_{\ca K} \to T$ (the ``unit'') and $\mu : T^2 \to T$ (the ``multiplication''), satisfying the usual axioms.  In our 2-categorical context, there are weaker notions of monad involving the usual axioms holding up to coherent isomorphism, but these will not concern us here. We will often denote a 2-monad as a pair $(\ca K,T)$ leaving the unit and multiplication implicit. We will always use the symbols $\eta$ and $\mu$ for the unit and multiplication of a 2-monad, and when there is more than one 2-monad present in a given context, $T$'s unit and multiplication will be denoted as $\eta^T$ and $\mu^T$.

By contrast with ordinary category theory, 2-monads can have different types of algebras: lax, colax, pseudo and strict; and different types of morphisms of algebras. Let $(\ca K,T)$ be a 2-monad. Recall that for $A \in \ca K$, a \emph{pseudo $T$-algebra structure} on $A$ consists of an arrow $a:TA \to A$, invertible coherence 2-cells $a_0:1_A \to a\eta_A$ and $a_2:aT(a) \to a\mu_A$, satisfying the following axioms:
\[ \xygraph{!{0;(4.5,0):}
{\xybox{\xygraph{!{0;(2,0):(0,.5)::} {a}="l" [r] {a\eta_Aa}="m" [d] {a}="r" "l":"m"^-{a_0a}:"r"^-{a_2\eta_{TA}}:@{<-}"l"^-{\id} "l" [d(.35)r(.7)] {\scriptsize{=}}}}}
[r]
{\xybox{\xygraph{!{0;(2.5,0):(0,.4)::} {aT(a)T^(a)}="tl" [r] {a\mu_AT^2(a)}="tr" [d] {a\mu_A\mu_{TA}}="br" [l] {aT(a)T(\mu_A)}="bl" "tl":"tr"^-{a_2T^2(a)}:"br"^-{a_2\mu_{TA}}:@{<-}"bl"^-{a_2T(\mu_A)}:@{<-}"tl"^-{aT(a_2)} "tl" [d(.5)r(.5)] {\scriptsize{=}}}}}
[r]
{\xybox{\xygraph{!{0;(2,0):(0,.5)::} {a}="l" [l] {aT(a)T(\eta_A)}="m" [d] {a}="r" "l":"m"_-{aT(a_0)}:"r"_-{a_2T(\eta_A)}:@{<-}"l"_-{\id} "l" [d(.35)l(.7)] {\scriptsize{=}}}}}} \]
which we shall call the \emph{left unit axiom}, the \emph{associativity axiom} and the \emph{right unit axiom} respectively. We denote a pseudo $T$-algebra as a pair $(A,a)$ leaving the data $a_0$ and $a_2$ implicit, and we will sometimes speak of the pseudo algebra $A$ when we wish $a$ to be implicit also. When $a_0$ and $a_2$ are identities, $(A,a)$ is said to be a \emph{strict} $T$-algebra.

A \emph{lax morphism} $(A,a) \to (B,b)$ between pseudo $T$-algebras is a pair $(f,\overline{f})$, where $f:A \to B$ and $\overline{f}:bT(f) \to fa$, satisfying the following axioms:
\[ \xygraph{
{\xybox{\xygraph{{f}="l" [dl] {bT(f)\eta_A}="m" [r(2)] {fa\eta_A}="r" "l":"m"_-{b_0f}:"r"_-{\overline{f}\eta_A}:@{<-}"l"_-{fa_0} "l" [d(.6)] {\scriptsize{=}}}}} [r(6)]
{\xybox{\xygraph{!{0;(1.5,0):(0,.6667)::} {bT(b)T^2(f)}="p1" [r(2)] {b\mu_BT^2(f)}="p2" [d] {fa\mu_A}="p3" [dl] {faT(a)}="p4" [ul] {bT(fa)}="p5" "p1":"p2"^-{b_2T^2(f)}:"p3"^-{\overline{f}\mu_A}:@{<-}"p4"^-{fa_2}:@{<-}"p5"^-{\overline{f}T(a)}:@{<-}"p1"^-{bT(\overline{f})} "p1" [d(.8)r] {\scriptsize{=}}}}}} \]
which we shall call the \emph{unit} and \emph{structure} axioms respectively. A \emph{colax morphism} is defined the same way, except that the coherence cell is reversed $\overline{f} : fa \to bT(f)$. When $\overline{f}$ is an isomorphism, $f$ is said to be a pseudomorphism, and when $\overline{f}$ is an identity, $f$ is said to be a strict morphism of algebras. Given lax $T$-algebra morphisms $f$ and $g:(A,a) \to (B,b)$, a $T$-algebra 2-cell $f \to g$ is a 2-cell $\phi:f \to g$ in $\ca K$ such that $\overline{g}(bT(\phi))=(\phi a)\overline{f}$, 2-cells of colax morphisms are defined similarly, and these notions agree in the pseudo case.

Given the various notions of algebras and algebra morphisms, there are various 2-categories of $T$-algebras. Those used in this article are described in the following table.

\begin{table}[h]
\centering
\begin{tabular}{|l|l|l|}
\hline
Name & Objects & Arrows \\ \hline \hline
$\PsAlgl{T}$ & pseudo $T$-algebras & lax morphisms \\ \hline
$\PsAlgc{T}$ & pseudo $T$-algebras & colax morphisms \\ \hline
$\PsAlg{T}$ & pseudo $T$-algebras & pseudomorphisms \\ \hline
$\PsAlgs{T}$ & pseudo $T$-algebras & strict morphisms \\ \hline
$\Algl{T}$ & strict $T$-algebras & lax morphisms \\ \hline
$\Algc{T}$ & strict $T$-algebras & colax morphisms \\ \hline
$\Alg{T}$ & strict $T$-algebras & pseudomorphisms \\ \hline
$\Algs{T}$ & strict $T$-algebras & strict morphisms \\ \hline
\end{tabular}
\end{table}
\noindent In each case, the 2-cells are just the $T$-algebra 2-cells between the appropriate $T$-algebra morphisms.

The basic examples of 2-monads to keep in mind in this article are
\begin{enumerate}
\item $\MCMnd$ for monoidal categories,
\item $\SMCMnd$ for symmetric monoidal categories,
\item $\BMCMnd$ for braided monoidal categories, and
\item $\FPMnd$ for categories with finite products
\end{enumerate}
all on $\ca K = \Cat$, and are discussed at length in Section 5 of \cite{Weber-PolynomialFunctors}. In particular they are all examples of \emph{polynomial 2-monads}. See also Section 2 of \cite{Weber-CodescCrIntCat} for an exposition.

\subsection{Algebraic left extensions}
\label{ssec:alg-left-ext}
In a 2-category with comma objects one has the notion of a pointwise left Kan extension \cite{Street-FibrationIn2cats}. Moreover, given a 2-monad $T$ on a 2-category $\ca K$ with comma objects, comma objects in $\PsAlg T$ are computed as in $\ca K$, and the projections are strict morphisms \cite{BWellKellyPower-2DMndThy, LackShul-Enhanced}. Thus in particular one can speak of pointwise left Kan extensions in $\PsAlg T$.
\begin{defn}\label{def:alg-left-extension}
Let $(\ca K,T)$ be a 2-monad and suppose that $\ca K$ has comma objects. Given pseudomorphisms $(g,\overline{g})$ and $(f,\overline{f})$ between pseudo $T$-algebras as on the left
\[ \xygraph{{\xybox{\xygraph{{(I,i)}="p0" [r(2)] {(J,j)}="p1" [dl] {(A,a)}="p2" "p0":"p1"^-{(f,\overline{f})}:"p2"^-{(h,\overline{h})}:@{<-}"p0"^-{(g,\overline{g})} "p0" [d(.5)r(.85)] :@{=>}[r(.3)]^-{\psi}}}}
[r(4)]
{\xybox{\xygraph{{I}="p0" [r(2)] {J}="p1" [dl] {A}="p2" "p0":"p1"^-{f}:"p2"^-{h}:@{<-}"p0"^-{g} "p0" [d(.5)r(.85)] :@{=>}[r(.3)]^-{\psi}}}}} \]
we say that $(g,\overline{g})$ \emph{admits algebraic left extension along} $(f,\overline{f})$ when
\begin{enumerate}
\item The pointwise left Kan extension of $g$ along $f$ exists in $\ca K$.
\label{ale-cond-1}
\item For any such pointwise left Kan extension $(h,\psi)$ in $\ca K$ as on the right in the previous display, there exists a unique isomorphism $\overline{h} : aT(h) \to hj$ making $(h,\overline{h})$ a pseudomorphism, and $\psi$ an algebra 2-cell which exhibits $(h,\overline{h})$ as the pointwise left Kan extension of $(g,\overline{g})$ along $(f,\overline{f})$ in $\PsAlg T$.
\label{ale-cond-2}
\end{enumerate}
\end{defn}
In less formal terms, $(g,\overline{g})$ admits algebraic left extension along $(f,\overline{f})$, when the pointwise left extension of $(g,\overline{g})$ along $(f,\overline{f})$ in $\PsAlg T$ exists, and is computed as in $\ca K$. The theory of algebraic Kan extensions addresses the following
\begin{qn}\label{qn:alg-kan-ext}
What conditions on $(f,\overline{f})$ and $(A,a)$ ensure that \emph{every} pseudo morphism $(g,\overline{g})$ admits algebraic left extension along $f$?
\end{qn}
\begin{exam}\label{exam:P-fin-alg-left-ext}
Let $I$, $J$ and $A$ be categories with finite products, with $I$ and $J$ small and $A$ cocomplete cartesian closed. When $T = \FPMnd$, the 2-monad on $\Cat$ for the categories with finite products, by having finite products, $I$, $J$ and $A$ are pseudo $\FPMnd$-algebras. A pseudo morphism in this context is a finite product preserving functor. In the context of Definition \ref{def:alg-left-extension}, the cocompleteness of $A$ ensures that (\ref{ale-cond-1}) holds, and the classical fact \cite{Lawvere-FunctorialSemantics}: ``the pointwise left Kan extension of a finite product preserving functor $g : I \to A$ is finite product preserving''; ends up implying condition (\ref{ale-cond-2}). 
\end{exam}

\subsection{Algebraic cocompleteness}
\label{ssec:alg-cocomp}
The reason that the classical fact recalled in Example \ref{exam:P-fin-alg-left-ext} is true, is that cartesian closedness ensures that $A$'s colimits are compatible with its pseudo $\FPMnd$-algebra structure, in the sense that $(-) \times X : A \to A$ is colimit preserving for all $X \in A$. In the general situation, such compatibility of colimits with algebraic structure is given by
\begin{defn}\label{def:algebraic-cocompleteness}
Let $T$ be a 2-monad on a 2-category $\ca K$ with comma objects and $f : I \to J$ be an arrow of $\ca K$. Then a pseudo $T$-algebra $(A,a)$ is \emph{algebraically cocomplete relative to} $f$ when
\begin{enumerate}
\item For all $g:I \to A$, the pointwise left extension of $g$ along $f$ exists in $\ca K$.
\item If $\psi$ exhibits $h$ as a pointwise left Kan extension of $g$ along $f$ in $\ca K$
\[ \xygraph{{\xybox{\xygraph{{I}="l" [r(2)] {J}="r" [dl] {A}="b" "l":"r"^-{f}:"b"^-{h}:@{<-}"l"^-{g} [d(.5)r(.85)] :@{=>}[r(.3)]^-{\psi}}}} [r(3.5)]
{\xybox{\xygraph{{TI}="l" [r(2)] {TJ}="r" [dl] {TA}="b" [d] {A}="bb" "l":"r"^-{Tf}:"b"^-{Th}:@{<-}"l"^-{Tg} "b":"bb"^{a} "l" [d(.5)r(.85)] :@{=>}[r(.3)]^-{T\psi}}}}} \]
then $aT(\psi)$ exhibits $aT(h)$ as a pointwise left Kan extension of $aT(g)$ along $T(f)$.
\end{enumerate}
\end{defn}
Propositions \ref{prop:algcocomp-M}, \ref{prop:algcocomp-Sm-or-Bm} and \ref{prop:algcocomp-Fp} explain how, in the cases $T = \MCMnd$, $\SMCMnd$, $\BMCMnd$ and $\FPMnd$, Definition \ref{def:algebraic-cocompleteness} captures the usual idea of the categorical structure encoded by the 2-monad being compatible with colimits. In the proofs of these results we use the well-known fact that if $\psi$ as on the left
\[ \xygraph{{\xybox{\xygraph{{I}="p0" [r(2)] {J}="p1" [dl] {\ca V}="p2" "p0":"p1"^-{f}:"p2"^-{h}:@{<-}"p0"^-{g} "p0" [d(.5)r(.85)] :@{=>}[r(.3)]^-{\psi}}}}
[r(4)]
{\xybox{\xygraph{{I}="p0" [r(2)] {J}="p1" [dl] {\ca V}="p2" "p0" [u] {P}="p3" [r(2)] {K}="p4"
"p0":"p1"^-{f}:"p2"^-{h}:@{<-}"p0"^-{g} "p3" (:"p0"_-{p},:"p4"^-{q}:"p1"^-{k})
"p0" [d(.5)r(.85)] :@{=>}[r(.3)]^-{\psi} "p3":@{}"p1"|-{\tn{pb}}}}}} \]
exhibits $h$ as a pointwise left Kan extension of $g$ along $f$, and $f$ is an opfibration, then for any $k : K \to J$, the composite on the right exhibits $hk$ as a left Kan extension of $gp$ along $q$ \cite{Street-FibrationIn2cats}. Moreover we use the fact that $\MCMnd$, $\BMCMnd$, $\SMCMnd$ and $\FPMnd$ preserve opfibrations \cite{Weber-Fam2fun, Weber-PolynomialFunctors}.
\begin{prop}\label{prop:algcocomp-M}
Let $\ca V$ be a monoidal category.
\begin{enumerate}
\item $\ca V$ is algebraically cocomplete relative to all functors between small categories as a pseudo $\MCMnd$-algebra iff $\ca V$ is cocomplete and its tensor product preserves colimits in each variable.\label{pcase:monoidally-cocomplete}
\item $\ca V$ is algebraically cocomplete relative to all functors between small discrete categories as a pseudo $\MCMnd$-algebra iff $\ca V$ has coproducts and its tensor product preserves coproducts in each variable.\label{pcase:dist-monoidal}
\end{enumerate}
\end{prop}
\begin{proof}
(\ref{pcase:monoidally-cocomplete})($\implies$): Suppose that small categories $I_k$ for $1 \leq k \leq n$, and colimit cocones as on the left in
\[ \xygraph{{\xybox{\xygraph{{I_k}="p0" [r(2)] {1}="p1" [dl] {\ca V}="p2" "p0":"p1"^-{}:"p2"^-{h_k}:@{<-}"p0"^-{g_k} "p0" [d(.5)r(.85)] :@{=>}[r(.3)]^-{\psi_k}}}}
[r(4)u(.05)]
{\xybox{\xygraph{{I}="p0" [r(2)] {n}="p1" [dl] {\ca V}="p2" "p0":"p1"^-{f}:"p2"^-{h}:@{<-}"p0"^-{g} "p0" [d(.5)r(.85)] :@{=>}[r(.3)]^-{\psi}}}}} \]
are given. Define $I = \coprod_{k=1}^n I_k$, denote by $n$ also the discrete category $\{1,...,n\}$, define $f : I \to n$ as the functor which sends $I_k$ to $k \in n$, and then define $g$, $h$ and $\psi$ so that $\psi_i = (\psi_{k})_i$ for $i \in I_k$. Then $\psi$ is easily verified to exhibit $h$ as a pointwise left Kan extension of $g$ along $f$. Since $f$ is an opfibration, so is $\MCMnd(f)$, and so since $\ca V$ is algebraically cocomplete the composite
\[ \xygraph{!{0;(0,-.75):(0,-3)::} {\MCMnd(I)}="l" [r(2)] {\MCMnd(n)}="r" [dl] {\MCMnd(\ca V)}="b" :[d(.75)] {\ca V}^{\bigotimes} "l":"r"_-{\MCMnd(f)}:"b"_-{\MCMnd(h)}:@{<-}"l"_-{\MCMnd(g)} [d(.5)r(.85)] :@{=>}[r(.3)]_-{\MCMnd(\psi)}
"l" [u] {\prod\limits_{k=1}^n I_k}="tl" [r(2)] {1}="tr" "tl"(:"tr":"r"_{(1,...,n)},:"l") "tl":@{}"r"|{\tn{pb}}} \]
is a colimit cocone. In the case where, for a given $1 \leq k \leq n$, $I_l = 1$ and $\psi_l = \id$ for $l \neq k$, this says that $\bigotimes : \ca V^n \to \ca V$ preserves colimits in the $k$-th variable.

(\ref{pcase:dist-monoidal})($\implies$): same as (\ref{pcase:monoidally-cocomplete})($\implies$) but with the $I_k$ assumed discrete.

(\ref{pcase:monoidally-cocomplete})($\impliedby$): given the hypotheses on $\ca V$ and a functor $f : I \to J$ between small categories, we must verify that for any sequence $j = (j_1,...,j_n)$ of objects of $J$, that the composite
\begin{equation}\label{eq:colimcocone-algcocomp-proof}
\xygraph{!{0;(2.5,0):(0,.24)::} {\MCMnd(f) \downarrow j}="p0" [r] {\MCMnd(I)}="p1" [dr] {\MCMnd(\ca V)}="p2" ([dl] {\MCMnd(J)}="p3" [l] {1}="p4", [r(.75)] {\ca V}="p5") "p0":"p1"^-{p}:"p2"^-{\MCMnd(g)}:@{<-}"p3"^-{\MCMnd(h)}:@{<-}"p4"^-{(j_1,...,j_n)}:@{<-}"p0"^-{}
"p1":"p3"_-{\MCMnd(f)} "p2":"p5"^-{\bigotimes}
"p0" [d(.8)r(.5)] :@{=>}[d(.4)]^-{\lambda} "p1" [d(.8)r(.4)] :@{=>}[d(.4)]_-{\MCMnd(\psi)}}
\end{equation}
is a colimit cocone, for any natural transformation $\psi$ which exhibits $h$ as a pointwise left Kan extension of $g$ along $f$. Note that $\MCMnd(f) \downarrow j \iso \prod_{k=1}^n f \downarrow j_k$ and that $\bigotimes \MCMnd(g) p$ is the composite functor
\[ \xygraph{!{0;(1.5,0):(0,1)::} {\prod_{k=1}^n f \downarrow j_k}="p0" [r(1.5)] {I^n}="p1" [r] {\ca V^n}="p2" [r] {\ca V}="p3" "p0":"p1"^-{\prod_k p_k}:"p2"^-{g^n}:"p3"^-{\bigotimes}} \]
where $p_k$ is the comma projection $p_k : f \downarrow j_k \to I$. The component of (\ref{eq:colimcocone-algcocomp-proof}) at $(\alpha_k : fi_k \to j_k)_{1{\leq}k{\leq}n}$ is the composite
\[ \xygraph{!{0;(3,0):(0,1)::} {\bigotimes_{k=1}^n gi_k}="p0" [r] {\bigotimes_{k=1}^n hfi_k}="p1" [r] {\bigotimes_{k=1}^n hj_k}="p2" "p0":"p1"^-{\bigotimes_k \psi_{i_k}}:"p2"^-{\bigotimes_k \alpha_k}} \]
which in each variable is a colimit cocone since $\psi$ is a pointwise left Kan extension.

(\ref{pcase:dist-monoidal})($\impliedby$): same argument as for (\ref{pcase:monoidally-cocomplete})($\impliedby$), except that now $I$ and $J$ are discrete, and so since $\MCMnd(f)$ is then an opfibration, one may replace the comma object in (\ref{eq:colimcocone-algcocomp-proof}) by pullback square.
\end{proof}
\begin{prop}\label{prop:algcocomp-Sm-or-Bm}
Let $\ca V$ be a symmetric (resp. braided) monoidal category.
\begin{enumerate}
\item $\ca V$ is algebraically cocomplete relative to all functors between small categories as a pseudo $\SMCMnd$-algebra (resp. as a pseudo $\BMCMnd$-algebra) iff $\ca V$ is cocomplete and its tensor product preserves colimits in each variable.\label{pcase:s-or-b-monoidally-cocomplete}
\item $\ca V$ is algebraically cocomplete relative to all functors between small discrete categories as a pseudo $\SMCMnd$-algebra (resp. as a pseudo $\BMCMnd$-algebra) iff $\ca V$ has coproducts and its tensor product preserves coproducts in each variable.\label{pcase:dist-s-or-b-monoidal}
\end{enumerate}
\end{prop}
\begin{proof}
One verifies (\ref{pcase:s-or-b-monoidally-cocomplete})($\implies$) and (\ref{pcase:dist-s-or-b-monoidal}) by the same arguments as for Proposition \ref{prop:algcocomp-M}. We give here the proof of (\ref{pcase:s-or-b-monoidally-cocomplete})($\impliedby$) in the symmetric case. The braided case is similar. Given $\psi$ as in the proof of Proposition \ref{prop:algcocomp-M}(\ref{pcase:monoidally-cocomplete})($\impliedby$), we must verify that
\[ \xygraph{!{0;(2,0):(0,.3)::} {\SMCMnd(f) \downarrow j}="p0" [r] {\SMCMnd(I)}="p1" [dr] {\SMCMnd(\ca V)}="p2" ([dl] {\SMCMnd(J)}="p3" [l] {1}="p4", [r(.75)] {\ca V}="p5") "p0":"p1"^-{p}:"p2"^-{\SMCMnd(g)}:@{<-}"p3"^-{\SMCMnd(h)}:@{<-}"p4"^-{(j_1,...,j_n)}:@{<-}"p0"^-{}
"p1":"p3"_-{\SMCMnd(f)} "p2":"p5"^-{\bigotimes}
"p0" [d(.8)r(.45)] :@{=>}[d(.4)]^-{\lambda} "p1" [d(.8)r(.45)] :@{=>}[d(.4)]_-{\SMCMnd(\psi)}} \]
is a colimit cocone. An object of $\SMCMnd(f) \downarrow j$ is a triple $(i,\rho,\alpha)$ as on the left
\[ \xygraph{{\xybox{\xygraph{{n}="p0" [r] {n}="p1" [d] {J}="p2" [l] {I}="p3" "p0":"p1"^-{\rho}:"p2"^-{j}:@{<-}"p3"^-{f}:@{<-}"p0"^-{i} "p0" [d(.55)r(.35)] :@{=>}[r(.3)]^-{\alpha}}}}
[r(3)u(.1)]
{\xybox{\xygraph{!{0;(.5,0):(0,2)::} {n}="p0" [r(2)] {n}="p1" [dl] {I}="p2" "p0":"p1"^-{\rho''}:"p2"^-{i'}:@{<-}"p0"^-{i} "p0" [d(.45)r(.75)] :@{=>}[r(.5)]^-{\beta}}}}
[r(4)d(.075)]
{\xybox{\xygraph{{\xybox{\xygraph{
!{0;(.5,0):(0,2)::}
{n}="p0" [r(2)] {n}="p1" [r(2)] {n}="p2" [dl] {J}="p3" [l(2)] {I}="p4"
"p0":"p1"^-{\rho''}:"p2"^-{\rho'}:"p3"^-{j}:@{<-}"p4"^-{f}:@{<-}"p0"^-{i} "p1":"p4"^-{i'} "p0" [d(.5)r(.8)] :@{=>}[r(.4)]^-{\beta} "p1" [d(.5)r(.5)] :@{=>}[r(.4)]^-{\alpha'}}}}
[r(1.35)] {=} [r(.95)u(.04)]
{\xybox{\xygraph{!{0;(.5,0):(0,2)::} {n}="p0" [r(2)] {n}="p1" [dl] {J}="p2" "p0":"p1"^-{\rho}:"p2"^-{j}:@{<-}"p0"^-{fi} "p0" [d(.45)r(.75)] :@{=>}[r(.5)]^-{\alpha}}}}}}}} \]
where $\rho \in \Sigma_n$. An arrow $(i,\rho,\alpha) \to (i',\rho',\alpha')$ is a pair $(\beta,\rho'')$ as in the middle, with $\rho'\rho'' = \rho$ in $\Sigma_n$, and moreover satisfying the equation on the right in the previous display. The inclusion of objects of the form $(i,1_n,\alpha)$ can be regarded as a functor
\[ F : \prod_{k=1}^n f \downarrow j_k \longrightarrow \SMCMnd(f) \downarrow j \]
and for $(i,\rho,\alpha)$ as above, the isomorphism $(\id,\rho^{-1}) : (i\rho^{-1},1_n,\alpha\rho^{-1}) \to (i,\rho,\alpha)$ exhibits $F$ as essentially surjective on objects, thus an equivalence, and thus final, so that the result follows from the proof of Proposition \ref{prop:algcocomp-M}(\ref{pcase:monoidally-cocomplete})($\impliedby$).
\end{proof}
\begin{prop}\label{prop:algcocomp-Fp}
Let $\ca V$ be a category with finite products.
\begin{enumerate}
\item $\ca V$ is algebraically cocomplete relative to all functors between small categories as a pseudo $\FPMnd$-algebra iff $\ca V$ is cocomplete and its cartesian product preserves colimits in each variable.\label{pcase:cart-monoidally-cocomplete}
\item $\ca V$ is algebraically cocomplete relative to all functors between small discrete categories as a pseudo $\FPMnd$-algebra iff $\ca V$ has coproducts and its cartesian product preserves coproducts in each variable.\label{pcase:cart-dist-monoidal}
\end{enumerate}
\end{prop}
\begin{proof}
Once again it is only necessary to modify the argument for (\ref{pcase:cart-monoidally-cocomplete})($\impliedby$). Proceeding analogously to the proof of Proposition \ref{prop:algcocomp-Sm-or-Bm} our task is to exhibit a final functor
\[ F : \prod_{k=1}^n f \downarrow j_k \longrightarrow \FPMnd(f) \downarrow j. \]
An object of $\FPMnd(f) \downarrow j$ is a triple $(i,\alpha,\beta)$ as on the left in
\[ \xygraph{{\xybox{\xygraph{{m}="p0" [r] {n}="p1" [d] {J}="p2" [l] {I}="p3" "p0":@{<-}"p1"^-{\alpha}:"p2"^-{j}:@{<-}"p3"^-{f}:@{<-}"p0"^-{i} "p0" [d(.55)r(.35)] :@{=>}[r(.3)]^-{\beta}}}}
[r(3)u(.1)]
{\xybox{\xygraph{!{0;(.5,0):(0,2)::} {m}="p0" [r(2)] {m'}="p1" [dl] {I}="p2" "p0":@{<-}"p1"^-{\alpha''}:"p2"^-{i'}:@{<-}"p0"^-{i} "p0" [d(.45)r(.75)] :@{=>}[r(.5)]^-{\gamma}}}}
[r(4)d(.075)]
{\xybox{\xygraph{{\xybox{\xygraph{
!{0;(.5,0):(0,2)::}
{m}="p0" [r(2)] {m'}="p1" [r(2)] {n}="p2" [dl] {J}="p3" [l(2)] {I}="p4"
"p0":@{<-}"p1"^-{\alpha''}:@{<-}"p2"^-{\alpha'}:"p3"^-{j}:@{<-}"p4"^-{f}:@{<-}"p0"^-{i} "p1":"p4"^-{i'} "p0" [d(.5)r(.8)] :@{=>}[r(.4)]^-{\gamma} "p1" [d(.5)r(.5)] :@{=>}[r(.4)]^-{\beta'}}}}
[r(1.35)] {=} [r(.95)u(.04)]
{\xybox{\xygraph{!{0;(.5,0):(0,2)::} {m}="p0" [r(2)] {n}="p1" [dl] {J}="p2" "p0":@{<-}"p1"^-{\alpha}:"p2"^-{j}:@{<-}"p0"^-{fi} "p0" [d(.45)r(.75)] :@{=>}[r(.5)]^-{\beta}}}}}}}} \]
and an arrow $(i,\alpha,\beta) \to (i',\alpha',\beta')$ is a pair $(\alpha'',\gamma)$ such that $\alpha'\alpha'' = \alpha$ and the equation on the right in the previous display holds. As in Proposition \ref{prop:algcocomp-Sm-or-Bm} we take $F$ to be the full inclusion of objects of the form $(i,1_n,\beta)$, and the assignations $(i,\alpha,\beta) \mapsto (i\alpha,1_n,\beta)$ describe the effect on objects of a left adjoint to $F$. Thus $F$, as a right adjoint, is a indeed final functor.
\end{proof}
Algebraic cocompleteness in the sense of Definition \ref{def:algebraic-cocompleteness} arises also for 2-monads on 2-categories other than $\Cat$. In particular one has
\begin{exam}\label{ex:distmonglob}
A monoidal globular category in the sense of \cite{Batanin-MonGlobCats} is a pseudo algebra for the 2-monad denoted $D_s$ in \cite{Batanin-MonGlobCats}, and a monoidal globular category conforming to Definition 5.3 of \cite{Batanin-MonGlobCats} is in particular, algebraically cocomplete relative to all morphisms of small discrete globular categories.
\end{exam}

\subsection{Existence of algebraic left extensions}
\label{ssec:ans-gen-qn-alg-Kan-ext}
We now give the sufficient conditions on $(f,\overline{f})$ and $(A,a)$ as in Question \ref{qn:alg-kan-ext}, so that every pseudomorphism $(g,\overline{g}) : I \to A$ admits algebraic left extension along $f$. The conditions we shall require on $(A,a)$ are that it be algebraically cocomplete relative to $f$ in the sense of Definition \ref{def:algebraic-cocompleteness}. We now turn to a discussion of the required conditions on $(f,\overline{f})$. These involve a generalisation of Guitart's notion of ``exact square'' to the setting of a 2-category $\ca K$ with comma objects.
\begin{defn}\label{def:ExactSquare}
A lax square as on the left
\[ \xygraph{{\xybox{\xygraph{{P}="p0" [r] {B}="p1" [d] {C}="p2" [l] {A}="p3" "p0":"p1"^-{q}:"p2"^-{g}:@{<-}"p3"^-{f}:@{<-}"p0"^-{p} "p0" [d(.55)r(.35)] :@{=>}[r(.3)]^-{\phi}}}}
[r(4)]
{\xybox{\xygraph{{P}="tl" [r(2)] {B}="tr" [l(2)d] {A}="l" [r(2)] {C}="r" [dl] {V}="b" "l":"r"^-{f}:"b"^-{l}:@{<-}"l"^-{h} [d(.5)r(.85)] :@{=>}[r(.3)]^-{\psi}
"tl"(:"tr"^-{q}:"r"^-{g},:"l"_-{p}) "tl" [d(.5)r(.85)] :@{=>}[r(.3)]^-{\phi}}}}} \]
in a 2-category $\ca K$ with comma objects is \emph{exact} when for all $\psi$ which exhibit $l$ as a pointwise left Kan extension of $h$ along $f$, the composite 2-cell on the right exhibits $lg$ as a pointwise left Kan extension of $hp$ along $q$.
\end{defn}
\begin{exam}\label{exam:exact-commas}
The proof of Proposition 24 \cite{Street-FibrationIn2cats} requires only comma objects in the ambient 2-category, and thus comma squares are exact in general. We shall revisit this in Proposition \ref{prop:commas-are-exact}.
\end{exam}
\begin{defn}\label{def:exact-T-morphism}
Let $(\ca K,T)$ be a 2-monad, suppose that $\ca K$ has comma objects and let $(f,\overline{f}) : (A,a) \to (B,b)$ be a colax morphism of pseudo $T$-algebras. Then $(f,\overline{f})$ is \emph{exact} when the square
\[ \xygraph{!{0;(1.5,0):(0,.6667)::} {TA}="p0" [r] {TB}="p1" [d] {B}="p2" [l] {A}="p3" "p0":"p1"^-{Tf}:"p2"^-{b}:@{<-}"p3"^-{f}:@{<-}"p0"^-{a} "p0" [d(.55)r(.4)] :@{=>}[r(.2)]^-{\overline{f}}} \]
is exact in the sense of Definition \ref{def:ExactSquare}.
\end{defn}
\begin{thm}\label{thm:AlgKan}
Suppose that $(\ca K,T)$ is a 2-monad, $\ca K$ has comma objects, $(f,\overline{f}) : (I,i) \to (J,j)$ is a colax morphism of pseudo $T$-algebras, and $(g,\overline{g}) : (I,i) \to (A,a)$ is a lax morphism of pseudo $T$-algebras.
\begin{enumerate}
\item If $(A,a)$ is algebraically cocomplete relative to $(f,\overline{f})$, then if $\psi$
\[ \xygraph{{I}="p0" [r(2)] {J}="p1" [dl] {A}="p2"
"p0":"p1"^-{f}:"p2"^-{h}:@{<-}"p0"^-{g}
"p0" [d(.5)r(.85)] :@{=>}[r(.3)]^-{\psi}} \]
exhibits $h$ as a pointwise left extension of $g$ along $f$ in $\ca K$, then the unique 2-cell $\overline{h}$ satisfying
\begin{equation}\label{eq:def-hbar} 
\xygraph{!{0;(1.5,0):(0,.6667)::} 
{TI}="q0" [r] {TJ}="q1" [r] {J}="q2" [d] {A}="q3" [l] {TA}="q4"
"q0":"q1"^-{Tf}:"q2"^-{j}:"q3"^-{h}:@{<-}"q4"^-{a}:@/^{1pc}/@{<-}"q0"^-{Tg} "q1":"q4"^-{Th}
"q1" [d(.5)l(.5)] :@{=>}[r(.2)]^-{T\psi}
"q1" [d(.5)r(.5)] :@{=>}[r(.2)]^-{\overline{h}}
"q2" [d(.5)r] {=} [r]
{TI}="r0" [ur] {TJ}="r1" [dr] {J}="r2" [dl] {A}="r3" [l] {TA}="r4" [ur] {I}="r5"
"r0":"r1"^-{Tf}:"r2"^-{j}:@/^{1pc}/"r3"^-{h}:@{<-}"r4"^-{a}:@{<-}"r0"^-{Tg}
"r5" (:@{<-}"r0"_-{i}, :"r2"^-{f}, :"r3"_-{g})
"r0" [d(.5)r(.4)] :@{=>}[r(.2)]^-{\overline{g}}
"r5" [d(.5)r(.35)] :@{=>}[r(.2)]^-{\psi}
"r5" [u(.3)] :@{=>}[r(.1)u(.2)]^-{\overline{f}}}
\end{equation}
endows $h$ with the structure of a lax morphism.
\label{thmcase:AK-lax}
\item If in the situation of (\ref{thmcase:AK-lax}) $(f,\overline{f})$ is exact and $(g,\overline{g})$ is a pseudomorphism, then $(h,\overline{h})$ is a pseudomorphism, and $\overline{h}$ is unique making $\psi$ a 2-cell in $\PsAlgc T$.
\label{thmcase:AK-pseudo}
\item If in the situation of (\ref{thmcase:AK-pseudo}) $(f,\overline{f})$ is a pseudomorphism, then $\overline{h}$ is unique making $\psi$ a 2-cell in $\PsAlg T$ which exhibits $(h,\overline{h})$ as a pointwise left extension of $(g,\overline{g})$ along $(f,\overline{f})$ in $\PsAlg T$.
\label{thmcase:AK-AlgLeftExt}
\end{enumerate}
\end{thm}
\begin{proof}
(\ref{thmcase:AK-lax}): The algebraic cocompleteness of $(A,a)$ ensures that $aT(\psi)$ is a left extension, and so one has does have $\overline{h}$ uniquely determined by (\ref{eq:def-hbar}). Since $\psi$ is a left extension, the unit axiom for $\overline{h}$ is equivalent to the equation
\[ \xygraph{
{I}="p0" [r(2)] {J}="p1" [r(1.5)u] {TJ}="p2" [r(1.5)d] {J}="p3" [d] {A}="p4" [l(3)] {A}="p5" [r(1.5)u] {TA}="p6"
"p0":"p1"^-{f}:"p2"^-{\eta_J}:"p3"^-{j}:"p4"^-{h}:@{<-}"p5"^-{1_A}:@/^{1pc}/@{<-}"p0"^-{g}
"p1":"p5"_-{h} "p6" (:@{<-}"p2"|-{Th}, :"p4"|-{a}, :@{<-}"p5"|-{\eta_A})
"p0" [d(.5)r] :@{=>}[r(.3)]^-{\psi}
"p1" [r(.75)] {=}
"p6" [r(.6)] :@{=>}[r(.3)]^-{\overline{h}}
"p6" [d(.75)l(.1)] :@{=>}[u(.2)r(.2)]^-{a_0}
"p3" [r] {=} [r]
{I}="q0" [r(2)] {J}="q1" [ur] {TJ}="q2" [dr] {J}="q3" [d] {A}="q4" [l(2)] {A}="q5"
"q0":"q1"^-{f}:"q2"^-{\eta_J}:"q3"^-{j}:"q4"^-{h}:@{<-}"q5"^-{1_A}:@/^{1pc}/@{<-}"q0"^-{g}
"q1":"q5"_-{h} "q1":"q3"_-{1_J}
"q0" [d(.5)r] :@{=>}[r(.3)]^-{\psi}
"q2" [d(.8)l(.1)] :@{=>}[u(.2)r(.2)]^-{j_0}
"q1" [d(.6)r] {=}} \]
and this follows from the calculation
\[ \xygraph{!{0;(.6,0):(0,1.2)::}
{\bullet}="p0" [r] {\bullet}="p1" [ur] {\bullet}="p2" [dr] {\bullet}="p3" [d] {\bullet}="p4" [l(2)] {\bullet}="p5" [ur] {\bullet}="p6"
"p0":"p1"^-{f}:"p2"^-{\eta}:"p3"^-{j}:"p4"^-{h}:@{<-}"p5"^-{1}:@/^{1pc}/@{<-}"p0"^-{g}
"p6" (:@{<-}"p2",:"p4",:@{<-}"p5") "p1":"p5"
"p0" [d(.4)r(.5)] {\scriptstyle{\psi}} "p1" [r(.5)] {\scriptstyle{=}} "p6" ([r(.5)] {\scriptstyle{\overline{h}}}, [d(.6)] {\scriptstyle{a_0}})
"p3" [r] {=}="eq1" [r]
{\bullet}="q0" [ur] {\bullet}="q1" [dr] {\bullet}="q2" [r] {\bullet}="q3" [d] {\bullet}="q4" [l(2)] {\bullet}="q5"
"q0" [r] {\bullet}="q6" "q5" [r] {\bullet}="q7"
"q0":"q1"^-{f}:"q2"^-{\eta}:"q3"^-{j}:"q4"^-{h}:@/^{1.5pc}/@{<-}"q5"^-{1}:@/^{.5pc}/@{<-}"q0"^-{g} "q6" (:@{<-}"q0",:"q2",:@/_{.5pc}/"q7") "q7" (:@{<-}"q2",:"q4",:@{<-}"q5")
"q6" ([u(.4)] {\scriptstyle{=}}, [d(.5)l(.15)] {\scriptstyle{=}}, [d(.4)r(.55)] {\scriptstyle{T\psi}}) "q2" [d(.5)r(.5)] {\scriptstyle{\overline{h}}} "q7" [d(.3)] {\scriptstyle{a_0}}
"q3" [r] {=}="eq2" [r]
{\bullet}="r0" [ur] {\bullet}="r1" [r] {\bullet}="r2" [dr] {\bullet}="r3" [dl] {\bullet}="r4" [l(2)] {\bullet}="r5"
"r0" [r] {\bullet}="r6" [r] {\bullet}="r7" "r5" [r] {\bullet}="r8"
"r0":"r1"^-{f}:"r2"^-{\eta_J}:"r3"^-{j}:@/^{.5pc}/"r4"^(.4){h}:@/^{1.5pc}/@{<-}"r5"^-{1_A}:@{<-}"r0"^-{g}
"r6" (:@{<-}"r0",:"r2",:"r7",:"r8") "r7" (:"r3",:"r4") "r8" (:@{<-}"r5",:"r4")
"r0" ([d(.5)r(.5)] {\scriptstyle{=}},[u(.5)r] {\scriptstyle{=}})
"r2" [d(.5)] {\scriptstyle{\overline{f}}}
"r7" ([d(.5)l(.5)] {\scriptstyle{\overline{g}}},[d(.4)r(.4)] {\scriptstyle{\psi}})
"r8" [d(.3)] {\scriptstyle{a_0}}
"r3" [r] {=}="eq3" [r]
{\bullet}="s0" [ur] {\bullet}="s1" [r] {\bullet}="s2" [dr] {\bullet}="s3" [dl] {\bullet}="s4" [l(2)] {\bullet}="s5"
"s0" [r] {\bullet}="s6" [r] {\bullet}="s7"
"s0":"s1"^-{f}:"s2"^-{\eta_J}:"s3"^-{j}:@/^{.5pc}/"s4"^(.4){h}:@/^{1.5pc}/@{<-}"s5"^-{1_A}:@{<-}"s0"^-{g}
"s6" (:@{<-}"s0",:"s2",:"s7") "s7" (:"s3",:"s4")
"s0":@/_{1.5pc}/"s7"
"s0" [u(.5)r] {\scriptstyle{=}} "s2" [d(.5)] {\scriptstyle{\overline{f}}}
"s7" [d(.4)r(.4)] {\scriptstyle{\psi}} "s5" [r] {\scriptstyle{=}}
"s6" [d(.3)] {\scriptstyle{i_0}}
"s3" [r]{=}="eq4" [r(1.1)]
{\bullet}="t0" [u] {\bullet}="t1" [r(1.2)] {\bullet}="t2" [d] {\bullet}="t3" [l(1.2)d] {\bullet}="t4"
"t0":"t1"^-{f}:"t2"^-{\eta_J}:"t3"^-{j}:@/^{.5pc}/"t4"^(.4){h}:@{<-}"t0"^-{g}
"t0":"t3" "t1":"t3"
"t0" ([d(.4)r(.5)] {\scriptstyle{\psi}},[r(.35)u(.3)] {\scriptstyle{=}})
"t2" [d(.35)l(.25)] {\scriptstyle{j_0}}
"eq1" [u(.25)] {\scriptstyle{\tn{nat.}\,\eta}}
"eq2" [u(.3)] {\scriptstyle{\tn{def.}\,\overline{h}}}
"eq3" [u(.3)] {\scriptstyle{\tn{unit}\,\overline{g}}}
"eq4" [u(.3)] {\scriptstyle{\tn{unit}\,\overline{f}}}} \]
The algebraic cocompleteness of $(A,a)$ ensures that $aT(a)T^2(\psi)$ is a left extension, and so the multiplicative axiom for $\overline{h}$ is equivalent to
\[ \xygraph{!{0;(1.5,0):(0,.6667)::} 
{T^2I}="p0" [ur] {TI}="p1" [r] {TJ}="p2" [dr] {J}="p3" [d] {A}="p4" [l] {TA}="p5" [l] {T^2A}="p6" [u] {T^2J}="p7" [r] {TJ}="p8"
"p0":"p1"^-{\mu}:"p2"^-{Tf}:"p3"^-{j}:"p4"^-{h}:@{<-}"p5"^-{a}:@{<-}"p6"^-{Ta}:@{<-}"p0"^-{T^2g}
"p7" (:@{<-}"p0"_-{T^2f},:"p2"^-{\mu},:"p8"^-{Tj},:"p6"^-{T^2h})
"p8" (:"p3"^-{j},:"p5"^-{Th})
"p1" [d(.5)] {=}
"p0" [d(.5)r(.65)] :@{=>}[r(.2)]^-{T^2\psi}
"p7" [d(.55)r(.5)] :@{=>}[r(.2)]^-{T\overline{h}}
"p8" [d(.55)r(.5)] :@{=>}[r(.2)]^-{\overline{h}}
"p8" [u(.3)l(.05)] :@{=>}[r(.1)u(.2)]^-{j_2}
"p3" [r(.7)] {=} [r(.7)]
{T^2I}="q0" [ur] {TI}="q1" [r] {TJ}="q2" [dr] {J}="q3" [d] {A}="q4" [l] {TA}="q5" [l] {T^2A}="q6" [u] {T^2J}="q7" [r] {TA}="q8"
"q0":"q1"^-{\mu}:"q2"^-{Tf}:"q3"^-{j}:"q4"^-{h}:@{<-}"q5"^-{a}:@{<-}"q6"^-{Ta}:@{<-}"q0"^-{T^2g}
"q7" (:@{<-}"q0"_-{T^2f},:"q2"^-{\mu},:"q6"^(.4){T^2h})
"q8" (:@{<-}"q2"_-{Th},:"q4"^-{a},:@{<-}"q6"^-{\mu})
"q1" [d(.5)] {=} "q7" [r(.5)] {=}
"q0" [d(.5)r(.65)] :@{=>}[r(.2)]^-{T^2\psi}
"q8" [r(.4)] :@{=>}[r(.2)]^-{\overline{h}}
"q5" [u(.3)l(.05)] :@{=>}[r(.1)u(.2)]^-{a_2}} \]
which follows from a similar calculation using the 2-naturality of $\mu$, the definition of $\overline{h}$ and the multiplicative axioms for $\overline{f}$ and $\overline{g}$.

(\ref{thmcase:AK-pseudo}): The exactness of $\overline{f}$ ensures that its composite with $\psi$ exhibits $hj$ as a pointwise left extension of $gi$ along $Tf$. Since $\overline{g}$ is an isomorphism, the right hand side of (\ref{eq:def-hbar}) exhibits $hj$ as a left extension of $aT(g)$ along $Tf$, and so $\overline{h}$ is invertible. Moreover the equation (\ref{eq:def-hbar}), reinterpretted using $\overline{g}^{-1}$ and $\overline{h}^{-1}$ instead of $\overline{g}$ and $\overline{h}$, is exactly the condition that $\psi$ be a 2-cell in $\PsAlgc T$.

(\ref{thmcase:AK-AlgLeftExt}): We must verify that $\psi$ is a pointwise left Kan extension in $\PsAlg T$. Thus given $(r,\overline{r})$, $(s,\overline{s})$ and $\sigma$ in $\PsAlg T$ as in
\begin{equation}\label{eq:psi-ple}
\xygraph{!{0;(2.5,0):(0,1)::} 
{\xybox{\xygraph{{(f \downarrow r,\pi)}="p0" [r(2)] {(K,k)}="p1" [d] {(J,j)}="p2" [dl] {(A,a)}="p3" [ul] {(I,i)}="p4" "p0":"p1"^-{(q,\id)}:"p2"^-{(r,\overline{r})}:"p3"^(.3){(h,\overline{h})}:@{<-}"p4"^(.7){(g,\overline{g})}:@{<-}"p0"^-{(p,\id)} "p4":"p2"^-{(f,\overline{f})} "p1":@`{"p1"+(2,-0.5),"p3"+(3,0.5)}"p3"^-{(s,\overline{s})}
"p0" [d(.5)r(.85)] :@{=>}[r(.3)]^-{\lambda} "p4" [d(.5)r(.85)] :@{=>}[r(.3)]^-{\psi} "p2" [r(.6)] :@{=>}[r(.3)]^-{\tau}}}}
[r(1.3)] {=} [r]
{\xybox{\xygraph{{(f \downarrow r,\pi)}="p0" [r(2)] {(K,k)}="p1" [d] {}="p2" [dl] {(A,a)}="p3" [ul] {(I,i)}="p4" "p0":"p1"^-{(q,\id)}:@{}"p2"^-{}:@{}"p3"^{}:@{<-}"p4"^(.7){(g,\overline{g})}:@{<-}"p0"^-{(p,\id)} "p4":@{}"p2"^-{} "p1":"p3"^-{(s,\overline{s})}
"p4" [r(.7)] :@{=>}[r(.3)]^-{\sigma}}}}}
\end{equation}
in which $\lambda$ is the comma object, we must exhibit a unique 2-cell $\tau$ in $\PsAlg T$ satisfying (\ref{eq:psi-ple}). Recall that comma objects in $\PsAlg T$ are computed as in $\ca K$, and that the projections may be taken to be strict. Above we have denoted by $\pi : T(f \downarrow r) \to f \downarrow r$ the 1-dimensional part of $(f \downarrow r)$'s pseudo $T$-algebra structure. Forgetting the pseudo algebra and pseudomorphism structures, one does have a unique 2-cell $\tau$ in $\ca K$ satisfying (\ref{eq:psi-ple}), so our task is to show that this 2-cell $\tau$ is an algebra 2-cell, which is to say that $\overline{s}(\tau k) = T(\tau)(\overline{h}T(r))(h\overline{r})$.

Since $(h\lambda)(\psi p)$ is an algebra 2-cell, we have
\begin{equation}\label{eq:leftext-for-psi-alg2cell-proof} 
\xygraph{!{0;(2.5,0):(0,1)::} 
{\xybox{\xygraph{{T(f \downarrow r)}="p0" [r(2)] {TK}="p1" [d] {K}="p2" [d] {J}="p3" [dl] {A}="p4" [ul] {I}="p5" [u] {f \downarrow r}="p6" "p0":"p1"^-{Tq}:"p2"^-{k}:"p3"^-{r}:"p4"^-{h}:@{<-}"p5"^-{g}:@{<-}"p6"^-{p}:@{<-}"p0"^-{\pi} "p6":"p2"^-{q} "p5":"p3"^-{f}
"p0" [d(.5)r] {=} "p6" [d(.5)r(.85)] :@{=>}[r(.3)]^-{\lambda} "p5" [d(.5)r(.85)] :@{=>}[r(.3)]^-{\psi}}}}
[r] {=} [r(1.5)]
{\xybox{\xygraph{{T(f \downarrow r)}="p0" [r(2)] {TK}="p1" [r(1.5)] {K}="p2" [d(2)l(.5)] {J}="p3" [l(2)d] {A}="p4" [l(2)u] {I}="p5" [u(2)l(.5)] {f \downarrow r}="p6" [r(1.5)d] {TI}="p7" [r(2)] {TJ}="p8" [dl] {TA}="p9"
"p0":"p1"^-{Tq}:"p2"^-{k}:"p3"^-{r}:"p4"^-{h}:@{<-}"p5"^-{g}:@{<-}"p6"^-{p}:@{<-}"p0"^-{\pi} "p7" (:@{<-}"p0"^-{Tp},:"p8"^-{Tf},:"p9"_-{Tg},:"p5"^-{i}) "p8" (:@{<-}"p1"_-{Tr},:"p3"^-{j},:"p9"^-{Th}:"p4"^-{a})
"p0" [d(.5)r(.85)] :@{=>}[r(.3)]^-{T\lambda}
"p7" [d(.5)r(.85)] :@{=>}[r(.3)]^-{T\psi}
"p9" [l(1.15)d(.1)] :@{=>}[r(.3)]^-{\overline{g}}
"p9" [r(.85)d(.1)] :@{=>}[r(.3)]^-{\overline{h}^{-1}}
"p8" [r(.6)] :@{=>}[r(.3)]^-{\overline{r}^{-1}}
"p7" [l(.75)] {=}}}}}
\end{equation}
and by the algebraic cocompleteness of $A$, the invertibility of $\overline{g}$, $\overline{h}^{-1}$ and $\overline{r}^{-1}$, and Example \ref{exam:exact-commas}, these composites exhibit $hrk$ as a left Kan extension of $gp\pi$ along $Tq$. Thus it suffices to verify the algebra 2-cell axiom only after precomposition with the composite on the left hand side of (\ref{eq:leftext-for-psi-alg2cell-proof}). This is done in the calculation
\[ \xygraph{!{0;(1.75,0):(0,1)::}
{\xybox{\xygraph{!{0;(1,0):(0,.7)::} 
{\bullet}="p0" [ur] {\bullet}="p1" [dr] {\bullet}="p2" [d(2)] {\bullet}="p3" [l] {\bullet}="p4" [l] {\bullet}="p5"
"p0":@{<-}"p1"^-{\pi}:"p2"^-{Tq}:"p3"^-{Ts}:"p4"^-{a}:@{<-}"p5"^-{g}:@{<-}"p0"^-{p}
"p0" [r] {\bullet}="q0" "p4" [l(.5)u] {\bullet}="q1"
"q0" (:@{<-}"p0",:@{<-}"p2",:@/^{.75pc}/"p4",:"q1")
"q1" (:"p4",:@{<-}"p5")
"p0" [d(.7)r(.25)] :@{=>}[r(.2)u(.2)]^-{\lambda}
"q1" [d(.8)l(.1)] :@{=>}[r(.2)u(.2)] [u(.15)l(.15)] {\scriptstyle{\psi}}
"q1" [r(.3)] :@{=>}[r(.25)]^-{\tau}
"q1" [r] :@{=>}[r(.25)]^-{\overline{s}}
"q0" [u(.4)] {\scriptstyle{=}}}}}
[r(.9)] {=} [r(.7)]
{\xybox{\xygraph{!{0;(.7,0):(0,1)::} 
{\bullet}="p0" [ur] {\bullet}="p1" [dr] {\bullet}="p2" [d(2)] {\bullet}="p3" [l] {\bullet}="p4" [l] {\bullet}="p5"
"p0":@{<-}"p1"^-{\pi}:"p2"^-{Tq}:"p3"^-{Ts}:"p4"^-{a}:@{<-}"p5"^-{g}:@{<-}"p0"^-{p}
"p0" [r] {\bullet}="q0"
"q0" (:@{<-}"p0",:@{<-}"p2",:"p4")
"p0" [r(.375)d] :@{=>}[r(.25)]^-{\sigma}
"q0" [r(.375)d] :@{=>}[r(.25)]^-{\overline{s}}
"q0" [u(.4)] {\scriptstyle{=}}}}}
[r(.7)] {=} [r(.8)]
{\xybox{\xygraph{!{0;(.7,0):(0,1)::} 
{\bullet}="p0" [ur] {\bullet}="p1" [dr] {\bullet}="p2" [d(2)] {\bullet}="p3" [l] {\bullet}="p4" [l] {\bullet}="p5"
"p0":@{<-}"p1"^-{\pi}:"p2"^-{Tq}:"p3"^-{Ts}:"p4"^-{a}:@{<-}"p5"^-{g}:@{<-}"p0"^-{p}
"p1" [d(1.5)] {\bullet}="q0"
"q0" (:@{<-}"p1",:"p3",:"p5")
"q0" [d(.9)l(.125)] :@{=>}[r(.25)]^-{\overline{g}}
"q0" [r(.4)] :@{=>}[r(.25)]^-{T\sigma}}}}
[r(.7)] {=} [r(.9)]
{\xybox{\xygraph{!{0;(1,0):(0,.7)::} 
{\bullet}="p0" [ur] {\bullet}="p1" [dr] {\bullet}="p2" [d(2)] {\bullet}="p3" [l] {\bullet}="p4" [l] {\bullet}="p5"
"p0":@{<-}"p1"^-{\pi}:"p2"^-{Tq}:"p3"^-{Ts}:"p4"^-{a}:@{<-}"p5"^-{g}:@{<-}"p0"^-{p}
"p0" [d(.75)r(.6)] {\bullet}="q0" "p0" [d(.75)r(1.4)] {\bullet}="q1"
"q0" (:@{<-}"p1",:"q1",:@/_{.5pc}/"p3",:"p5")
"q1" (:@{<-}"p2",:@/_{.2pc}/"p3")
"q0" [d(.8)l(.125)] :@{=>}[r(.25)]^-{\overline{g}}
"q0" [d(.8)r(.8)] :@{=>}[r(.1)u(.25)]^-{T\psi}
"q0" [d(.5)r(1.2)] :@{=>}[r(.1)u(.25)] [u(.1)l(.2)] {\scriptstyle{T\tau}}
"q1" [u(.7)l(.25)] :@{=>}[r(.25)]^-{T\lambda}}}}
[r(.75)] {=} [r(.9)]
{\xybox{\xygraph{!{0;(1,0):(0,.7)::} 
{\bullet}="p0" [ur] {\bullet}="p1" [dr] {\bullet}="p2" [d(2)] {\bullet}="p3" [l] {\bullet}="p4" [l] {\bullet}="p5"
"p0":@{<-}"p1"^-{\pi}:"p2"^-{Tq}:"p3"^-{Ts}:"p4"^-{a}:@{<-}"p5"^-{g}:@{<-}"p0"^-{p}
"p0" [r] {\bullet}="q0" "p5" [r(.6)u] {\bullet}="q1" "p3" [l(.6)u] {\bullet}="q2"
"q0" (:@{<-}"p0",:@{<-}"p2",:"q1")
"q1" (:@{<-}"q2",:"p4",:@{<-}"p5")
"q2" (:@{<-}"p2",:"p3")
"p0" [d(.6)r(.25)] :@{=>}[r(.25)]^-{\lambda}
"q0" [d(.55)r(.1)] :@{=>}[r(.25)]^-{\overline{r}}
"q1" [d(.55)r(.5)] :@{=>}[r(.25)]^-{\overline{h}}
"q2" [r(.25)d(.1)] :@{=>}[r(.25)]^-{T\tau}
"q1" [d(.8)l(.1)] :@{=>}[r(.2)u(.2)] [u(.15)l(.15)] {\scriptstyle{\psi}}
"q0" [u(.4)] {\scriptstyle{=}}}}}} \]
which uses the definition of $\tau$ and the algebra 2-cell axioms for $\sigma$, $\lambda$ and $\psi$.
\end{proof}
The condition on $(A,a)$ in Theorem \ref{thm:AlgKan} of being algebraically cocomplete relative to $(f,\overline{f})$ ensures in particular that the pointwise left extension $\psi$ exists, and so in the language of Definition \ref{def:alg-left-extension}, Theorem \ref{thm:AlgKan}(\ref{thmcase:AK-AlgLeftExt}) says the following.
\begin{cor}\label{cor:AlgKan}
Suppose $(\ca K,T)$ is a 2-monad, $\ca K$ has comma objects, $(f,\overline{f}) : (I,i) \to (J,j)$ is an exact pseudomorphism of pseudo $T$-algebras, and $(A,a)$ is algebraically cocomplete. Then any pseudomorphism $(g,\overline{g}) : (I,i) \to (A,a)$ admits algebraic left extension along $(f,\overline{f})$.
\end{cor}

\section{The Main Theorem and applications}
\label{sec:formulating-main-result}

In the situations in which we wish to apply Corollary \ref{cor:AlgKan}, $(f,\overline{f})$ itself comes from a particular monad theoretic context. This context comes from the theory of internal algebras as described in \cite{Weber-CodescCrIntCat}. One has three 2-monads $(\ca M,R)$, $(\ca L,S)$ and $(\ca K,T)$ participating in this context, with
\begin{enumerate}
\item $T$ describing the type of ambient structure,
\item $S$ describing the one type of structure that can be considered as internal to any pseudo $T$-algebra $A$,
\item $R$ describing another type of structure that can be considered as internal to any pseudo $T$-algebra $A$, and
\item one has forgetful functors $U^G_A : \Alg S(A) \to \Alg R(A)$ definable from the context, where $\Alg S(A)$ (resp. $\Alg R(A)$) is the category of $S$-algebras (resp. $T$-algebras) internal to $A$.
\end{enumerate}
This context is established in Section \ref{ssec:internal-algebras} and in particular, $U^G_A$ is given in Definition \ref{defn:forgetful-between-cats-of-int-algebras}. The point of this article is to understand the systematic computation of the left adjoint to $U^G_A$. In Section \ref{ssec:alg-left-ext-int-alg-class} we explain that under the right conditions this is obtained via a process of algebraic left extension. Then in Section \ref{ssec:MainTheorem} we describe the main theorem of this article, which roughly speaking says that for contexts which arise from polynomial monads in $\Cat$, one always has the right conditions. From \cite{Weber-CodescCrIntCat, Weber-OpPoly2Mnd} this result covers many situations arising from operads, and we give examples in Section \ref{ssec:operadic-examples}.

\subsection{Forgetful functors between categories of internal algebras}
\label{ssec:internal-algebras}
Recall from \cite{Weber-CodescCrIntCat} that given 2-monads $(\ca L,S)$ and $(\ca K,T)$, an adjunction $F : (\ca L,S) \to (\ca K,T)$ between them consists of (1) a 2-functor $F_! : \ca L \to \ca K$, (2) a 2-natural transformation $F^c : F_!S \to TF_!$ providing the coherence of a colax monad morphism, and (3) a right adjoint $F^* : \ca K \to \ca L$ of $F_!$. We denote by $F^l : SF^* \to F^*T$ the mate of $F^c$, which endows $F^*$ with the structure of a lax monad morphism. By virtue of this structure the 2-functor $F^*$ lifts to any of the 2-categories of algebras of $T$ and $S$ compatibly with the inclusions amongst them, we denote by $\overline{F}$ such liftings.
\begin{defn}\label{defn:internal-algebra}
(\cite{Weber-CodescCrIntCat} Definition 3.1.4)
Let $F : (\ca L,S) \to (\ca K,T)$ be an adjunction of 2-monads, suppose that $\ca L$ has a terminal object $1$, and let $A$ be a pseudo $T$-algebra. An \emph{$S$-algebra internal to $A$} (relative to $F$) is a lax morphism $1 \to \overline{F}A$ of $S$-algebras. The category of $S$-algebras internal to $A$ is defined to be $\PsAlgl S(1,\overline{F}A)$ and is denoted as $\Alg S(A)$.
\end{defn}
One may regard an operad $T$ with set of colours $I$ as a 2-monad on $\Cat/I$ following \cite{Weber-OpPoly2Mnd}, and as explained in Example 3.2.1 of \cite{Weber-CodescCrIntCat}
\begin{enumerate}
\item One has an adjunction of 2-monads $(\Cat/I,T) \to (\Cat,\SMCMnd)$.
\item A $T$-algebra in a pseudo $\SMCMnd$-algebra $\ca V$ in the sense of Definition \ref{defn:internal-algebra}, is an algebra for the operad $T$ in (the symmetric monoidal category) $\ca V$ in the usual sense.
\end{enumerate}
Given a morphism $F : S \to T$ of operads, one has a forgetful functor
\begin{equation}\label{eq:forgetful-operadic-context}
\Alg T(\ca V) \longrightarrow \Alg S(\ca V)
\end{equation}
and we now give the general monad-theoretic context giving rise to such forgetful functors.
\begin{defn}\label{def:monad-adjunction-over-base}
Let $H : (\ca M,R) \to (\ca K,T)$ and $F : (\ca L,S) \to (\ca K,T)$ be adjunctions of 2-monads. Then an \emph{adjunction of 2-monads over $T$} is an adjunction $G : (\ca M,R) \to (\ca L,S)$ of 2-monads such that $(F_!,F^c)(G_!,G^c) = (H_!,H^c)$, as colax morphisms of 2-monads.
\end{defn}
\begin{rem}\label{rem:monad-adjunction-over-base}
The condition that $(F_!,F^c)(G_!,G^c) = (H_!,H^c)$ as colax morphisms of 2-monads says that $F_!G_! = H_!$ at the level of 2-functors, and that $H^c$ is the composite
\[ \xygraph{{\xybox{\xygraph{!{0;(1.5,0):(0,1)::} {F_!G_!R}="p0" [r] {F_!SG_!}="p1" [r] {TF_!G_!}="p2"
"p0":"p1"^-{F_!G^c}:"p2"^-{F^cG_!}}}}
[r(5)]
{\xybox{\xygraph{{\Algs T}="tl" [r(2)] {\Algs S}="tr" [dl] {\Algs R}="tb" "tl":"tr"^-{\overline{F}}:"tb"^-{\overline{G}}:@{<-}"tl"^-{\overline{H}}
"tl" [d(2)l(.3)] {\ca K}="l" [r(2)] {\ca L}="r" [dl] {\ca M}="b" "l":"r"^(.45){F^*}|(.575)*=<3pt>{}:"b"^-{G^*}:@{<-}"l"^-{H^*}
"tl":"l"_{U^T} "tr":"r"^{U^S} "tb":"b"^(.3){U^R}
"tl" [d(.5)r] {\iso} [u(.2)] {\scriptstyle{\overline{\gamma}}}
"l" [d(.4)r(.9)] {\iso} [u(.2)] {\scriptstyle{\gamma}}}}}} \]
on the left. Taking right adjoints of $F_!G_! = H_!$ gives an isomorphism $\gamma : G^*F^* \iso H^*$ compatible with the lax monad morphism coherences, that is, this isomorphism is a 2-cell in $\MND(\TwoCAT)$ in the sense of \cite{Street-FTM}. From the formal theory of monads, this compatibility gives a lifting of $\gamma$ to $\overline{\gamma}$ making the prism on the right in the previous display commute. From the explicit description of $\overline{F}$, $\overline{G}$, $\overline{H}$ and $\overline{\gamma}$, see Remark 3.1.3 of \cite{Weber-CodescCrIntCat} for an indication, one may easily verify directly that $\gamma$ lifts to any of the other types of 2-categories of algebras of $R$, $S$ and $T$.
\end{rem}
Since $F^*$, $G^*$ and $H^*$ are right adjoints they preserve the terminal object $1$, and moreover since $U^R$, $U^S$ and $U^T$ are monadic they create all limits, and so $\overline{F}$, $\overline{G}$ and $\overline{H}$ also preserve $1$. For the sake of convenience, we assume terminal objects are chosen so that they are preserved strictly by these 2-functors, so for example, $\overline{G}(1) = 1$.
\begin{defn}\label{defn:forgetful-between-cats-of-int-algebras}
In the context of Definition \ref{def:monad-adjunction-over-base} in which $\ca L$ and $\ca M$ have terminal objects, and given $A \in \PsAlg T$, one has a functor
\[ \begin{array}{lccr} {U^G_A : \Alg S(A) \longrightarrow \Alg R(A)} &&& {a \mapsto \overline{\gamma}_A\overline{G}(a)} \end{array} \]
with object map as indicated.
\end{defn}
\begin{exams}\label{exams:forgetfuls-from-operad-morphisms}
As in Notation 3.3.1 of \cite{Weber-CodescCrIntCat} we denote by
\[ F : (\Cat/I,S) \longrightarrow (\Cat/J,T) \]
the adjunction of 2-monads arising, as in Examples 3.2.2 of \cite{Weber-CodescCrIntCat}, from a morphism of operads $F : S \to T$ with underlying object map $f : I \to J$. Since the process which regards operads as polynomial 2-functors is functorial, indeed it is the functor $\overline{\ca N}$ of Proposition 3.2 of \cite{Weber-OpPoly2Mnd}, the above adjunction is over $\SMCMnd$. Thus we are in the context of Definition \ref{defn:forgetful-between-cats-of-int-algebras}, and for a given pseudo $\SMCMnd$-algebra $\ca V$, the forgetful functor $U^F_{\ca V}$ is exactly (\ref{eq:forgetful-operadic-context}).
\end{exams}

\subsection{Algebraic left extension between internal algebra classifiers}
\label{ssec:alg-left-ext-int-alg-class}
In the general situation of Definition \ref{defn:forgetful-between-cats-of-int-algebras}, it is of interest in examples to understand how to compute the left adjoint to $U^G_A$. From \cite{Weber-CodescCrIntCat} we know that under some conditions on an adjunction $F : (\ca L,S) \to (\ca K,T)$ of 2-monads, one has a universal strict $T$-algebra $T^S$ containing an internal $S$-algebra, called the \emph{internal algebra classifier} with respect to $F$.

In this section we will see that if this is so for all the adjunctions of 2-monads participating in a given instance of Definition \ref{defn:forgetful-between-cats-of-int-algebras}, then one has a strict morphism $T^G : T^R \to T^S$ of $T$-algebras between the corresponding internal algebra classifiers, and $U^G_A$ can then be regarded as the process of precomposing with $T^G$. Thus under the right conditions, the left adjoint to $U^G_A$ will be computed via algebraic left extension along $T^G$.

Let us recall some of the theory of internal algebra classifiers from \cite{Weber-CodescCrIntCat}. Given an adjunction of 2-monads $F$ as above, one has the liftings $\overline{F}$ of $F^*$ to the other 2-categories of algebras, as recalled in Section \ref{ssec:internal-algebras}, and from these one may exhibit a canonical 2-functor
\[ J_F : \Algs T \longrightarrow \Algl S. \]
When $F = 1_{(\ca K,T)}$, this is just the inclusion $J_T : \Algs T \hookrightarrow \Algl T$. See \cite{Weber-CodescCrIntCat} Remark 4.1.1 for more detail. The left adjoint to $J_F$ when it exists is denoted $(-)^{\dagger}_F$, and when $\ca L$ and thus $\Algl S$ has a terminal object $1$, the strict $T$-algebra $1^{\dagger}_F$ is $(T^S,a^S)$, the internal $S$-algebra classifier of \cite{Weber-CodescCrIntCat} Definition 4.1.2.

The circumstances under which $T^S$ exists and is well-behaved are codified in
\begin{defn}\label{def:internalisable-monad-adjunction}
An adjunction $F : (\ca L,S) \to (\ca K,T)$ of 2-monads is \emph{internalisable} when
\begin{enumerate}
\item $\ca K$ and $\ca L$ have all limits and colimits,
\item $\ca K$ is of the form $\tn{Cat}(\ca E)$ for some category $\ca E$ with pullbacks,
\item $S$ and $T$ have rank{\footnotemark{\footnotetext{This means that the underlying endofunctors of $S$ and $T$ preserve $\lambda$-filtered colimits for some regular cardinal $\lambda$.}}}, and
\item $T$ preserves internal functors whose object maps are invertible.
\end{enumerate}
\end{defn}
In other words $F$ is internalisable precisely when it satisfies all the hypotheses of Proposition 4.1.4 of \cite{Weber-CodescCrIntCat}, a mild variant of which we recall now.
\begin{prop}\label{prop:good-int-alg}
\cite{Weber-CodescCrIntCat}
If $F:(\ca L,S) \to (\ca K, T)$ is an internalisable adjunction of 2-monads, then $(-)^{\dagger}_F$ and hence $T^S$ exist, and one has equivalences
\[ \PsAlg{T}(T^S,A) \catequiv \PsAlgl{S}(1,\overline{F}A) \]
pseudonaturally in $A \in \PsAlg{T}$.
\end{prop}
Thus given an internalisable adjunction $F : (\ca L,S) \to (\ca K,T)$ of 2-monads, one has a strict $T$-algebra $T^S$ determined up to isomorphism by the 2-natural isomorphisms
\[ \begin{array}{lccr} {\varphi^F_A : \Algs T(T^S,A) \iso \Algl S(1,\overline{F}A)} &&&
{\varphi'^F_A : \PsAlg T(T^S,A) \catequiv \PsAlgl S(1,\overline{F}A)} \end{array} \]
as on the left, which moreover enjoys a bicategorical universal property determining it up to equivalence amongst all pseudo $T$-algebras, as encoded by the pseudo natural equivalences on the right. In particular by Definition \ref{defn:internal-algebra} one has
\[ \PsAlg T(T^S,A) \catequiv \Alg S(A) \]
for any pseudo $T$-algebra $A$.
\begin{const}\label{const:T^G}
Let $F : (\ca L,S) \to (\ca K,T)$ and $H : (\ca M,R) \to (\ca K,T)$ be adjunctions of 2-monads, and $G : (\ca M,R) \to (\ca L,S)$ be an adjunction of 2-monads over $T$ in the sense of Definition \ref{def:monad-adjunction-over-base}. Suppose that $F$ and $H$ are internalisable in the sense of Definition \ref{def:internalisable-monad-adjunction}. We now construct the strict $T$-algebra morphism
\[ T^G : T^R \longrightarrow T^S. \]
For a general adjunction $G$ of 2-monads over $T$, the forgetful functor $U^G_A$ is
\[ \overline{\gamma}_A \comp \overline{G}(-) : \PsAlgl S(1,\overline{F}A) \longrightarrow \PsAlgl R(1,\overline{H}A)  \]
by Definition \ref{defn:forgetful-between-cats-of-int-algebras}, and this is 2-natural in $A \in \PsAlg T$. Restricting just to strict $T$-algebras $A$, this restricts to $\overline{\gamma}_A \comp \overline{G}(-) : \Algl S(1,\overline{F}A) \to \Algl R(1,\overline{H}A)$. If $G$ is internalisable then by the Yoneda Lemma, one has a unique strict $T$-algebra morphism $T^G : T^R \to T^S$ such that
\[ \xygraph{!{0;(3,0):(0,.3333)::}
{\Algs T(T^S,A)}="p0" [r] {\Algs T(T^R,A)}="p1" [d] {\Algl R(1,\overline{H}A)}="p2" [l] {\Algl S(1,\overline{F}A)}="p3""p0":"p1"^-{(-) \comp T^G}:"p2"^-{\varphi^H_A}:@{<-}"p3"^-{\overline{\gamma}_A \comp \overline{G}(-)}:@{<-}"p0"^-{\varphi^F_A}} \]
commutes for all $A \in \Algs T$.
\end{const}
\begin{rem}\label{rem:T^G-via-units}
In \cite{Weber-CodescCrIntCat} the component at $X$ of the unit of $(-)^{\dagger}_F \ladj J_F$ was denoted $g^F_X : X \to \overline{F}X^{\dagger}_F$, and one defines $g^S_T := g^F_1$. Since $\varphi^F_{T^S}(1_{T^S}) = g^S_T$ and $g^R_T$ is described similarly, $T^G$ could equally-well be defined as the unique strict $T$-algebra morphism making
\[ \xygraph{!{0;(2,0):(0,.5)::} {1}="p0" [r] {\overline{H}T^R}="p1" [d] {\overline{H}T^S}="p2" [l] {\overline{G}\overline{F}T^S}="p3" "p0":"p1"^-{g^R_T}:"p2"^-{\overline{H}T^G}:@{<-}"p3"^-{\overline{\gamma}_{T^S}}:@{<-}"p0"^-{\overline{G}g^S_T}} \]
commute in $\Algl R$. Using this point of view, it is straight forward to show directly that $T^FT^G = T^H$.
\end{rem}
We now explain why $U^G_A$ of Definition \ref{defn:forgetful-between-cats-of-int-algebras}, can be identified as the process of precomposition with $T^G$ when $F$ and $H$ are internalisable. To formulate this precisely, for a 2-category $\ca X$, we denote by $\tn{Psd}(\ca X,\Cat)$ the 2-category of 2-functors $\ca X \to \Cat$, pseudonatural transformations and modifications. In the context of Construction \ref{const:T^G}, $\PsAlgl S(1,\overline{F}(-))$, $\PsAlgl R (1,\overline{H}(-))$, $\PsAlg T(T^S,-)$ and $\PsAlg T(T^R,-)$ are objects of $\tn{Psd}(\PsAlg T,\Cat)$, and $\varphi'^F$ and $\varphi'^H$ are equivalences.
\begin{prop}\label{prop:precomp-T^G-as-forgetful}
In the context of Construction \ref{const:T^G} one has
\[ \xygraph{!{0;(4,0):(0,.25)::} {\PsAlg T(T^S,-)}="p0" [r] {\PsAlg T(T^R,-)}="p1" [d] {\PsAlgl R(1,\overline{H}(-))}="p2" [l] {\PsAlgl S(1,\overline{F}(-))}="p3" "p0":"p1"^-{(-) \comp T^G}:"p2"^-{\varphi'^H}:@{<-}"p3"^-{\overline{\gamma} \comp \overline{G}(-)}:@{<-}"p0"^-{\varphi'^F}:@{}"p2"|-*{\iso}} \]
in $\tn{Psd}(\PsAlg T,\Cat)$.
\end{prop}
\begin{proof}
By Power's coherence theorem \cite{Lack-Codescent, Power-GeneralCoherenceResult} one has, for each pseudo $T$-algebra $A$, a strict $T$-algebra $A'$ and an equivalence $s_A : A \to A'$ in $\PsAlg T$, this being 2-natural in $A$. For each $A$, fix a choice of adjoint pseudo inverse $t_A : A' \to A$ in $\PsAlg T$. The components of the pseudonatural equivalences $\varphi'^F$ and $\varphi'^H$, were described in the proof of Proposition 4.1.4 of \cite{Weber-CodescCrIntCat}, and are the vertical composites in the diagram
\[ \xygraph{!{0;(5,0):(0,.2)::}
{\PsAlgl S(1,\overline{F}A)}="l0" [r] {\PsAlgl R(1,\overline{H}A)}="r0" "l0":"r0"^-{\overline{\gamma}_A\overline{G}(-)} "l0" [d] 
{\Algl S(1,\overline{F}A')}="l1" [r] {\Algl R(1,\overline{H}A')}="r1" "l1":"r1"^-{\overline{\gamma}_{A'}\overline{G}(-)} "l1" [d] 
{\Algs T(T^S,A')}="l2" [r] {\Algs T(T^R,A')}="r2" "l2":"r2"^-{(-) \comp T^G} "l2" [d] 
{\Alg T(T^S,A')}="l3" [r] {\Alg T(T^R,A')}="r3" "l3":"r3"^-{(-) \comp T^G} "l3" [d] 
{\PsAlg T(T^S,A)}="l4" [r] {\PsAlg T(T^R,A)}="r4" "l4":"r4"_-{(-) \comp T^G}
"l0":"l1"_-{\overline{F}(s_A) \comp (-)}:"l2"_-{\varphi^F_{A'}}:"l3"_-{i_{T^S,A'}}:"l4"_-{t_A \comp (-)}
"r0":"r1"^-{\overline{H}(s_A) \comp (-)}:"r2"^-{\varphi^H_{A'}}:"r3"^-{i_{T^R,A'}}:"r4"^-{t_A \comp (-)}
"l3":@{}"r4"|-*{\iso}} \]
whose unlabelled regions commute on the nose, where $i_{T^S,A'}$ and $i_{T^R,A'}$ are the inclusions. The isomorphism in the bottom square is the mate of the identity
\[ ((-) \comp T^G)(s_A \comp (-)) = (s_A \comp (-))((-) \comp T^G). \]
Thus the bottom isomorphism is 2-natural in $A$, by the functoriality, described in \cite{KellyStreet-ElementsOf2Cats}, of the process of taking mates.
\end{proof}
From Corollary \ref{cor:AlgKan} and Proposition \ref{prop:precomp-T^G-as-forgetful} we obtain the following immediate
\begin{cor}\label{cor:alg-left-ext-is-left-adjoint-to-forgetful}
Let $F : (\ca L,S) \to (\ca K,T)$ and $H : (\ca M,R) \to (\ca K,T)$ be adjunctions of 2-monads, and $G : (\ca M,R) \to (\ca L,S)$ be an adjunction of 2-monads over $T$ in the sense of Definition \ref{def:monad-adjunction-over-base}. Suppose that $F$ and $H$ are internalisable in the sense of Definition \ref{def:internalisable-monad-adjunction}. Let $A$ be a pseudo $T$-algebra. If
\begin{enumerate}
\item $T^G$ is exact, and
\label{cor-hyp:exactness-T^G}
\item $A$ is algebraically cocomplete relative to $U^T(T^G)$,
\label{cor-hyp:alg-cocomp-A}
\end{enumerate}
then the left adjoint of $U^G_A$ is computed by algebraic left extension along $T^G$.
\end{cor}

\subsection{Formulating the main theorem}
\label{ssec:MainTheorem}
We understood the algebraic cocompleteness of $A$ in the key examples where $T$ is $\MCMnd$, $\BMCMnd$, $\SMCMnd$ or $\FPMnd$, as commonly-encountered conditions. If one was faced with a different $T$, then a similar analysis as in the proofs of Propositions \ref{prop:algcocomp-M}, \ref{prop:algcocomp-Sm-or-Bm} and \ref{prop:algcocomp-Fp} would express the corresponding condition on $A$ in explicit terms. Thus it remains to be understood when hypothesis (\ref{cor-hyp:exactness-T^G}) of Corollary \ref{cor:alg-left-ext-is-left-adjoint-to-forgetful} are satisfied. The main theorem of this article identifies general conditions on $G$, which arise from polynomial monads over $\Cat$, which guarantee this hypothesis.

We begin by recalling some of the background on polynomial 2-monads from \cite{Weber-CodescCrIntCat, Weber-OpPoly2Mnd, Weber-PolynomialFunctors}. A \emph{polynomial from $I$ to $J$} in $\Cat$ \cite{Weber-PolynomialFunctors} is a diagram
\[ \xygraph{{I}="p0" [r] {E}="p1" [r] {B}="p2" [r] {J}="p3" "p0":@{<-}"p1"^-{s}:"p2"^-{p}:"p3"^-{t}} \]
in which $p$ is an exponentiable functor. A \emph{polynomial 2-functor} is a 2-functor $T : \Cat/I \to \Cat/J$ such that $T \iso \Sigma_t\Pi_p\Delta_s$ for some polynomial $(s,p,t)$ as above, where $\Sigma_t$ is the process of composition with $t$, $\Delta_s$ is the process of pulling back along $s$, and $\Pi_p$ is right adjoint to pulling back along $p$. Such polynomials form a 2-bicategory $\Polyc {\Cat}$, in which the objects are categories and a morphism $I \to J$ is a polynomial as above. The assignment of a polynomial $(s,p,t)$ to its associated polynomial 2-functor $\Sigma_t\Pi_p\Delta_s$ is the effect on arrows of a homomorphism
\[ \PFun {\Cat} : \Polyc{\Cat} \longrightarrow \TwoCAT. \]
A 2-cell $(f_1,f_2):(s_1,p_1,t_1) \to (s_2,p_2,t_2)$ in $\Polyc {\Cat}$ is a diagram
\[ \xygraph{!{0;(1.5,0):(0,.5)::} {I}="p0" [ur] {E_1}="p1" [r] {B_1}="p2" [dr] {J}="p3" [dl] {B_2}="p4" [l] {E_2}="p5" "p0":@{<-}"p1"^-{s_1}:"p2"^-{p_1}:"p3"^-{t_1}:@{<-}"p4"^-{t_2}:@{<-}"p5"^-{p_2}:"p0"^-{s_2}
"p1":"p5"_-{f_2} "p2":"p4"^-{f_1}
"p1":@{}"p4"|-{\tn{pb}} "p0" [r(.5)] {\scriptstyle{=}} "p3" [l(.5)] {\scriptstyle{=}}} \]
and for $S,T : \Cat/I \to \Cat/J$, a 2-natural transformation $\phi : S \to T$ is \emph{polynomial} when it can be factored as $\phi = \tau\PFun{\Cat}(f_1,f_2)\sigma$ where $f_1$ and $f_2$ are as above, $\sigma : S \iso \PFun{\Cat}(s_1,p_1,t_1)$ witnesses $S$ as a polynomial 2-functor, and $\tau : \PFun{\Cat}(s_2,p_2,t_2) \iso T$ witnesses $T$ as a polynomial 2-functor. In this way one can speak of \emph{polynomial 2-monads}, and \emph{polynomial adjunctions of 2-monads}, as those in the image $\PFun{\Cat}$, modulo isomorphisms witnessing the participating 2-functors as polynomial.

Given a monad $(I,T)$ in $\Polyc{\Cat}$, we denote also by $T$ the 2-monad on $\Cat/I$ obtained by applying $\PFun{\Cat}$. Our basic examples, discussed in more detail in Section 5 of \cite{Weber-PolynomialFunctors}, are
\[ \xygraph{!{0;(5,0):(0,.15)::} 
{\xybox{\xygraph{{\MCMnd \, :} [r(2)]
*!(0,-.07){\xybox{\xygraph{{1}="p0" [r] {\N_*}="p1" [r] {\N}="p2" [r] {1}="p3" "p0":@{<-}"p1"^-{}:"p2"^-{U^{\N}}:"p3"^-{}}}}}}}
[r]
{\xybox{\xygraph{{\BMCMnd \, :} [r(2)]
*!(0,-.07){\xybox{\xygraph{{1}="p0" [r] {\B_*}="p1" [r] {\B}="p2" [r] {1}="p3" "p0":@{<-}"p1"^-{}:"p2"^-{U^{\N}}:"p3"^-{}}}}}}}
[d]
{\xybox{\xygraph{{\FPMnd \, :} [r(2.25)]
*!(0,-.07){\xybox{\xygraph{{1}="p0" [r] {\S_*^{\op}}="p1" [r(1.25)] {\S^{\op}}="p2" [r] {1}="p3" "p0":@{<-}"p1"^-{}:"p2"^-{(U^{\S})^{\op}}:"p3"^-{}}}}}}}
[l]
{\xybox{\xygraph{{\SMCMnd \, :} [r(2)]
*!(0,-.07){\xybox{\xygraph{{1}="p0" [r] {\P_*}="p1" [r] {\P}="p2" [r] {1}="p3" "p0":@{<-}"p1"^-{}:"p2"^-{U^{\P}}:"p3"^-{}}}}}}}} \]
the polynomial 2-monads for monoidal categories, braided monoidal cateories, symmetric monoidal categories, and categories with finite products. For a general polynomial 2-monad $(\Cat/I,T)$, we denote its underlying endo-polynomial as
\[ \xygraph{{I}="p0" [r] {E_T}="p1" [r] {B_T}="p2" [r] {I.}="p3" "p0":@{<-}"p1"^-{s_T}:"p2"^-{p_T}:"p3"^-{t_T}} \]
So in the case where $T$ arises from an operad as in Section 3 of \cite{Weber-OpPoly2Mnd}, $I$ is the set of colours of the operad, and $B_T$ is a groupoid whose objects are the operations of the operad and morphisms are given by the symmetric group actions.

Monads in $\Polyc{\Cat}$ are the objects of a category $\PolyMnd{\Cat}$, in which a morphism $(f,F) : (I,S) \to (J,T)$ is a commutative diagram
\[ \xygraph{!{0;(1.5,0):(0,.6667)::} {I}="p0" [r] {E_S}="p1" [r] {B_S}="p2" [r] {I}="p3" [d] {J}="p4" [l] {B_T}="p5" [l] {E_T}="p6" [l] {J}="p7" "p0":@{<-}"p1"^-{s_S}:"p2"^-{p_S}:"p3"^-{t_S}:"p4"^-{f}:@{<-}"p5"^-{t_T}:@{<-}"p6"^-{p_T}:"p7"^-{s_T}:@{<-}"p0"^-{f} "p1":"p6"_-{F_2} "p2":"p5"^-{F_1} "p1":@{}"p5"|-{\tn{pb}}} \]
compatible with the monad structures on $S$ and $T$. As explained in \cite{Weber-OpPoly2Mnd} such a morphism is an adjunction of monads in $\Polyc{\Cat}$, and the corresponding adjunction of 2-monads obtained by applying $\PFun{\Cat}$ to this is denoted as
\[ F : (\Cat/I,S) \longrightarrow (\Cat/J,T). \]
The adjunction $F_! \ladj F^*$ in this case is $\Sigma_f \ladj \Delta_f$. The following result explains why most of the polynomial adjunctions of 2-monads we encounter in this way are internalisable in the sense of Definition \ref{def:internalisable-monad-adjunction}. Recall, that this means that the associated internal algebra classifier exists, giving it a strict universal property with respect to all strict algebras, and moreover, it also enjoys a bicategorical universal property with respect to all pseudo algebras.
\begin{prop}\label{prop:internalisable-poly-adjunctions}
A polynomial adjunction of 2-monads
\[ F : (\Cat/I,S) \longrightarrow (\Cat/J,T) \]
such that $p_T$ is a discrete fibration or opfibration with finite fibres and $J$ is discrete, is internalisable.
\end{prop}
\begin{proof}
Clearly $\Cat/I$ and $\Cat/J$ have all limits and colimits. Since $p_S$ is obtained by pulling back $p_T$, it enjoys the same properties as $p_T$. By Theorem 4.5.1 of \cite{Weber-PolynomialFunctors}, $S$ and $T$ preserve sifted colimits, and so in particular are finitary. Since $J$ is discrete, $\Cat/J = \tn{Cat}(\Set/J)$. Recall from \cite{Bourke-Thesis} that in any 2-category of the form $\tn{Cat}(\ca E)$ for $\ca E$ a category with pullbacks, codescent morphisms are exactly those internal functors which are bijections on objects. Since codescent objects are examples of sifted colimits, and so $T$ preserves them, it follows that $T$ preserves internal functors which are bijections on objects.
\end{proof}
In particular for any polynomial adjunction of 2-monads
\[ F : (\Cat/I,S) \longrightarrow (\Cat/J,T) \]
in which $T$ is $\MCMnd$, $\BMCMnd$, $\SMCMnd$ or $\FPMnd$, Proposition \ref{prop:internalisable-poly-adjunctions} applies. We can now state our main theorem.
\begin{thm}\label{thm:main}
Let $F : (\Cat/J,S) \to (\Cat/K,T)$ and $H : (\Cat/I,R) \to (\Cat/K,T)$ be adjunctions of 2-monads, and $G : (\Cat/I,R) \to (\Cat/J,S)$ be an adjunction of 2-monads over $T$. Suppose that $F$, $G$ and $H$ are polynomial adjunctions of 2-monads. If $I$, $J$ and $K$ are discrete and $p_T$ is a discrete opfibration with finite fibres, then $T^G$ is exact.
\end{thm}
The proof of this theorem will be obtained in Section \ref{ssec:alg-square-before-codescent}.

\subsection{Examples}
\label{ssec:operadic-examples}
We use the term \emph{operad} to refer to what are commonly known as ``coloured symmetric operads'', and also as ``symmetric multicategories''. As was explained in \cite{Weber-OpPoly2Mnd}, an operad $T$ with set of colours $I$ can be identified as a morphism
\[ \xygraph{!{0;(1.5,0):(0,.6667)::} {I}="p0" [r] {E}="p1" [r] {B}="p2" [r] {I}="p3" [d] {1}="p4" [l] {\P}="p5" [l] *!(0,.025){\xybox{\xygraph{{\P_*}}}}="p6" [l] {1}="p7" "p0":@{<-}"p1"^-{}:"p2"^-{}:"p3"^-{}:"p4"^-{}:@{<-}"p5"^-{}:@{<-}"p6"^-{}:"p7"^-{}:@{<-}"p0"^-{} "p1":"p6"_-{} "p2":"p5"^-{b} "p1":@{}"p5"|-{\tn{pb}}} \]
of polynomial monads in which $b$ is a discrete fibration, and the objects of $B$ are the operations of $T$. By applying $\PFun {\Cat}$ one thus obtains an adjunction of 2-monads $(\Cat/I,T) \to (\Cat,\SMCMnd)$, which by Proposition \ref{prop:internalisable-poly-adjunctions} is internalisable since $p_{\SMCMnd} : \P_* \to \P$ is a discrete fibration with finite fibres.  Moreover, one can recover $\Cat$-operads as such polynomial monad morphisms in which $b$ has the structure of a split fibration, and clubs in the sense of Kelly \cite{Kelly-ClubsDoctrines} are recovered by such polynomial monad morphisms in which $I=1$. Similarly for morphisms of operads and their variants, and so the categories of operads, $\Cat$-operads and clubs are all canonically identifiable subcategories of $\PolyMnd{\Cat}/\SMCMnd$, and thus after applying $\PFun {\Cat}$, as categories of internalisable adjunctions of 2-monads over $(\Cat,\SMCMnd)$.

A symmetric monoidal category $\ca V$ can be regarded as an operad, with colours the objects of $\ca V$, and operations $(A_1,...,A_n) \to B$ given by morphisms $A_1 \tensor ... \tensor A_n \to B$. This process is the effect on objects of forgetful 2-functors
\[ \begin{array}{lccr} {\ca U_{\tn{s}} : \Algs {\tnb{S}} \longrightarrow \tnb{Opd}} &&& {\ca U_{\tn{ps}} : \PsAlg {\tnb{S}} \longrightarrow \tnb{Opd}} \end{array} \]
into the 2-category $\tnb{Opd}$ of operads, from the 2-category $\PsAlg {\SMCMnd}$ (resp. $\Algs {\SMCMnd}$) of symmetric monoidal categories and strong monoidal functors (resp. of symmetric strict monoidal categories and strict monoidal functors). As explained in Section 6.4 of \cite{Weber-CodescCrIntCat}, the process{\footnotemark{\footnotetext{As in \cite{Weber-OpPoly2Mnd} we denote by $T$ also the associated 2-monad on $\Cat/I$. When $T$ is the terminal operad $\Com$, this associated 2-monad (on $\Cat$) is just $\SMCMnd$, and so given these conventions, its is also consistent to write $\Com^T$ instead of $\SMCMnd^T$, as was done in \cite{Weber-CodescCrIntCat}.}}} $T \mapsto \SMCMnd^T$ is the effect on objects of a left 2-adjoint to $\ca U_{\tn{s}}$ and a left biadjoint to $\ca U_{\tn{ps}}$.

We denote by $\ca F$ this left 2-adjoint $\tnb{Opd} \to \Algs {\SMCMnd}$. Given a morphism of operads $G : R \to T$, since ${\SMCMnd}^G : {\SMCMnd}^R \to {\SMCMnd}^T$ and $\ca F(G)$ are both obtained from the universal property of ${\SMCMnd}^R = \ca F(R)$, one has ${\SMCMnd}^G = \ca F(G)$. Denoting $I$ and $J$ for the set colours of $R$ and $T$ respectively, the adjunctions of 2-monads
\[ \begin{array}{lcccr} {G : (\Cat/I,R) \longrightarrow (\Cat/J,T)} && {(\Cat/I,R) \longrightarrow (\Cat,\SMCMnd)} && {(\Cat/J,T) \longrightarrow (\Cat,\SMCMnd)} \end{array} \]
in this situation conform to the hypotheses of Theorem \ref{thm:main}. Thus by applying Proposition \ref{prop:algcocomp-Sm-or-Bm}, Corollary \ref{cor:AlgKan} and Theorem \ref{thm:main}, one obtains the following result.
\begin{cor}\label{cor:operad-exactness}
Let $G : R \to T$ be a morphism of operads.
\begin{enumerate}
\item $\ca F(G) : \ca F(R) \to \ca F(T)$ is an exact symmetric monoidal functor.
\item If $\ca V$ is symmetric monoidal closed and cococomplete, and $H : \ca F(R) \to \ca V$ is a symmetric strong monoidal functor, then the left Kan extension $\ca F(T) \to \ca V$ of $H$ along $\ca F(G)$ is a symmetric strong monoidal functor.
\end{enumerate}
\end{cor}
In Part 3 of \cite{BataninBerger-HtyThyOfAlgOfPolyMnd} various contemporary operadic notions: cyclic operads, modular operads and various notions properad; are exhibited as algebras of polynomial monads defined over $\Set$, in which the middle map of the underlying polynomial has finite fibres. As explained in \cite{KockJ-PolyFunTrees, SzawielZawadowski-TheoriesOfAnalyticMonads} such polynomial monads may be identified with ordinary (coloured symmetric) operads whose symmetric group actions are $\Sigma$-free. Thus via \cite{Weber-OpPoly2Mnd} these operads may also be regarded as polynomial monads now over $\Cat$. Hence for a $\Sigma$-free operad $T$ with set of colours $I$, one has \emph{two} associated polynomial monads. Denoting the associated categorical polynomial monad of \cite{Weber-OpPoly2Mnd} as on the left
\[ \xygraph{!{0;(1.5,0):(0,1)::} 
{I}="p0" [r] {E_T}="p1" [r] {B_T}="p2" [r] {I}="p3"
"p0":@{<-}"p1"^-{s}:"p2"^-{p}:"p3"^-{t}
"p3" [r]
{I}="q0" [r] {\pi_0E_T}="q1" [r] {\pi_0B_T}="q2" [r] {I}="q3"
"q0":@{<-}"q1"^-{\pi_0s}:"q2"^-{\pi_0p}:"q3"^-{\pi_0t}} \]
the corresponding $\Set$-based polynomial used in \cite{BataninBerger-HtyThyOfAlgOfPolyMnd} is given on the right, where $\pi_0 : \Cat \to \Set$ is the process of taking the connected components of a category. Recall, $B_T$ is the groupoid whose objects are the operations of $T$, and morphisms are obtained from the symmetric group actions. The property of $\Sigma$-freeness is equivalent to saying that $B_T$ is equivalent to a discrete category. Thus $B_T$ is a groupoid in which there is at most one morphism between any two objects, and the functor $q_{B_T} : B_T \to \pi_0B_T$ which sends objects of $B_T$ to their connected components, is an equivalence. The ``operations'' of the $\Set$-based polynomial monad on the right, are the operations of $T$ \emph{modulo} the symmetric group actions.

Consistently with \cite{Weber-OpPoly2Mnd}, we denote by $T$ also the 2-monad on $\Cat/I$ obtained by applying $\PFun {\Cat}$ to $(s,p,t)$, and by $T/\Sigma$ the 2-monad on $\Cat/I$ induced similarly by $(\pi_0s,\pi_0p,\pi_0t)$. In fact one can define the 2-monad $T/\Sigma$ even when the operad $T$ is not $\Sigma$-free, but then $T/\Sigma$ will not necessarily be a polynomial monad{\footnotemark{\footnotetext{For instance when the operad $T$ is $\Com$, $T$ is $\SMCMnd$ whereas $T/\Sigma$ is the 2-monad for commutative monoids, which is not cartesian, and thus not polynomial.}}}. A substantial part of \cite{Weber-OpPoly2Mnd} is devoted to understanding how $T$ and $T/\Sigma$ are related. By the general definition of $T/\Sigma$ one always has a morphism $q_T : T \to T/\Sigma$ of 2-monads, and thus an induced 2-functor $\overline{q}_T : \Algs T \to \Algs {T/\Sigma}$. Strict algebras of $T/\Sigma$ are the $\Cat$-valued algebras of the original operad $T$, the strict algebras can be identified as ``weakly-equivariant'' $\Cat$-valued algebras of the original operad $T$, and in these terms $\overline{q}_T$ is the inclusion. When the original operad $T$ is $\Sigma$-free various nice things happen. To begin with, $T/\Sigma$ is a polynomial monad and $q_T$ is a morphism of polynomial monads by Theorem  6.6 of \cite{Weber-OpPoly2Mnd}. Moreover, $\overline{q}_T$ is now part of a Quillen equivalence with respect to the Lack model structures \cite{Lack-HomotopyAspects2Monads} on $\Algs T$ and $\Algs {(T/\Sigma)}$, by Theorem 7.7 of \cite{Weber-OpPoly2Mnd}.

Even in the $\Sigma$-free case, there is an advantage to using $T$ over $T/\Sigma$, despite the fact that the polynomial which generates $T/\Sigma$ is simpler. This advantage is that for $T$, there is as explained above, a polynomial adjunction of 2-monads $(\Cat/I,T) \to (\Cat,\SMCMnd)$, and the notion of internal algebra thus arising coincides with the notion of algebra for the operad $T$ internal to a symmetric monoidal category, as explained at the end of Section 4 of \cite{Weber-OpPoly2Mnd}. This is not so for $T/\Sigma$ as the following example shows.
\begin{exam}\label{exam:NonSym}
There is a $\Sigma$-free operad $\mathsf{NSOp}$ for non-symmetric operads with one colour described as a $\Set$-based polynomial monad in Section 9.2 of \cite{BataninBerger-HtyThyOfAlgOfPolyMnd}. One may identify the set of colours of $\mathsf{NSOp}$ as $\N$. Suppose there is a polynomial monad morphism as on the left
\[ \xygraph{!{0;(1.5,0):(0,.6667)::}
{\N}="p0" [r] {\pi_0E_{\mathsf{NSOp}}}="p1" [r(1.25)] {\pi_0B_{\mathsf{NSOp}}}="p2" [r] {\N}="p3" [d] {1}="p4" [l] {\P}="p5" [l(1.25)] *!(0,.025){\xybox{\xygraph{{\P_*}}}}="p6" [l] {1}="p7"
"p0":@{<-}"p1"^-{}:"p2"^-{}:"p3"^-{}:"p4"^-{}:@{<-}"p5"^-{}:@{<-}"p6"^-{}:"p7"^-{}:@{<-}"p0"^-{} "p1":"p6"_-{} "p2":"p5"^-{} "p1":@{}"p5"|-{\tn{pb}}
"p3" [r(1.5)]
{\N}="q0" [r] {\pi_0E_{\mathsf{NSOp}}}="q1" [r(1.25)] {\pi_0B_{\mathsf{NSOp}}}="q2" [r] {\N}="q3" [d] {1}="q4" [l] {\N}="q5" [l(1.25)] *!(0,.025){\xybox{\xygraph{{\N_*}}}}="q6" [l] {1}="q7"
"q0":@{<-}"q1"^-{}:"q2"^-{}:"q3"^-{}:"q4"^-{}:@{<-}"q5"^-{}:@{<-}"q6"^-{}:"q7"^-{}:@{<-}"q0"^-{} "q1":"q6"_-{} "q2":"q5"^-{} "q1":@{}"q5"|-{\tn{pb}}} \]
giving an adjunction of 2-monads $(\Cat/\N,\mathsf{NSOp}/\Sigma) \to (\Cat,\SMCMnd)$, for which the corresponding notion of internal algebra agrees with that of non-symmetric operads within a symmetric monoidal category. Then since the polynomial of $\mathsf{NSOp}/\Sigma$ is componentwise discrete, it would factor through the polynomial underlying the monoidal category 2-monad $\MCMnd$, as on the right in the previous display. The internal structure arising from the resulting adjunction of 2-monads $(\Cat/\N,\mathsf{NSOp}/\Sigma) \to (\Cat,\MCMnd)$ is thus a type of structure expressable within any monoidal category $\ca V$, which when $\ca V$ happens to be symmetric, coincides with the notion of non-symmetric operad within $\ca V$. In other words one would conclude that the notion of non-symmetric operad makes sense in any monoidal category. However, to express the associative law of substitution for an operad one requires at least a braiding.
\end{exam}
\begin{exam}\label{exam:modular-envelope}
The $\Sigma$-free operads $\mathsf{CycOp}$ and $\mathsf{ModOp}$ for cyclic and modular operads respectively, are described as $\Set$-based polynomial monads in Sections 9.6 and 10.1 of \cite{BataninBerger-HtyThyOfAlgOfPolyMnd}, and one has an evident inclusion $J : \mathsf{CycOp} \hookrightarrow \mathsf{ModOp}$ witnessed at the level of polynomials. Thus $\SMCMnd^{\mathsf{ModOp}}$ has the universal property that for all symmetric monoidal categories $\ca V$, modular operads in $\ca V$ may be regarded as symmetric strong monoidal functors $\SMCMnd^{\mathsf{ModOp}} \to \ca V$. Similarly cyclic operads in $\ca V$ correspond to symmetric strong monoidal functors $\SMCMnd^{\mathsf{CycOp}} \to \ca V$, and composition with $\SMCMnd^J : \SMCMnd^{\mathsf{CycOp}} \to \SMCMnd^{\mathsf{ModOp}}$ gives the forgetful functor
\[ U^J_{\ca V} : \mathsf{ModOp}(\ca V) \longrightarrow \mathsf{CycOp}(\ca V) \]
between categories of modular and cyclic operads in $\ca V$. By Corollary \ref{cor:operad-exactness}, the left adjoint to $U^J_{\ca V}$, which is known as the \emph{modular envelope construction}, is given by left Kan extending along $\SMCMnd^J$, when $\ca V$ is symmetric monoidal closed and cocomplete.
\end{exam}
\begin{exams}\label{exams:monad-from-operad-in-general}
Let $T$ be an operad with set of colours $I$ and take $G : R \to T$ of Corollary \ref{cor:operad-exactness} to be the inclusion of $I$ as a discrete category. Then the 2-monad $R$ in this case is just the identity on $\Cat/I$, and the category of $R$-algebras internal to a symmetric monoidal category $\ca V$ is just the category $\ca V^I$ of $I$-indexed families in $\ca V$. When $\ca V$ is symmetric monoidal closed cocomplete, left Kan extension and restriction along $\ca F(G)$ gives the monad on $\ca V^I$ whose algebras are algebras of the operad $T$ in $\ca V$ by Corollary \ref{cor:operad-exactness}, and since the adjunction $\tn{Lan}_{\ca F(G)} \ladj U^G_{\ca V}$ in this case is monadic. In particular when $\ca V = \Set$, one recovers $T/\Sigma$ (just seen as acting on $\Set/I$).
\end{exams}
\begin{rem}\label{rem:nonsym-and-braided-analogues}
As explained in Section 2.3 of \cite{Weber-CodescCrIntCat}, non-symmetric (coloured) operads may be regarded as adjunctions of 2-monads into $(\Cat,\MCMnd)$. For braided operads one instead works over $(\Cat,\BMCMnd)$. Thus in the same way as with $(\Cat,\SMCMnd)$ above, one obtains non-symmetric and braided analogues of Corollary \ref{cor:operad-exactness}. Moreover for each of these variants, Theorem \ref{ssec:MainTheorem} may be applied to give a version of Corollary \ref{cor:operad-exactness} in which $G : R \to T$ is a morphism of $\Cat$-operads.
\end{rem}
\begin{exam}\label{exam:Dubuc}
The non-symmetric operad for pointed sets includes into that for monoids. By applying the non-symmetric analogue of Corollary \ref{cor:operad-exactness} in this case, one recovers the construction of Dubuc \cite{Dubuc-FreeMonoids} of a free monoid on a pointed object as the process of left Kan extending along $\ca F(G)$. This example is considered in more detail in \cite{MelliesTabareau-TAlgTheoriesKan}.
\end{exam}

\section{Exact squares}
\label{sec:ExactSquares}

In this section we study exact squares in various contexts. In Section \ref{sec:exact-in-Cat} we recall exact squares in $\Cat$ as originally defined by Guitart \cite{Guitart-ExactSquares}. In Section \ref{ssec:commas-are-exact} we show that comma squares are exact in general, and then in Section \ref{ssec:exact-pb} we explain when pullbacks and bipullbacks are exact. A first application of these results we show in Section \ref{ssec:general-exactness-in-colaxidem-case}, that if a 2-monad $(\ca K,T)$ is colax-idempotent, then all colax morphisms of $T$-algebras are exact. By contrast there are many examples of non-exact morhisms of $T$-algebras when $T$ is $\MCMnd$, $\SMCMnd$ or $\BMCMnd$, as we see in Section \ref{ssec:exact-colax-monoidal-functors}, where exactness in these cases is characterised in combinatorial terms. Finally in Section \ref{ssec:exact-nat} we explain why, in the context of Theorem \ref{thm:main}, $\mu^T$ is diexact.

\subsection{Exact squares in $\Cat$}
\label{sec:exact-in-Cat}
In this section we recall what it means for a lax square
\begin{equation}\label{eq:lax-square-in-Cat}
\xygraph{{P}="p0" [r] {B}="p1" [d] {C}="p2" [l] {A}="p3" "p0":"p1"^-{q}:"p2"^-{g}:@{<-}"p3"^-{f}:@{<-}"p0"^-{p} "p0" [d(.55)r(.35)] :@{=>}[r(.3)]^-{\phi}}
\end{equation}
in $\Cat$ to be exact in the sense of Guitart \cite{Guitart-ExactSquares}. This has various formulations recalled in Theorem \ref{thm:exact-Cat-alternative-charns}. We give a few other equivalent conditions in Lemmas \ref{lem:coend-for-exactness-in-Cat}, \ref{lem:combinatorial-exactness-Cat} and \ref{lem:combinatorial-exactness-Cat-fact-version} which bring out the combinatorics inherent in the notion. Lemma \ref{lem:combinatorial-exactness-Cat} will be generalised to the setting of lax squares of 2-categories in Section \ref{sec:exact-via-codescent}.

Recall that a \emph{profunctor} from a small category $A$ to a small category $B$ is a functor $F:A^{\op} \times B \to \Set$, and that the composite of $F:A \to B$ and $G:B \to C$ is given in terms of coends by the formula
\[ (G \comp F)(a,c) = \int^{b \in B} G(b,c) \times F(a,b). \]
In particular from a functor $f:A \to B$, one has profunctors $B(f,1):A \to B$ and $B(1,f):B \to A$ defined on objects by
\[ \begin{array}{lccr} {B(f,1)(a,b) = B(fa,b)} &&& {B(1,f)(b,a) = B(a,fb).} \end{array} \]
One has a bicategory $\Prof$ of small categories and profunctors between them, with horizontal composition given by the above coend formula, and the assignments $f \mapsto B(f,1)$ and $f \mapsto B(1,f)$ define the effect on morphisms of identity on objects locally fully faithful homomorphisms
\[ \begin{array}{lccr} {\Cat^{\co} \longrightarrow \Prof} &&& {\Cat^{\op} \longrightarrow \Prof.} \end{array} \]
For any functor $f$ one has an adjunction $B(f,1) \ladj B(1,f)$. Note in particular that to give a lax square (\ref{eq:lax-square-in-Cat}) is to give a 2-cell
\[ \tilde{\phi} : B(q,1) \comp A(1,p) \longrightarrow C(f,g)  \]
in $\Prof$.
\begin{thm}\label{thm:exact-Cat-alternative-charns}
\cite{Guitart-ExactSquares}
For a lax square (\ref{eq:lax-square-in-Cat}) in $\Cat$ the following are equivalent:
\begin{enumerate}
\item $\tilde{\phi}$ is an isomorphism.\label{thmcase:exact-prof}
\item $\forall a \in A$, the functor $a \downarrow p \to fa \downarrow g$ induced by $\phi$ is initial.\label{thmcase:exact-initial}
\item $\forall b \in B$, the functor $q \downarrow b \to f \downarrow gb$ induced by $\phi$ is final.\label{thmcase:exact-final}
\item For any natural transformation $\psi$ which exhibits $l$ as a pointwise left Kan extension of $h$ along $f$, the composite on the left exhibits $lg$ as a pointwise left Kan extension of $hp$ along $q$.\label{thmcase:exact-left-extension}
\[ \xygraph{{\xybox{\xygraph{{P}="tl" [r(2)] {B}="tr" [l(2)d] {A}="l" [r(2)] {C}="r" [dl] {V}="b" "l":"r"^-{f}:"b"^-{l}:@{<-}"l"^-{h} [d(.5)r(.85)] :@{=>}[r(.3)]^-{\psi}
"tl"(:"tr"^-{q}:"r"^-{g},:"l"_-{p}) "tl" [d(.5)r(.85)] :@{=>}[r(.3)]^-{\phi}}}}
[r(4)]
{\xybox{\xygraph{{P}="tl" [r(2)] {A}="tr" [l(2)d] {B}="l" [r(2)] {C}="r" [dl] {V}="b" "l":"r"^-{g}:"b"^-{r}:@{<-}"l"^-{k} [d(.5)r(.85)] :@{<=}[r(.3)]^-{\rho}
"tl"(:"tr"^-{p}:"r"^-{f},:"l"_-{q}) "tl" [d(.5)r(.85)] :@{<=}[r(.3)]^-{\phi}}}}} \]
\item For any natural transformation $\rho$ which exhibits $r$ as a pointwise right Kan extension of $k$ along $g$, the composite on the right exhibits $rf$ as a pointwise right Kan extension of $kq$ along $p$.\label{thmcase:exact-right-extension}
\end{enumerate}
\end{thm}
\begin{defn}\label{def:Guitart-exactness-in-Cat}
\cite{Guitart-ExactSquares}
A lax square (\ref{eq:lax-square-in-Cat}) in $\Cat$ is \emph{exact} when it satisfies the equivalent conditions of Theorem \ref{thm:exact-Cat-alternative-charns}.
\end{defn}
In \cite{Guitart-ExactSquares} Guitart took formulation (\ref{thmcase:exact-prof}) of Theorem \ref{thm:exact-Cat-alternative-charns} as the definition of exactness, and this clearly generalises directly to the enriched setting, and further still to the setting of proarrow equipments \cite{Wood-Proarrows-I, Wood-Proarrows-II}. This line of generalisation is pursued in \cite{Koudenburg-Thesis, MelliesTabareau-TAlgTheoriesKan}. We will also use this point of view below in Section \ref{ssec:pi0-exactness} when discussing exact squares of 2-categories. However, in Definition \ref{def:ExactSquare} above, we instead generalised formulation (\ref{thmcase:exact-left-extension}) in the obvious way.

Let us unpack formulation (\ref{thmcase:exact-prof}) directly. For $a \in A$, $b \in B$ and $x \in P$, one has functions
\[ \begin{array}{lccr} {B(qx,b) \times A(a,px) \to C(fa,gb)} &&&
{(\beta, \alpha) \mapsto g\beta \comp \phi_x \comp f\alpha} \end{array} \]
which are dinatural in $x$, and the corresponding component of $\tilde{\phi}$ is the induced function
\[ \tilde{\phi}_{a,b} : \int^{x \in P} B(qx,b) \times A(a,px) \longrightarrow C(fa,gb).  \]
To proceed further one must compute the above coend. The comma category $q \downarrow b$, whose objects are pairs $(x,\beta)$ where $x \in P$ and $\beta:qx \to b$, comes with a functor $q \downarrow b \to P$ given on objects by $(x,\beta) \mapsto x$. Similarly the comma category $a \downarrow p$ comes with a functor $a \downarrow p \to P$, and so one can pull these back to produce the category $(q \downarrow b) \times_P (a \downarrow p)$. A typical object of this category is a triple $(x,\beta,\alpha)$ where $x \in P$, $\beta:qx \to b$ and $\alpha:a \to px$, and for each $x \in P$ one has a function
\[ \kappa_x : B(qx,b) \times A(a,px) \longrightarrow \pi_0((q \downarrow b) \times_P (a \downarrow p)) \]
which sends $(\beta,\alpha)$ to the connected component which contains $(x,\beta,\alpha)$.
\begin{lem}\label{lem:coend-for-exactness-in-Cat}
The family of functions $(\kappa_x : x \in P)$ is the universal dinatural family which exhibits
\[ \int^{x \in P} B(qx,b) \times A(a,px) = \pi_0((q \downarrow b) \times_P (a \downarrow p)).  \]
\end{lem}
\noindent We leave the elementary task of exhibiting a direct proof of this result to the reader, which consists of verifying that the $\kappa_x$ are dinatural in $x$, and that $\kappa_x$ satisfies the appropriate universal property. Note however that we do recover this result below from a much more general result in Remark \ref{rem:set-case-consequence}.

Regarding $C(fa,gb)$ as a discrete category, one has a functor
\[ \overline{\phi}_{a,b} : (q \downarrow b) \times_P (a \downarrow p) \longrightarrow C(fa,gb) \]
given on objects by $(x,\beta,\alpha) \mapsto g\beta \comp \phi_x \comp f\alpha$. A morphism $(x_1,\beta_1,\alpha_1) \to (x_2,\beta_2,\alpha_2)$ is by definition a morphism $\gamma:x_1 \to x_2$ of $P$ such that $p(\gamma)\alpha_1 = \alpha_2$ and $\beta_1 = \beta_2q(\gamma)$, and given such a $\gamma$, $g(\beta_1)\phi_{x_1}f(\alpha_1) = g(\beta_2)\phi_{x_2}f(\alpha_2)$ because of the naturality of $\phi$. Since $\tilde{\phi}_{a,b} = \pi_0(\overline{\phi}_{a,b})$ an explicit combinatorial characterisation of what it means for $\phi$ to be exact is given as follows.
\begin{lem}\label{lem:combinatorial-exactness-Cat}
A lax square
\[ \xygraph{{P}="p0" [r] {B}="p1" [d] {C}="p2" [l] {A}="p3" "p0":"p1"^-{q}:"p2"^-{g}:@{<-}"p3"^-{f}:@{<-}"p0"^-{p} "p0" [d(.55)r(.35)] :@{=>}[r(.3)]^-{\phi}} \]
in $\Cat$ is exact iff for all $a \in A$ and $b \in B$, the functor $\pi_0:\Cat \to \Set$ which on objects sends a category to its set of connected components, inverts the functor $\overline{\phi}_{a,b}$ defined above.
\end{lem}
In the context of Lemma \ref{lem:combinatorial-exactness-Cat}, given $a \in A$, $b \in B$ and $\gamma : fa \to gb$ we denote by $\tn{Fact}_{\phi}(a,\gamma,b)$ the following category. Its objects are triples $(\alpha,x,\beta)$ where $x \in P$, $\alpha : a \to px$ and $\beta : qx \to b$, such that $g(\beta)\phi_xf(\alpha) = \gamma$. An arrow $(\alpha_1,x_1,\beta_1) \to (\alpha_2,x_2,\beta_2)$ is an arrow $\delta : x_1 \to x_2$ such that $p(\delta)\alpha_1 = \alpha_2$ and $\beta_1 = \beta_2q(\delta)$. Identities and compositions inherited from $P$.

Since $\pi_0(\overline{\phi}_{a,b})$ is bijective iff the fibres of $\overline{\phi}_{a,b}$ are connected, and clearly $\overline{\phi}_{a,b}^{-1}\{\gamma\} = \tn{Fact}_{\phi}(a,\gamma,b)$, Lemma \ref{lem:combinatorial-exactness-Cat} reformulates to the following often-convenient combinatorial characterisation of exactness, from which the proof of Theorem \ref{thm:exact-Cat-alternative-charns} is easily obtained.
\begin{lem}\label{lem:combinatorial-exactness-Cat-fact-version}
A lax square
\[ \xygraph{{P}="p0" [r] {B}="p1" [d] {C}="p2" [l] {A}="p3" "p0":"p1"^-{q}:"p2"^-{g}:@{<-}"p3"^-{f}:@{<-}"p0"^-{p} "p0" [d(.55)r(.35)] :@{=>}[r(.3)]^-{\phi}} \]
in $\Cat$ is exact iff for all $a \in A$, $b \in B$, and $\gamma : fa \to gb$, the category $\tn{Fact}_{\phi}(a,\gamma,b)$ defined above is connected.
\end{lem}

\subsection{Comma squares}
\label{ssec:commas-are-exact}
In this section we recall that in our setting, comma squares are exact. This result, given as Proposition \ref{prop:commas-are-exact} below, is really due to Ross Street, and as mentioned already, appears as Proposition 24 of \cite{Street-FibrationIn2cats}. We go over this result again carefully here because doing so leads to other general results not in the literature in Section \ref{ssec:exact-pb}.

Recall that the left Kan extension $\psi$, which exhibits $l$ as a (pointwise) left extension of $h$ along $f$
\[ \xygraph{{A}="p0" [r(2)] {C}="p1" [dl] {V}="p2" "p0":"p1"^-{f}:"p2"^-{l}:@{<-}"p0"^-{h} "p0" [d(.5)r(.85)] :@{=>}[r(.3)]^-{\psi}} \]
is \emph{preserved} by $k : V \to W$ when $k\psi$ exhibits $kl$ as a (pointwise) left Kan extension of $kh$ along $f$, and $\psi$ is said to be an \emph{absolute} (pointwise) left Kan extension when it is preserved by all arrows out of $V$. It is true in general that left adjoints preserve (pointwise) left Kan extensions, and that the unit $\eta:1_A \to uf$ of an adjunction $f \ladj u : C \to A$ exhibits $u$ as an absolute left Kan extension of $1_A$ along $f$. In fact by Proposition 20 of \cite{Street-FibrationIn2cats} they are also absolute pointwise left Kan extensions.

Other basic well known facts we shall routinely use below concern the ``composability'' and ``cancellability'' of left Kan extensions \cite{StreetWalters-YonedaStructures}. That is, suppose that one has
\[ \xygraph{!{0;(2,0):(0,.5)::} {I}="p0" [r] {J}="p1" [r] {K}="p2" [dl] {A}="p3" "p0":"p1"^-{f}:"p2"^-{g}:"p3"^-{k}:@{<-}"p0"^-{i} "p1":"p3"^-{j}
"p0" [d(.5)r(.675)] :@{=>}[r(.15)]^-{\psi_1} "p1" [d(.5)r(.175)] :@{=>}[r(.15)]^-{\psi_2}} \]
in which $\psi_1$ exhibits $j$ as a left Kan extension of $i$ along $f$. Then the \emph{composability} of left Kan extensions says that if $\psi_2$ exhibits $k$ as a left Kan extension of $j$ along $k$, then the composite exhibits $k$ as a left Kan extension of $i$ along $gf$, and the \emph{cancellability} of left Kan extensions says the converse.
\begin{lem}\label{lem:vert-comp-comma}
In a 2-category with comma objects, the vertical composite of comma squares as on the left, can be factored as
\[ \xygraph{{\xybox{\xygraph{!{0;(1.5,0):(0,.6667)::} {q \downarrow h}="ttl" [r] {D}="ttr" [dl] {f \downarrow g}="tl" [r] {B}="tr" [d] {C}="br" [l] {A}="bl"
"tl":"tr"^-{q}:"br"^-{g}:@{<-}"bl"^-{f}:@{<-}"tl"^-{p} [d(.5)r(.35)] :@{=>}[r(.3)]^-{\lambda}
"ttl":"ttr"^-{q_2}:"tr"^-{h} "tl":@{<-}"ttl"^-{p_2} [d(.5)r(.35)] :@{=>}[r(.3)]^-{\lambda_2}}}} [r(1.5)] {=} [r(2.5)]
{\xybox{\xygraph{!{0;(1.5,0):(0,.6667)::} {q \downarrow h}="tll" [r] {f \downarrow gh}="tl" [r] {D}="tr" [d] {C}="br" [l] {A}="bl"
"tl":"tr"^-{q_3}:"br"^-{gh}:@{<-}"bl"^-{f}:@{<-}"tl"^-{p_3} [d(.5)r(.35)] :@{=>}[r(.3)]^-{\lambda_3}
"tll" (:"bl"_-{pp_2}|-{}="beq",:"tl"^-{k},:@/^{2pc}/"tr"^-{q_2}|-{}="teq") "tl":@{}"teq"|-{=} "beq":@{}"tl"|-{=}}}}} \]
where $\lambda_3$ is another comma square, and $\id:pp_2 \to p_3k$ exhibits $p_3$ as an absolute left Kan extension of $pp_2$ along $k$.
\end{lem}
\begin{proof}
By the universal property of $\lambda_3$ there is a unique $k:q \downarrow h \to f \downarrow gh$ as shown, that is, such that $p_3k=pp_2$, $q_3k=q_2$ and $\lambda_3k=(g\lambda_2)(\lambda p_2)$. Note also that one has $p_4:f \downarrow gh \to f \downarrow g$ unique such that $pp_4=p_3$, $qp_4=hq_3$ and $\lambda p_4=\lambda_3$ by the universal property of $\lambda$. Using the universal property of $\lambda_2$ there is a unique $i:f \downarrow gh \to q \downarrow h$ such that $p_2i=p_4$, $q_2i=q_3$ and $\lambda_2i=\id$. Since $p_3ki=p_3$, $q_3ki=q_3$ and $\lambda_3ki=\lambda_3$, we have $ki=1_{f \downarrow gh}$. Since $pp_2ik=pp_3k$, $qp_2ik=hq_2$, and $\lambda p_2ik=\lambda_3k$ we have a commutative square as on the left in
\[ \xygraph{{\xybox{\xygraph{!{0;(1.5,0):(0,.6667)::} {fpp_2}="tl" [r] {fpp_2ik}="tr" [d] {gqp_2ik}="br" [l] {gqp_2}="bl" "tl":"tr"^-{f \id}:"br"^-{\lambda p_2ik}:@{<-}"bl"^-{g\lambda_2}:@{<-}"tl"^-{\lambda p_2}}}} [r(4)]
{\xybox{\xygraph{!{0;(1.5,0):(0,.6667)::} {qp_2}="tl" [r] {qp_2ik}="tr" [d] {hq_2ik}="br" [l] {hq_2}="bl" "tl":"tr"^-{q\eta'}:"br"^-{\lambda_2ik}:@{<-}"bl"^-{h \id}:@{<-}"tl"^-{\lambda_2}}}}} \]
and so by the universal property of $\lambda$, there is a unique $\eta':p_2 \to p_2ik$ such that $p\eta'=\id$ and $q\eta'=\lambda_2$. Since $q_2ik=q_2$ and $\lambda_2i=\id$ we have a commutative square as on the right in the previous display, and so by the universal property of $\lambda_2$, there is a unique $\eta:1_{q \downarrow h} \to ik$ such that $p_2\eta'=\eta$ and $q_2\eta=\id$. Since $p_3k\eta=\id$ and $q_3k\eta=\id$, we have $k\eta=\id$. Since $p\eta'i=\id$ and $q\eta'i=\id$ we have $\eta'i=\id$, and so $p_2\eta i = \id$, moreover $q_2\eta i=\id$, and so $\eta i = \id$. Thus $\eta$ is the unit of an adjunction $k \ladj i$ in which the counit is an identity. Since $pp_2\eta = \id$, and units of adjunctions are absolute left Kan extensions, the result follows.
\end{proof}
\begin{prop}\label{prop:commas-are-exact}
\cite{Street-FibrationIn2cats}
If $f : A \to C$ and $g : B \to C$ are morphisms of a 2-category $\ca K$ with comma objects, then the defining comma square of $f \downarrow g$ is exact.
\end{prop}
\begin{proof}
If $\psi$ exhibits $l$ as a pointwise left Kan extension $h$ along $f$ then by definition the composite on the left
\[ \xygraph{{\xybox{\xygraph{!{0;(.8,0):(0,1.25)::} {f \downarrow g}="tl" [r(2)] {B}="tr" [l(2)d] {A}="l" [r(2)] {C}="r" [dl] {V}="b" "l":"r"^-{f}:"b"^-{l}:@{<-}"l"^-{h} [d(.5)r(.8)] :@{=>}[r(.4)]^-{\psi}
"tl"(:"tr"^-{q}:"r"^-{g},:"l"_-{p}) "tl" [d(.5)r(.8)] :@{=>}[r(.4)]^-{\lambda}}}}
[r(4.5)u(.333)]
{\xybox{\xygraph{!{0;(1.5,0):(0,.6667)::} {q \downarrow h}="tll" [r] {f \downarrow gh}="tl" [r] {D}="tr" [d] {C}="br" [l] {A}="bl" [r(.5)d] {V}="bb"
"tl":"tr"^-{q_3}:"br"^-{gh}:@{<-}"bl"_-{f}:@{<-}"tl"^-{p_3} [d(.5)r(.35)] :@{=>}[r(.3)]^-{\lambda_3}
"tll" (:"bl"_-{pp_2}|-{}="beq",:"tl"^-{k},:@/^{2pc}/"tr"^-{q_2}|-{}="teq") "tl":@{}"teq"|-{=} "beq":@{}"tl"|-{=}
"bb" (:@{<-}"bl"^{h},:@{<-}"br"_{l}) "bb" [l(.1)u(.5)] :@{=>}[r(.2)]^{\psi}}}}} \]
exhibits $lg$ as a left Kan extension of $hp$ along $q$. We must show that this composite is itself a pointwise left Kan extension. Let $h : D \to B$. Factoring the vertical composite of the comma squares defining $f \downarrow g$ and $q \downarrow h$ as in Lemma \ref{lem:vert-comp-comma}, and pasting this on top of $\psi$ gives the diagram on the right in the previous display. Since $\id$ exhibits $hp_3$ as a left Kan extension of $hpp_2$ along $k$ by Lemma \ref{lem:vert-comp-comma}, and $(l\lambda_3)(\psi p_3)$ exhibits $lgh$ as a left Kan extension of $hp_3$ along $q_3$, the composite exhibits $lgh$ as a left Kan extension of $hpp_2$ along $q_2$ by the composablility of left Kan extensions.
\end{proof}

\subsection{Pullbacks and bipullbacks}
\label{ssec:exact-pb}
We now turn to the discussion of when pullback squares and bipullback squares are exact. The natural conditions we find involve the notions of fibration and  \cite{Street-FibrationIn2cats}, and their bicategorical analogues \cite{Street-FibrationInBicats} as introduced by Street, in a 2-category. Elementary descriptions of fibrations and split fibrations in terms of cartesian 2-cells were given in Section 2 of \cite{Weber-2Toposes} and Section 3 of \cite{Weber-Fam2fun}. The bicategorical notions we require, such as bipullbacks and the bicategorical analogues of fibrations and opfibrations, as well as isofibrations, were given an exposition in Section 3 of \cite{Weber-Fam2fun}. In particular, the bicategorical analogues of fibrations and opfibrations are called bi-fibrations and bi-opfibrations.

There are two ways to express that an arrow $f : A \to C$ in a 2-category $\ca K$ with comma objects has the structure of a fibration (resp. opfibration). There is an elementary description in terms of $f$-cartesian (resp. $f$-opcartesian) 2-cells generalising the notion of a cleavage. Alternatively there is a 2-monad $\Phi_{\ca K,C}$ (resp. $\Psi_{\ca K,C}$) on $\ca K/C$, and then to give $f$ the structure of a fibration (resp. opfibration) is to give it a pseudo $\Phi_{\ca K,C}$-algebra (resp. a pseudo $\Psi_{\ca K,C}$-algebra) structure. If $\ca K$ also has pullbacks, then the equivalence of these alternative definitions is, as explained in Theorem 2.7 of \cite{Weber-2Toposes}, about the interplay between comma objects involving $f$ and pullbacks involving $f$.

Let us recall some of this interplay. Given a pullback square and the corresponding comma square
\[ \xygraph{{\xybox{\xygraph{!{0;(1.25,0):(0,.8)::} {P}="tl" [r] {B}="tr" [d] {C}="br" [l] {A}="bl" "tl":"tr"^-{q}:"br"^-{g}:@{<-}"bl"^-{f}:@{<-}"tl"^-{p} "tl":@{}"br"|-{\tn{pb}}}}} [r(3.5)]
{\xybox{\xygraph{!{0;(1.25,0):(0,.8)::} {f \downarrow g}="tl" [r] {B}="tr" [d] {C}="br" [l] {A}="bl" "tl":"tr"^-{q_2}:"br"^-{g}:@{<-}"bl"^-{f}:@{<-}"tl"^-{p_2} "tl" [d(.5)r(.35)] :@{=>}[r(.3)]^-{\lambda}}}}} \]
one has the unique map $i:P \to f \downarrow g$ such that $p_2i=p$, $q_2i=q$ and $\lambda i = \id$. Theorem 2.7 of \cite{Weber-2Toposes} says in part that $g$ is a fibration iff for all $f$, $i$ has a right adjoint $r$ over $A$, and for this adjunction the unit is invertible. Dually, $f$ is an opfibration iff for all $g$, $i$ has a left adjoint $l$ over $B$, and for this adjunction the counit is invertible.

To obtain the adjunction $i \ladj r$ from the property that $g$ is a fibration, one takes a $g$-cartesian lift $\lambda_2:q_3 \to q_2$ of $\lambda$, induces $r:f \downarrow g \to P$ as the unique map such that $pr=gp_2$ and $qr=q_3$, and then one obtains the counit $\varepsilon:ir \to 1_{f \downarrow g}$ as the unique 2-cell such that $p_2\varepsilon=\id$ and $q_2\varepsilon=\lambda_2$. The $g$-cartesianness of $\lambda_2$ enables us to verify $\varepsilon$ is indeed the counit of an adjunction, and the fully faithfulness of $i$ ensures that the unit $\nu:1_P \to ri$ is invertible. All this is explained in more detail in the proof of Theorem 2.7 in \cite{Weber-2Toposes}. Dually, to obtain the adjunction $l \ladj i$ from the property that $f$ is an opfibration, one takes an $f$-opcartesian lift $\lambda'_2:p_2 \to p_3$ of $\lambda$, induces $l:f \downarrow g \to P$ as the unique map such that $ql=fq_2$ and $pl=p_3$, and then one obtains the unit $\nu':1_{f \downarrow g} \to il$ as the unique 2-cell such that $q_2\nu'=\id$ and $p_2\nu'=\lambda'_2$. The reason for recalling this detail is to obtain
\begin{lem}\label{lem:factor-pb}
Let $f$, $g$ and $i:P \to f \downarrow g$ be given as above, in a 2-category $\ca K$ with comma objects and pullbacks. If $g$ is a fibration then the composite 2-cell on the left in
\[ \xygraph{{\xybox{\xygraph{!{0;(1.5,0):(0,.6667)::} {P}="tll" [r] {f \downarrow g}="tl" [r] {B}="tr" [d] {C}="br" [l] {A}="bl" "tll"(:"tl"^-{i}:"tr"^-{q_2}:"br"^-{g}:@{<-}"bl"^-{f}:@{<-}"tl"_-{p_2},:"bl"_{p}) "tll":@/^{2pc}/"tr"^-{q}|-{}="deq"
"tl" [d(.5)r(.45)] :@{=>}[r(.2)]^-{\lambda} "tll" [d(.5)r(.65)] :@{=>}[r(.2)]^-{p\nu} "deq":@{}"tl"|(.5){=}}}} [r(5)]
{\xybox{\xygraph{!{0;(1.5,0):(0,.6667)::} {f \downarrow g}="tll" [r] {P}="tl" [r] {B}="tr" [d] {C}="br" [l] {A}="bl" "tll"(:"tl"^-{l}:"tr"^-{q}:"br"^-{g}:@{<-}"bl"^-{f}:@{<-}"tl"_-{p},:"bl"_{p_2}) "tll":@/^{2pc}/"tr"^-{q_2}|-{}="deq"
"tl" [d(.5)r(.45)] :@{=>}[r(.2)]^-{\id} "tll" [d(.5)r(.65)] :@{=>}[r(.2)]^-{p_2\nu'} "deq":@{}"tl"|(.5){=}}}}} \]
is the identity. If $f$ is an opfibration then the composite 2-cell on the right in the previous display is $\lambda$.
\end{lem}
\begin{proof}
The case where $f$ is an opfibration follows immediately from the definitions and was observed in \cite{Weber-2Toposes} Example 2.20. Supposing that $g$ is a fibration, note that both $\lambda$ and $fp_2\nu'$ are equal to the composite
\[ \xygraph{!{0;(1.5,0):(0,.6667)::}
{P}="tl" [r] {f \downarrow g}="tm" [dl] {P}="ml" [r] {f \downarrow g}="mm" [r] {B}="mr" [d(1.25)] {C}="br" [l(2)] {A}="bl"
"tl"(:"ml"_-{1}:@{<-}"tm"|-{r},:"tm"^-{i}:"mm"^-{1}) "ml":"mm"_-{i}:"mr"^-{q_2}:"br"^{g}:@{<-}"bl"^-{f}:@{<-}"ml"^{p}:@/_{1.5pc}/"mr"_-{q}
"tl" [d(.35)r(.2)] :@{=>}[r(.2)]^-{\nu} "tl" [d(.65)r(.6)] :@{=>}[r(.2)]^-{\varepsilon}
"mm" [d(.35)] {=} [d(.55)] {=}} \]
because of one of the triangle equations for the adjunction $i \ladj r$, and since by the definition of $\varepsilon$ recalled above one has $gq_2\varepsilon = \lambda$.
\end{proof}
There is an analogous discussion for bi-fibrations and bi-opfibrations, involving the interplay of comma objects and isocomma objects, and an analogue of Lemma \ref{lem:factor-pb} for bi-fibrations and bi-opfibrations. Recall that given arrows $g:B \to C$ and $f:A \to C$ in $\ca K$, their \emph{isocomma object} is a square as on the left
\[ \xygraph{{\xybox{\xygraph{!{0;(1.25,0):(0,.8)::} {f \downarrow_{\iso}g}="tl" [r] {B}="tr" [d] {C}="br" [l] {A}="bl" "tl":"tr"^-{q}:"br"^-{g}:@{<-}"bl"^-{f}:@{<-}"tl"^-{p}
"tl" [d(.5)r(.5)] {\iso} "tl" [d(.25)r(.5)] {\scriptsize \tn{$\lambda$}}}}} [r(3.5)]
{\xybox{\xygraph{!{0;(1.25,0):(0,.8)::} {P}="tl" [r] {B}="tr" [d] {C}="br" [l] {A}="bl" "tl":"tr"^-{}:"br"^-{g}:@{<-}"bl"^-{f}:@{<-}"tl"^-{}
"tl" [d(.5)r(.5)] {\iso}}}}} \]
satisfying the analogous universal property to that of comma object, but only amongst squares over $f$ and $g$ with an invertible 2-cell. Recall that a pseudo-commuting square as on the right in the previous display is a bipullback when the induced map $P \to f\downarrow_{\iso}g$ in $\ca K$ is an equivalence. The following two lemmas will enable us to adapt the arguments which exhibit exact pullbacks, to arguments that exhibit exact bipullbacks, in the proof of Proposition \ref{prop:exact-pullbacks}.
\begin{lem}\label{lem:pullback-cofully-faithful}
Given a pullback square
\[ \xygraph{{P}="tl" [r] {B}="tr" [d] {C}="br" [l] {A}="bl" "tl":"tr"^-{q}:"br"^-{g}:@{<-}"bl"^-{f}:@{<-}"tl"^-{p}:@{}"br"|-{\tn{pb}}} \]
in a 2-category with comma objects, in which $g$ is an isofibration and $f$ has a fully faithful left (resp. right) adjoint. Then $q$ has a fully faithful left (resp. right) adjoint.
\end{lem}
\begin{proof}
We prove the result involving left adjoints; the result involving right adjoints is dual (work in $\ca K^{\co}$). Denote by $l:C \to A$, $\eta:1_C \to fl$ and $\varepsilon:lf \to 1_A$ the data of the adjunction $l \ladj f$, and note that $\eta$ is invertible. Since $g$ is an isofibration we have $h:E \to E$ and an isomorphism $\eta':1_E \to h$ such that $g\eta'=\eta g$ and so $gh=flg$. Using the pullback there exists unique $l':E \to A$ such that $pl'=lg$ and $ql'=h$. We have $pl'g=lgq=lfp$ and $ql'q=hq$, so that $\varepsilon p:pl'q \to p$ and $\eta'^{-1}q:ql'q \to q$. Since the triangle equations for $l \ladj f$ can be written as $f\varepsilon = \eta^{-1}f$ and $\varepsilon l = l\eta^{-1}$, we have $f\varepsilon p = \eta^{-1}fp=\eta^{-1}gq=g\eta'^{-1}q$, and so we have $\varepsilon':l'q \to 1_P$ unique such that $p\varepsilon'=\varepsilon p$ and $q\varepsilon'=\eta'^{-1}q$. It suffices by Lemma 2.6 of \cite{Weber-2Toposes} to show that $q\varepsilon'$ and $\varepsilon'l$ are invertible, and $q\varepsilon'$ 
clearly is by definition. The invertibility of $p\varepsilon'$ is equivalent to that of $p\varepsilon'l'$ and $q\varepsilon'l'$ using the pullback. Now $q\varepsilon'l'=\eta'^{-1}ql'$ and $p\varepsilon'l'=\varepsilon pl'=\varepsilon lg = l\eta^{-1}g$ so the result follows.
\end{proof}
\begin{lem}\label{lem:precomp-exact-equiv}
Suppose that $e:Q \to P$ has a fully faithful left or right adjoint, and that the square on the left
\[ \xygraph{{\xybox{\xygraph{{P}="tl" [r] {B}="tr" [d] {C}="br" [l] {A}="bl" "tl":"tr"^-{q}:"br"^-{g}:@{<-}"bl"^-{f}:@{<-}"tl"^-{p} [d(.5)r(.35)] :@{=>}[r(.3)]^-{\phi}}}} [r(3)]
{\xybox{\xygraph{{Q}="tl" [r] {B}="tr" [d] {C}="br" [l] {A}="bl" "tl":"tr"^-{qe}:"br"^-{g}:@{<-}"bl"^-{f}:@{<-}"tl"^-{pe} [d(.5)r(.35)] :@{=>}[r(.3)]^-{\phi e}}}}} \]
is an exact square in a 2-category with pullbacks and comma objects. Then the square on the right in the previous display is exact.
\end{lem}
\begin{proof}
First observe that if
\[ \xygraph{{A}="l" [r(2)] {C}="r" [dl] {V}="b" "l":"r"^-{f}:"b"^-{k}:@{<-}"l"^-{h} [d(.5)r(.85)] :@{=>}[r(.3)]^-{\psi}} \]
exhibits $k$ as a left Kan extension of $h$ along $f$, and $d:G \to A$ has a fully faithful left or right adjoint, then for all $s:C \to V$, the effect of pasting with $\psi d$ is obtained by applying the composite function
\[ \ca K(C,V)(k,s) \xrightarrow{\tn{paste with} \, \psi} \ca K(A,V)(h,sf) \xrightarrow{(-) \comp d} \ca K(G,X)(hd,sfd) \]
which is a bijection since $\psi$ is a left Kan extension and $(-) \comp d:\ca K(A,X) \to \ca K(G,X)$ is fully faithful. Thus $\psi d$ exhibits $k$ as a left Kan extension of $hd$ along $fd$.

Now suppose $\psi$ exhibits $k$ as a pointwise left Kan extension of $h$ along $f$, and $x:F \to B$. We must show that the composite 2-cell
\[ \xygraph{{qe \downarrow x}="p1" [r(2)] {q \downarrow x}="p2" [r(2)] {F}="p3" [d] {B}="p4" [d] {C}="p5" [dl] {V}="p6" [ul] {A}="p7" [u] {D}="p8" [l(2)] {E}="p9"
"p1":"p2"^-{e_2}:"p3"^-{q_1}:"p4"^-{x}:"p5"^-{g}:"p6"^-{k}:@{<-}"p7"^-{h}:@{<-}"p8"^-{p}:@{<-}"p9"^-{e}:@{<-}"p1"^-{p_2}
"p1":@/^{2pc}/"p3"^-{q_2} "p2":"p8"_{p_1}:"p4"^-{q} "p7":"p5"^-{f}
"p1" [d(.5)r] {\scriptsize{\tn{pb}}} "p2" [u(.35)] {\scriptsize{=}}
"p2" [r(.85)d(.5)] :@{=>}[r(.3)]^{\lambda} "p8" [r(.85)d(.5)] :@{=>}[r(.3)]^{\phi} "p7" [r(.85)d(.5)] :@{=>}[r(.3)]^{\psi}} \]
exhibits $kgx$ as a left Kan extension of $hpep_2$ along $q_2$. Since $\phi$ is exact, the composite of $\psi$, $\phi$ and $\lambda$ is a left Kan extension. Since $p_1$ is a split fibration (as part of a 2-sided discrete fibration $D \to F$) it is in particular an isofibration. Thus by Lemma \ref{lem:pullback-cofully-faithful} $e_2$ has a fully faithful left or right adjoint, and so the result follows by the observation made at the beginning of the proof.
\end{proof}
Note that an equivalence clearly satisfies the hypothesis on $e$ in the above, and so this last result includes the statement that exact squares are stable by precomposition with equivalences.
\begin{prop}\label{prop:exact-pullbacks}
Let $\ca K$ be a 2-category with comma objects and pullbacks.
\begin{enumerate}
\item A pullback square as on the left
\[ \xygraph{{\xybox{\xygraph{{P}="tl" [r] {B}="tr" [d] {C}="br" [l] {A}="bl" "tl":"tr"^-{q}:"br"^-{g}:@{<-}"bl"^-{f}:@{<-}"tl"^-{p}:@{}"br"|-{\tn{pb}}}}}
[r(3)]
{\xybox{\xygraph{{P}="tl" [r] {B}="tr" [d] {C}="br" [l] {A}="bl" "tl":"tr"^-{q}:"br"^-{g}:@{<-}"bl"^-{f}:@{<-}"tl"^-{p}:@{}"br"|*{\iso}}}}} \]
is exact when $f$ is an opfibration or $g$ is a fibration.\label{propcase:ex-pb}
\item If moreover $\ca K$ admits isocomma objects, then a bipullback square as on the right in the previous display is exact when $f$ is a bi-opfibration or $g$ is a bi-fibration.\label{propcase:ex-bipb}
\end{enumerate}
\end{prop}
\begin{proof}
(\ref{propcase:ex-pb}): Suppose that $\psi$ exhibits $k$ as a pointwise left Kan extension of $h$ along $f$. We must show that the composite on the left hand side of
\[ \xygraph{!{0;(2,0):}
{\xybox{\xygraph{{P}="tl" [r(2)] {B}="tr" [l(2)d] {A}="l" [r(2)] {C}="r" [dl] {V}="b" "l":"r"^-{f}:"b"^-{k}:@{<-}"l"^-{h} [d(.5)r(.85)] :@{=>}[r(.3)]^-{\psi}
"tl"(:"tr"^-{q}:"r"^-{g},:"l"_-{p}) "tl":@{}"r"|-{\tn{\scriptsize pb}}}}}
[r] {=} [r(1.5)]
{\xybox{\xygraph{!{0;(2,0):(0,.5)::}
{P}="tll" [r] {f \downarrow g}="tl" [r] {B}="tr" [d] {C}="br" [l] {A}="bl" [r(.5)d] {V}="b"
"tll"(:"tl"^-{i}:"tr"^-{q_1}:"br"^-{g}:@{<-}"bl"_-{f}:@{<-}"tl"_-{p_1},:"bl"_{p}) "bl":"b"_{h}:@{<-}"br"_{k}
"tl" [d(.5)r(.425)] :@{=>}[r(.15)]^-{\lambda} "tll" [d(.5)r(.65)] :@{=>}[r(.2)]^-{p\nu}
"bl" [d(.4)r(.425)] :@{=>}[r(.15)]^-{\psi}}}}} \]
exhibits $kg$ as a pointwise left Kan extension of $hp$ along $q$ when $f$ is an opfibration or $g$ is a fibration. In the case where $g$ is a fibration, one has by Lemma \ref{lem:factor-pb} that the above equation holds, in which $\nu$ is the unit of an adjunction $i \ladj r$ coming from the fact that $g$ is a fibration. Since $hp\nu$ exhibits $hp_1$ as a left Kan extension of $hp$ along $i$, and the composite of $\psi$ and $\lambda$ exhibits $kg$ as a left Kan extension of $fp_1$ along $q_1$, the composite exhibits $kg$ as a left Kan extension of $hp$ along $q$. Thus when $g$ is a fibration $\psi p$ exhibits $kg$ as a left Kan extension of $hp$ along $q$.

For the general cases let $x:D \to B$, consider the diagram on the left when $f$ is an opfibration
\[ \xygraph{{\xybox{\xygraph{!{0;(1.5,0):(0,.6667)::}
{f \downarrow g}="tl" [r] {P}="tm" [r] {B}="tr" [d] {C}="br" [l] {A}="bl" [r(0.5)d] {V}="b" "tr" [u] {D}="ttr" [l(2)] {q_1 \downarrow x}="ttl" [r] {q \downarrow x}="ttm"
"tl":"tm"^-{l}:"tr"^-{q}:"br"^{g}:@{<-}"bl"_-{f}:@{<-}"tl"^{p_1}
"tm":"bl"^{p} "bl":"b"_-{h} "br":"b"^-{k}
"tl" [d(.5)r(.65)] :@{=>}[r(.2)]^-{p_1\nu'} "tm":@{}"br"|-{\tn{pb}}
"bl" [d(.5)r(.425)] :@{=>}[r(.15)]^-{\psi}
"ttl":"tl"_{p_3} "ttl":"ttm"^-{q_4}:"ttr"^-{q_2}:"tr"^{x} "ttm":"tm"_{p_2} "ttl":@/^{1.5pc}/"ttr"^-{q_3}
"ttm" [d(.5)r(.4)] :@{=>}[r(.2)]^-{\lambda_2} "ttl" [d(.5)r(.5)] {\tn{\scriptsize pb}} "ttm" [u(.25)] {=}}}}
[r(5)]
{\xybox{\xygraph{!{0;(1.5,0):(0,.6667)::}
{P}="tl" [r] {f \downarrow g}="tm" [r] {B}="tr" [d] {C}="br" [l] {A}="bl" [r(0.5)d] {V}="b" "tr" [u] {D}="ttr" [l(2)] {q \downarrow x}="ttl" [r] {q_1 \downarrow x}="ttm"
"tl":"tm"^-{i}:"tr"^-{q_1}:"br"^{g}:@{<-}"bl"_-{f}:@{<-}"tl"^{p}
"tm":"bl"^{p_1} "bl":"b"_-{h} "br":"b"^-{k}
"tl" [d(.5)r(.65)] :@{=>}[r(.2)]^-{p\nu} "tm" [d(.5)r(.4)] :@{=>}[r(.2)]^-{\lambda}
"bl" [d(.5)r(.425)] :@{=>}[r(.15)]^-{\psi}
"ttl":"tl"_{p_2} "ttl":"ttm"^-{q_5}:"ttr"^-{q_3}:"tr"^{x} "ttm":"tm"_{p_3} "ttl":@/^{1.5pc}/"ttr"^-{q_2}
"ttm" [d(.5)r(.4)] :@{=>}[r(.2)]^-{\lambda_3} "ttl" [d(.5)r(.5)] {\tn{\scriptsize pb}} "ttm" [u(.25)] {=}}}}} \]
and the diagram on the right in the previous display when $g$ is a fibration. To form these diagrams one takes the comma objects $\lambda_2$ and $\lambda_3$, and then since $ql=q_1$ and $q_1i=q$, one induces $q_4$ and $q_5$ so that $\lambda_2q_4=\lambda_3$ and $\lambda_3q_5=\lambda_2$. The goal is to verify that the composite of $\lambda_2$, the pullback square and $\psi$ exhibits $kgx$ as a left Kan extension of $hpp_2$. Note this composite 2-cell equals the composite on the right in the previous display by definition.

Suppose that $f$ is an opfibration. The composite on the left is a left Kan extension, since it is by definition the composite of $\lambda_3$, $\lambda$ and $\psi$, $\psi$ is a pointwise left Kan extension and by Proposition \ref{prop:commas-are-exact}. By Theorem 3.5 of \cite{Weber-Fam2fun} $p_2$ is a fibration. Thus by the case considered at the beginning of this proof $hp_1\nu'p_3$ exhibits $hpp_2$ as a left Kan extension of $hp_1p_3$ along $q_4$. Thus the result in this case follows by the cancellability of left Kan extensions.

Suppose that $g$ is a fibration. Since $p_3$ is a fibration by Theorem 3.5 of \cite{Weber-Fam2fun}, by the case considered at the beginning of this proof $hp\nu p_2$ exhibits $hp_1p_3$ as a left Kan extension of $hpp_2$ along $q_5$. Since $\psi$ is a pointwise left Kan extension, the composite of $\lambda_3$, $\lambda$ and $\psi$ exhibits $kgx$ as a left Kan extension of $hp_1p_3$ along $q_3$. Thus the result in this case follows by the composability of left Kan extensions.

(\ref{propcase:ex-bipb}): By Lemma \ref{lem:precomp-exact-equiv} it suffices to consider only the case where the bipullback is an isocomma square. From here one establishes the analogue of Lemma \ref{lem:factor-pb} for bi-fibrations and bi-opfibrations, and proceeds exactly as for (\ref{propcase:ex-pb}), except in that all the pullbacks that arise in the discussion are replaced by the corresponding isocomma objects.
\end{proof}

\subsection{Colax idempotent 2-monads}
\label{ssec:general-exactness-in-colaxidem-case}
Recall that a 2-monad $(\ca K,T)$ is colax idempotent when $\eta^TT \ladj \mu^T$, basic example being $\FPMnd$. For such 2-monads it turns out that all colax algebra morphisms are exact.

Recall that when $(\ca K,T)$ is colax idempotent, to give a pseudo $T$-algebra structure on $A \in \ca K$, is to give a right adjoint $a$ to $\eta^T_A : A \to TA$ in $\ca K$. Moreover, given pseudo $T$-algebras $(I,i)$ and $(J,j)$, and a morphism $f : I \to J$ in $\ca K$, the unique 2-cell $\overline{f}$ satisfying
\begin{equation}\label{eq:colax-coh-morphism-alg-coKZ} 
\xygraph{!{0;(2.25,0):(0,1)::} 
{\xybox{\xygraph{!{0;(.75,0):(0,1.3333)::} {I}="p0" [r(2)] {TI}="p1" [r(2)] {TJ}="p2" [dl] {J}="p3" [l(2)] {I}="p4" "p0":"p1"^-{\eta_I}:"p2"^-{Tf}:"p3"^-{j}:@{<-}"p4"^-{f}:@{<-}"p0"^-{1_A}
"p1":"p4"^{i} "p0" [d(.5)r(.8)] :@{=>}[r(.4)]^{u_I} "p1" [d(.5)r(.3)] :@{=>}[r(.4)]^{\overline{f}}
}}}
[r] {=} [r]
{\xybox{\xygraph{!{0;(.75,0):(0,1.3333)::} {I}="p0" [r(2)] {J}="p1" [r(2)] {TJ}="p2" [dl] {J}="p3" [l(2)] {I}="p4" "p0":"p1"^-{f}:"p2"^-{\eta_J}:"p3"^-{j}:@{<-}"p4"^-{f}:@{<-}"p0"^-{1_I}:@{}"p3"|-{=}
"p1":"p3"_-{1_J} "p1" [d(.5)r(.8)] :@{=>}[r(.4)]^{u_J}}}}}
\end{equation}
where $u_I$ (resp. $u_J$) is the unit of $\eta_I \ladj i$ (resp. $\eta_J \ladj j$),  provides the coherence datum for a colax morphism $(I,i) \to (J,j)$ of pseudo $T$-algebras.
\begin{prop}\label{prop:exactness-coKZ}
Let $(\ca K,T)$ be a colax idempotent 2-monad on a 2-category $\ca K$ with comma objects. Then any colax morphism of pseudo $T$-algebras is exact.
\end{prop}
\begin{proof}
Given pseudo $T$-algebras $(I,i)$ and $(J,j)$ and $f : I \to J$, we must show that the unique 2-cell $\overline{f}$ satisfying (\ref{eq:colax-coh-morphism-alg-coKZ}) is exact. Given $\psi$ as below which exhibits $h$ as a pointwise left Kan extension of $g$ along $f$, for any $k : K \to TJ$, by (\ref{eq:colax-coh-morphism-alg-coKZ}) we have
\begin{equation}\label{eq:exactness-coKZ} 
\xygraph{!{0;(3,0):(0,1)::} 
{\xybox{\xygraph{!{0;(.75,0):(0,1.3333)::} {I}="p0" [r(2)] {TI}="p1" [r(2)] {TJ}="p2" [dl] {J}="p3" [l(2)] {I}="p4" "p0":"p1"^-{\eta_I}:"p2"^-{Tf}:"p3"^-{j}:@{<-}"p4"_-{f}:@{<-}"p0"^-{1_A}
"p1":"p4"^{i} "p0" [d(.5)r(.8)] :@{=>}[r(.4)]^{u_I} "p1" [d(.5)r(.3)] :@{=>}[r(.4)]^{\overline{f}}
"p3" [dl] {A}="b"
"b" (:@{<-}"p3"_-{h},:@{<-}"p4"^-{g})
"p4" [d(.5)r(.8)] :@{=>}[r(.4)]^{\psi}
"p0" [u] {\eta_Jf \downarrow k}="t0" [r(2)]  {Tf \downarrow k}="t1" [r(2)] {K}="t2"
"p0":@{<-}"t0"^-{p_3}:"t1"^-{q_3}:"t2"^-{q_1}:"p2"^-{k} "t1":"p1"_-{p_1} "t0" [d(.45)r] {\scriptstyle{\tn{pb}}} "t1" [d(.5)r(.8)] :@{=>}[r(.4)]^-{\lambda_1}}}}
[r] {=} [r(.9)]
{\xybox{\xygraph{!{0;(.75,0):(0,1.3333)::} {I}="p0" [r(2)] {J}="p1" [r(2)] {TJ}="p2" [dl] {J}="p3" [l(2)] {I}="p4" "p0":"p1"^-{f}:"p2"^-{\eta_J}:"p3"^-{j}:@{<-}"p4"_(.6){f}:@{<-}"p0"^-{1_I}:@{}"p3"|-{=}
"p1":"p3"_-{1_J} "p1" [d(.5)r(.8)] :@{=>}[r(.4)]^{u_J}
"p3" [dl] {A}="b"
"b" (:@{<-}"p3"_-{h},:@{<-}"p4"^-{g})
"p4" [d(.5)r(.8)] :@{=>}[r(.4)]^{\psi}
"p0" [u] {\eta_Jf \downarrow k}="t0" [r(2)]  {\eta_J \downarrow k}="t1" [r(2)] {K}="t2"
"p0":@{<-}"t0"^-{p_3}:"t1"^-{q_4}:"t2"^-{q_2}:"p2"^-{k} "t1":"p1" "t1":"p1"_-{p_2} "t0" [d(.45)r] {\scriptstyle{\tn{pb}}} "t1" [d(.5)r(.8)] :@{=>}[r(.4)]^-{\lambda_2}}}}}
\end{equation}
in which $\lambda_1$ and $\lambda_2$ are comma 2-cells. Our goal is to show that the composite of $\lambda_1$, $\overline{f}$ and $\psi$ on the left hand side of (\ref{eq:exactness-coKZ}) exhibits $hjk$ as a left Kan extension of $gip_1$ along $q_1$. As the first projection maps of comma squares, $p_1$ and $p_2$ are fibrations, and so by Proposition \ref{prop:exact-pullbacks}, the pullback squares in the above diagrams are exact. Thus on the right hand side of (\ref{eq:exactness-coKZ}) $\psi p_3$ exhibits $hp_2$ as a left Kan extension of $gp_3$ along $q_4$. Moreover as units of adjunctions are absolute pointwise left Kan extensions, $gu_Ip_3$ exhibits $gip_1$ as a left Kan extension of $gp_3$ along $q_3$, and $(hj\lambda_2)(hu_Jp_2)$ exhibits $hjk$ as a left Kan extension of $hp_2$ along $q_2$. By the composability of left Kan extensions, the composite on the right hand side of (\ref{eq:exactness-coKZ}) is a left Kan extension, and so by the cancellability of left Kan extensions and (\ref{eq:exactness-coKZ}), the result follows.
\end{proof}
\begin{exam}\label{ex:fpp-ple}
Let $I$ and $J$ be small categories with finite products, and $A$ be a cocomplete cartesian closed category. If $f : I \to J$ is any functor and $g : I \to A$ is a finite product preserving functor, then by Theorem \ref{thm:AlgKan}(\ref{thmcase:AK-pseudo}) we recover the classical fact \cite{Lawvere-FunctorialSemantics} that a pointwise left Kan extension of $g$ along $f$ is finite product preserving. Moreover, one can weaken the hypotheses required of $A$ given a fixed $f$, asking only for enough colimits in $A$ to enable computation of left Kan extensions along $f$, and that just these colimits are preserved by functors of the form $X \times (-)$ for $X \in A$.
\end{exam}

\subsection{Exact colax monoidal functors}
\label{ssec:exact-colax-monoidal-functors}
In this section we characterise the exact colax morphisms of algebras for the 2-monads $\MCMnd$, $\BMCMnd$ and $\SMCMnd$ of monoidal, braided monoidal and symmetric monoidal categories respectively, on $\CAT$. From these characterisations, one easily exhibits examples of colax morphisms of algebras which are not exact for these 2-monads, by contrast with the previous section.

We now characterise exact colax monoidal functors, using the explicit reformulation of exactness given by Lemma \ref{lem:combinatorial-exactness-Cat-fact-version}. Given a colax monoidal functor $F : \ca V \to \ca W$, $X \in \ca V$, $Y_1$ and $Y_2 \in \ca W$, and $f: FX \to Y_1 \tensor Y_2$, we denote by $\tn{Fact}_F(X,f,Y_1,Y_2)$ the following category. Its objects are 5-tuples $(g,Z_1,Z_2,h_1,h_2)$ providing a factorisation
\[ FX \xrightarrow{Fg} F(Z_1 \tensor Z_2) \xrightarrow{\overline{F}_2} FZ_1 \tensor FZ_2 \xrightarrow{h_1 \tensor h_2} Y_1 \tensor Y_2 \]
of $f$. An arrow $(g,Z_1,Z_2,h_1,h_2) \to (g',Z'_1,Z'_2,h'_1,h'_2)$ is a pair $(k_1,k_2)$ where $k_1 : Z_1 \to Z'_1$ and $k_2 : Z_2 \to Z'_2$ such that $(k_1 \tensor k_2)g = g'$, $h_1 = k_1F(k_1)$ and $h_2 = k_2F(k_2)$. Identities and compositions are inherited from $\ca V$.
\begin{prop}\label{prop:exact-colax-monoidal-functors}
A colax monoidal functor $F : \ca V \to \ca W$ between monoidal categories is exact as a colax morphism of pseudo $\MCMnd$-algebras iff
\begin{enumerate}
\item For any $X \in \ca V$ and $f:FX \to I$ in $\ca W$, there is a unique $g : X \to I$ in $\ca V$ such that $f = \overline{F}_0F(g)$. \label{propcase:colax-monoidal-nullary}
\item For any $X \in \ca V$, $Y_1$, $Y_2 \in \ca W$ and $f : FX \to Y_1 \tensor Y_2$, the category $\tn{Fact}_F(X,f,Y_1,Y_2)$ defined above is connected. \label{propcase:colax-monoidal-binary}
\end{enumerate}
\end{prop}
\begin{proof}
Applying the Lemma \ref{lem:combinatorial-exactness-Cat-fact-version}, the exactness of
\[ \xygraph{!{0;(1.5,0):(0,.6667)::}
{\MCMnd\ca V}="p0" [r] {\MCMnd\ca W}="p1" [d] {\ca W}="p2" [l] {\ca V}="p3" "p0":"p1"^-{\MCMnd F}:"p2"^-{\bigotimes}:@{<-}"p3"^-{F}:@{<-}"p0"^-{\bigotimes} "p0" [d(.55)r(.35)] :@{=>}[r(.3)]^-{\overline{F}}} \]
says that for all $n \in \N$, $X \in \ca V$, $Y_1,...,Y_n \in \ca W$ and $f : FX \to \bigotimes_{i=1}^nY_i$, the category $\tn{Fact}_{\overline{F}}(X,f,(Y_i)_i)$, the objects of which are factorisations
\[ \begin{array}{c} {{FX} \xrightarrow{Fg} {F\bigotimes_iX_i} \xrightarrow{\overline{F}_n} {\bigotimes_iFX_i} \xrightarrow{\bigotimes_ih_i} {\bigotimes_iY_i}} \end{array} \]
of $f$, is connected. Fixing $n \in \N$, we denote by $P_n$ the statement that for all $X$, $(Y_i)_{1{\leq}i{\leq}n}$ and $f$ the category $\tn{Fact}_{\overline{F}}(X,f,(Y_i)_i)$ is connected. Thus the exactness of $(F,\overline{F})$ as a colax morphism of pseudo $\MCMnd$-algebras is the statement $(\forall n \in \N, P_n)$.

When $n = 0$ the sequence $(Y_i)_i$ can only be the empty sequence $()$ and we recall that the $0$-ary tensor product is by definition the tensor unit $I$. Thus the category $\tn{Fact}_{\overline{F}}(X,(),f)$ is discrete, and is just the set of $g : X \to I$ in $\ca V$ such that $f = \overline{F}_0F(g)$. Thus $P_0$ is equivalent to (\ref{propcase:colax-monoidal-nullary}). Since by definition $\tn{Fact}_{\overline{F}}(X,f,(Y_1,Y_2)) = \tn{Fact}_F(X,f,Y_1,Y_2)$, $P_2$ is just (\ref{propcase:colax-monoidal-binary}) and so we have shown $(\implies)$. Note also that $P_1$ is always true because $(1_X,X,f)$ is an initial object of $\tn{Fact}_{\overline{F}}(X,f,(Y))$. Thus it suffices to show that for $n \geq 2$, $P_2 \wedge P_n \implies P_{n+1}$.

By the coherence theorem for monoidal categories it suffices to consider the case where $\ca V$ and $\ca W$ are strict. Let $f : FX \to \bigotimes_{i=1}^{n+1} Y_i$ and define the categories $\ca E$ and $\ca B$ as
\[ \begin{array}{lccr} {\ca E = \tn{Fact}_{\overline{F}}(X,f,(Y_i)_i)} &&&
{\ca B = \tn{Fact}_F(X,f,\bigotimes_{i=1}^nY_i,Y_{n+1}).} \end{array} \]
By $P_2$, $\ca B$ is connected. For any object $(g,(X_i)_i,(h_i)_i)$ of $\ca E$ one has, by factoring the coherence $\overline{F}_{n+1}$ as in the commutative diagram
\[ \xygraph{!{0;(2.5,0):(0,.4)::} {FX}="p0" [r] {F\bigotimes_{i=1}^{n+1}X_i}="p1" [r(1.25)] {\bigotimes_{i=1}^{n+1}FX_i}="p2" [r] {\bigotimes_{i=1}^{n+1}Y_i}="p3" [l(1.625)d] {F(\bigotimes_{i=1}^nX_i) \tensor FX_{n+1}}="p4" "p0":"p1"^-{Fg}:"p2"^-{\overline{F}_{n+1}}:"p3"^-{\bigotimes_ih_i}:@{<-}"p4"^-{} "p1":"p4"^(.6){\overline{F}_2}:"p2"^(.4){\overline{F}_n \tensor \id}} \]
an object $(g,(\bigotimes_{i=1}^nX_i,X_{n+1}),((\bigotimes_{i=1}^nh_i)\overline{F}_n,h_{n+1}))$ of $\ca B$, this being the object map of a functor $Q : \ca E \to \ca B$.

Let $(g_1,(X_{\bullet},X_{n+1}),(h_{\bullet},h_{n+1}))$ be an object of $\ca B$. Then by $P_n$ one can factor $h_{\bullet}$ as
\[ \begin{array}{c} {FX_{\bullet} \xrightarrow{Fg_2} F\bigotimes_{i=1}^nX_i \xrightarrow{\overline{F}} \bigotimes_{i=1}^nFX_i \xrightarrow{\bigotimes_ih_i} \bigotimes_{i=1}^nY_i} \end{array} \]
and so one has $(g,(X_i)_i,(h_i)_i)$ in $\ca E$, where $g = (g_2 \tensor 1_{X_{n+1}})g_1$. Moreover one has a morphism $(g_2,1_{X_{n+1}}) : (g_1,(X_{\bullet},X_{n+1}), (h_{\bullet},h_{n+1})) \to Q(g,(X_i)_i,(h_i)_i)$ in $\ca B$. That is, for all $B \in \ca B$, there exists $E \in \ca E$ and $B \to QE$ in $\ca B$. Thus since $\ca B$ is non-empty, $\ca E$ is non-empty.

Since $\ca B$ is connected, it suffices to show that if $E_1$ and $E_2$ are objects of $\ca E$ such that there exists an arrow $QE_1 \to QE_2$ of $\ca B$, then there exists an undirected path $E_1 \to ... \leftarrow E_2$ in $\ca E$. Denoting
\[ \begin{array}{lccr} {E_1 = (g_1,(X_{1,i})_i,(h_{1,i})_i)} &&& {E_2 = (g_2,(X_{2,i})_i,(h_{2,i})_i)} \end{array} \]
the data of $QE_1 \to QE_2$ amounts to $a$ and $b$ as in
\[ \xygraph{!{0;(1.25,0):(0,.7)::} {FX}="p0" [u(2)r] {F\bigotimes_{i=1}^{n+1}X_{1,i}}="p1" [r(3)] {F(\bigotimes_{i=1}^nX_{1,i}) \tensor FX_{1,n+1}}="p2" [dr] {\bigotimes_{i=1}^{n+1}FX_{1,i}}="p3" [dr] {\bigotimes_{i=1}^{n+1}Y_i}="p4" [dl] {\bigotimes_{i=1}^{n+1}FX_{2,i}}="p5" [dl] {F(\bigotimes_{i=1}^nX_{2,i}) \tensor FX_{2,n+1}}="p6" [l(3)] {F\bigotimes_{i=1}^{n+1}X_{2,i}}="p7" "p0":"p1"^-{Fg_1}:"p2"^-{\overline{F}_2}:"p3"^(.6){\overline{F}_n \tensor \id}:"p4"^(.6){\bigotimes_ih_{1,i}}:@{<-}"p5"^(.4){\bigotimes_ih_{2,i}}:@{<-}"p6"^(.4){\overline{F}_n \tensor \id}:@{<-}"p7"^-{\overline{F}_2}:@{<-}"p0"^-{Fg_2}
"p1":"p7"^-{F(a \tensor b)} "p2":"p6"_-{Fa \tensor Fb}} \]
such that
\[ \begin{array}{lcccr} {(a \tensor b)g_1 = g_2} && {\bigotimes_{i=1^n}h_{2,i}\overline{F}_nF(a) = \bigotimes_{i=1}^nh_{1,i}\overline{F}_n} && {h_{2,n+1}F(b) = h_{1,n+1}} \end{array} \]
so that in particular, the above diagram commutes. Denoting by $f_2$ the common morphism $F\bigotimes_{i=1}^nX_{1,i} \to \bigotimes_{i=1}^nY_i$, the category $\tn{Fact}_{\overline{F}}(\bigotimes_{i=1}^nX_{1,i},f_2,(Y_i)_i)$ is connected by $P_n$. Thus one has $m \in \N$ and for $1 \leq i \leq n$ an undirected path
\[ X_{1,i} \xrightarrow{\delta_{1,i}} Z_{1,i} \leftarrow {...} \rightarrow Z_{m,i} \xleftarrow{\delta_{m+1,i}} X_{2,i} \]
in $\ca V$, morphisms $c_j : F\bigotimes_{i=1}^nX_{1,i} \to F\bigotimes_{i=1}^nZ_{j,i}$ and $k_{j,i} : FZ_{j,i} \to Y_i$ of $\ca W$ assembling together to form an undirected path
\[ \xygraph{!{0;(2,0):(0,.5)::} {(1,(X_{1,i})_i,h_{1,i})}="p0" [d] {(c_1,(Z_{1,i}),(k_{1,i})_i)}="p1" [r] {...}="p2" [r] {(c_m,(Z_{m,i}),(k_{m,i})_i)}="p3" [u] {(a,(X_{2,i})_i,(h_{2,i})_i)}="p4" "p0":"p1"_-{(\delta_{1,i})_i}:@{<-}"p2"^-{(\delta_{2,i})_i}:"p3"^-{(\delta_{m,i})_i}:@{<-}"p4"_-{(\delta_{m+1,i})_i}} \]
in $\tn{Fact}_{\overline{F}}(\bigotimes_{i=1}^nX_{1,i},f_2,(Y_i)_i)$. Extending the definition of $Z_{j,i}$, $k_{j,i}$, and $\delta_{j,i}$ as follows
\[ \begin{array}{rcl}
{Z_{j,i}} & = & {\left\{
\begin{array}{lll}
{X_{1,n+1}} && {i = n+1} \\
{X_{2,i}} && {j=m+1, i \leq n}
\end{array}
\right.}
\end{array} \]
\[ \begin{array}{rcl}
{k_{j,i}} & = & {\left\{
\begin{array}{lll}
{h_{1,n+1}} && {i = n+1} \\
{h_{2,i}} && {j=m+1, i \leq n}
\end{array}
\right.}
\end{array} \]
\[ \begin{array}{rcl}
{\delta_{j,i}} & = & {\left\{
\begin{array}{lll}
{1_{X_{1,n+1}}} && {j \leq m+1, i = n+1} \\
{1_{X_{2,i}}} && {j = m+2, i \leq n} \\
{b} && {j = m+2, i \leq n+1}
\end{array}
\right.}
\end{array} \]
and defining
\[ \begin{array}{rcl}
{d_j} & = & {\left\{
\begin{array}{lll}
{(c_1 \tensor 1_{X_{1,n+1}})g_1} && {j \leq m} \\
{(a \tensor 1_{X_{1,n+1}})g_1} && {j = m+1}
\end{array}
\right.}
\end{array} \]
the information at hand assembles together to form an undirected path
\[ \xygraph{!{0;(1.9,0):(0,.5)::} {(g_1,(X_{1,i})_i,(h_{1,i})_i)}="p0" [d] {(d_1,(Z_{1,i}),(k_{1,i})_i)}="p1" [r] {...}="p2" [r] {(d_m,(Z_{m,i}),(k_{m,i})_i)}="p3" [u] {(d_{m+1},(Z_{m+1,i}),(k_{m+1,i})_i)}="p4" [r(2)] {(g_2,(X_{2,i})_i,(h_{2,i})_i)}="p5" "p0":"p1"_-{(\delta_{1,i})_i}:@{<-}"p2"^-{(\delta_{2,i})_i}:"p3"^-{(\delta_{m,i})_i}:@{<-}"p4"_-{(\delta_{m+1,i})_i}:"p5"^-{(\delta_{m+2,i})_i}} \]
from $E_1$ to $E_2$ in $\ca E$ as required.
\end{proof}
\begin{exam}\label{exam:cartesian-comonoidal->exact}
When $\ca V$ and $\ca W$ are cartesian monoidal any functor $F : \ca V \to \ca W$ is colax monoidal in a unique way, with $\overline{F}$ being provided by the product obstruction maps. Condition \ref{propcase:colax-monoidal-nullary} of Proposition \ref{prop:exact-colax-monoidal-functors} follows immediately in this case because the units of $\ca V$ and $\ca W$ are terminal. Condition \ref{propcase:colax-monoidal-binary} of Proposition \ref{prop:exact-colax-monoidal-functors} follows in this case because
\[ FX \xrightarrow{F\nabla_X} F(X \times X) \xrightarrow{\overline{F}_2} FX \times FX \xrightarrow{f\tn{pr}_1 \times f\tn{pr}_2} Y_1 \times Y_2 \]
is initial in $\tn{Fact}_F(X,f,Y_1,Y_2)$.
\end{exam}
\begin{exam}\label{exam:counterex-exactness-M}
Given a monoidal category $\ca V$, taking $F$ in Proposition \ref{prop:exact-colax-monoidal-functors} to be the unique $\ca V \to 1$, condition \ref{propcase:colax-monoidal-nullary} says exactly that $\ca V$'s unit must be terminal. Thus for any $\ca V$ whose unit is not terminal, the unique (strict monoidal) functor $\ca V \to 1$ is not exact. In particular this applies to any $\ca V$ of the form $\MCMnd\ca A$ where $\ca A$ is non-empty.
\end{exam}
\begin{exam}\label{exam:exact-comonoids}
A colax monoidal functor $F : 1 \to \ca W$ is the same thing as a comonoid $(C,\varepsilon,\delta)$ in $\ca W$, with $C$ being the effect of $F$ on the unique object of $1$, the counit $\varepsilon : C \to I$ being the unit coherence $\overline{F}_0$ and $\delta : C \to C \tensor C$ being the unique component of $\overline{F}_2$. Condition \ref{propcase:colax-monoidal-nullary} of Proposition \ref{prop:exact-colax-monoidal-functors} in this case says that the only morphism $C \to I$ in $\ca W$ is the counit $\varepsilon$ of the comonoid. Condition \ref{propcase:colax-monoidal-binary} of Proposition \ref{prop:exact-colax-monoidal-functors} in this case says that any $f : C \to Y_1 \tensor Y_2$ factors as
\[ C \xrightarrow{\delta} C \tensor C \xrightarrow{h_1 \tensor h_2} Y_1 \tensor Y_2 \]
for unique morphisms $h_1$ and $h_2$. Taking $\ca W$ to be the category $\tnb{Vect}_k$ of vector spaces over a field $k$ with its usual tensor product (that classifies bilinear maps), the vector space $k[x]$ of polynomials in 1-variable admits a simple well-known comonoid structure, with counit and comultiplication defined on the basis $\{1,x,x^2,...\}$ by
\[ \begin{array}{lccr}
{\begin{array}{rcl}
{\varepsilon(x^n)} & = & {\left\{
\begin{array}{lll}
1 && {n = 0} \\
0 && {n \neq 0}
\end{array}
\right.}
\end{array}}
&&&
{\begin{array}{rcl}
{\delta(x^n)} & = & {\sum\limits_{i=0}^n x^i \tensor x^{n-i}.}
\end{array}} \end{array} \]
Since there is more than one linear map $k[x] \to k$, any comonoid structure on $k[x]$ will fail to satisfy condition \ref{propcase:colax-monoidal-nullary}, giving more counterexamples to exactness for the 2-monad $\MCMnd$. We leave to the reader the routine exercise of showing that Condition \ref{propcase:colax-monoidal-binary} is also violated for the above comonoid structure on $k[x]$.
\end{exam}
A complete understanding of when a symmetric (resp. braided) colax monoidal functor is exact as a colax morphism of pseudo $\SMCMnd$-algebras (resp. $\BMCMnd$-algebras) is obtained from Proposition \ref{prop:exact-colax-monoidal-functors} and
\begin{prop}\label{prop:exactness-for-Sm-and-Bm}
\begin{enumerate}
\item A symmetric colax monoidal functor $F : \ca V \to \ca W$ between symmetric monoidal categories is exact as a colax morphism of pseudo $\SMCMnd$-algebras iff it is exact as a colax morphism of pseudo $\MCMnd$-algebras.
\item A braided colax monoidal functor $F : \ca V \to \ca W$ between braided monoidal categories is exact as a colax morphism of pseudo $\BMCMnd$-algebras iff it is exact as a colax morphism of pseudo $\MCMnd$-algebras.
\end{enumerate}
\end{prop}
\begin{proof}
We give the proof only in the symmetric case since the proof in the braided case is almost identical. We denote by $J_{\ca V} : \MCMnd\ca V \to \SMCMnd\ca V$ the identity on objects inclusion, which is the $\ca V$-component of a morphism of monads $\MCMnd \to \SMCMnd$. Our task is to verify that the square on the left, which is the colax $\SMCMnd$-morphism coherence for $F$,
\[ \xygraph{{\xybox{\xygraph{!{0;(1.5,0):(0,.6667)::} {\SMCMnd\ca V}="p0" [r] {\SMCMnd\ca W}="p1" [d] {\ca W}="p2" [l] {\ca V}="p3" "p0":"p1"^-{\SMCMnd F}:"p2"^-{\bigotimes}:@{<-}"p3"^-{F}:@{<-}"p0"^-{\bigotimes} "p0" [d(.55)r(.35)] :@{=>}[r(.3)]^-{\overline{F}}}}}
[r(4)]
{\xybox{\xygraph{!{0;(1.5,0):(0,.6667)::} {\MCMnd\ca V}="p0" [r] {\MCMnd\ca W}="p1" [d] {\ca W}="p2" [l] {\ca V}="p3" "p0":"p1"^-{\MCMnd F}:"p2"^-{\bigotimes}:@{<-}"p3"^-{F}:@{<-}"p0"^-{\bigotimes} "p0" [d(.55)r(.35)] :@{=>}[r(.3)]^-{\overline{F}J_{\ca V}}}}}} \]
is exact iff the square on the right is exact, because the square on the right is the colax $\MCMnd$-morphism coherence 2-cell for $F$. By Lemma \ref{lem:combinatorial-exactness-Cat-fact-version}, for $X \in \ca V$, $n \in \N$, $Y_1,...,Y_n \in \ca W$ and $f : FX \to \bigotimes_{i=1}^n Y_i$, it suffices to show that
\[ \tn{Fact}_{\overline{F}J_{\ca V}}(X,f,(Y_i)_i) \catequiv \tn{Fact}_{\overline{F}}(X,f,(Y_i)_i).  \]

In the proof of Proposition \ref{prop:exact-colax-monoidal-functors} we unpacked the category $\tn{Fact}_{\overline{F}J_{\ca V}}(X,f,(Y_i)_i)$ in explicit terms, and we now proceed to the same for $\tn{Fact}_{\overline{F}}(X,f,(Y_i)_i)$. An object of this last category is a 4-tuple $(g,(Z_i)_i,\rho,(h_i)_i)$ where $Z_i \in \ca V$, $g : X \to \bigotimes_{i=1}^nZ_i$, $\rho \in \Sigma_n$ and $h_i : FZ_i \to Y_{\rho i}$, providing a factorisation
\[ FX \xrightarrow{Fg} F\bigotimes_{i=1}^nZ_i \xrightarrow{\overline{F}_n} \bigotimes_{i=1}^nFZ_i \xrightarrow{\bigotimes(\rho,(h_i)_i)} \bigotimes_{i=1}^nY_i  \]
of $f$. A morphism $(g_1,(Z_{1,i})_i,\rho,(h_{1,i})_i) \to (g_2,(Z_{2,i})_i,\rho,(h_{2,i})_i)$ is a pair $(\rho_3,(\delta_i)_i)$ where $\rho_3 \in \Sigma_n$ and $\delta_i : Z_{1,i} \to Z_{2,\rho_3i}$ such that
\[ \begin{array}{lcccr} {g_2 = \bigotimes(\rho_3,(\delta_i)_i)g_1} && {\rho_2\rho_3 = \rho_1} && {h_{1,i} = h_{2,\rho_3i}F(\delta_i).} \end{array} \]
With this explicit description in hand it is clear that $\tn{Fact}_{\overline{F}J_{\ca V}}(X,f,(Y_i)_i)$ may be identified as the full subcategory of $\tn{Fact}_{\overline{F}}(X,f,(Y_i)_i)$ consisting of objects of the form $(g,(Z_i)_i,1_n,(h_i)_i)$. Since for any $(g,(Z_i)_i,\rho,(h_i)_i)$ one has an isomorphism
\[ \begin{array}{c} {(\rho,(1_{Z_i})_i) : (g,(Z_i)_i,\rho,(h_i)_i) \longrightarrow (\bigotimes(\rho,(1_{Z_i})_i)g,(Z_{\rho^{-1}i})_i,1_n,(h_{\rho^{-1}i})_i)} \end{array} \]
the inclusion of this subcategory is essentially surjective on objects.
\end{proof}
By Proposition \ref{prop:exactness-for-Sm-and-Bm} and Examples \ref{exam:cartesian-comonoidal->exact}-\ref{exam:exact-comonoids} one has the following examples and counterexamples of exactness for the 2-monad $\SMCMnd$. For each of these there is an evident analogue for the 2-monad $\BMCMnd$.
\begin{exams}\label{exams:exactness-for-Sm-and-Bm}
\begin{enumerate}
\item Symmetric monoidal functors between cartesian monoidal categories are exact as colax $\SMCMnd$-morphisms.
\item When the unit of a symmetric $\ca V$ is not terminal, the unique $\ca V \to 1$ is not exact as a (strict) $\SMCMnd$-algebra morphism.
\item The comonoid $k[x]$ in $\tnb{Vect}_k$ recalled in Example \ref{exam:exact-comonoids} is cocommutative, and a cocommutative comonoid in a symmetric monoidal category $\ca W$ may be identified as a symmetric colax monoidal functor $1 \to \ca W$. Identifying $k[x]$ in this way gives another example of a non-exact colax morphism of pseudo $\SMCMnd$-algebras.
\end{enumerate}
\end{exams}

\subsection{Diexact naturality squares}
\label{ssec:exact-nat}
It turns out that the naturality squares of the units and multiplications of the 2-monads arising in many of the situations of interest for us, are exact in both possible senses. To avoid confusion on this last point we make
\begin{conv}\label{conv:commuting-exactness}
A commutative square in a 2-category $\ca K$ with comma objects, as on the left
\[ \xygraph{!{0;(2.5,0):} {\xybox{\xygraph{{P}="tl" [r] {B}="tr" [d] {C}="br" [l] {A}="bl"
"tl":"tr"^-{q}:"br"^{g}:@{<-}"bl"^-{f}:@{<-}"tl"^{p}}}} [r]
{\xybox{\xygraph{{P}="tl" [r] {B}="tr" [d] {C}="br" [l] {A}="bl"
"tl":"tr"^-{q}:"br"^{g}:@{<-}"bl"^-{f}:@{<-}"tl"^{p} "tl" [d(.5)r(.35)] :@{=>}[r(.3)]^-{\id}}}} [r]
{\xybox{\xygraph{{P}="tl" [r] {A}="tr" [d] {C}="br" [l] {B}="bl"
"tl":"tr"^-{p}:"br"^{f}:@{<-}"bl"^-{g}:@{<-}"tl"^{q}}}} [r]
{\xybox{\xygraph{{P}="tl" [r] {B}="tr" [d] {C}="br" [l] {A}="bl"
"tl":"tr"^-{q}:"br"^{g}:@{<-}"bl"^-{f}:@{<-}"tl"^{p}:@{}"br"|-*{\iso}}}}} \]
is said to be exact when the identity 2-cell second from the left is exact in the sense of Definition \ref{def:ExactSquare}. In general this is different from saying that the commutative square second from the right is exact, this being the exactness of the evident square in which the identity is oriented the other way as $\id : gq \to fp$. We adopt the same convention for squares as on the right in the previous display, which commute up to isomorphism when we wish to avoid naming the isomorphism under consideration.
\end{conv}
\noindent In this section we give general conditions under which diexact 2-natural transformations arise, this applying in particular to the units and multiplications of $\MCMnd$, $\SMCMnd$ and $\BMCMnd$. An interesting consequence of these considerations is that algebraic cocompleteness can be transferred along polynomial adjunctions of 2-monads as we see in Proposition \ref{prop:transfer-algebraic-cocompleteness}.
\begin{defn}\label{def:exact-nat}
Let $S$ and $T : \ca K \to \ca L$ be 2-functors, $\phi : S \to T$ be a 2-natural transformation and suppose that $\ca L$ has comma objects. Then $\phi$ is \emph{exact} when for all $f : A \to B$ in $\ca K$, the square on the left
\[ \xygraph{{\xybox{\xygraph{!{0;(1.25,0):(0,.8)::} {SA}="p0" [r] {SB}="p1" [d] {TB}="p2" [l] {TA}="p3" "p0":"p1"^-{Sf}:"p2"^-{\phi_B}:@{<-}"p3"^-{Tf}:@{<-}"p0"^-{\phi_A}}}}
[r(4.5)]
{\xybox{\xygraph{!{0;(1.25,0):(0,.8)::} {SA}="p0" [r] {TA}="p1" [d] {TB}="p2" [l] {SB}="p3" "p0":"p1"^-{\phi_A}:"p2"^-{Tf}:@{<-}"p3"^-{\phi_B}:@{<-}"p0"^-{Sf}}}}} \]
is exact in $\ca L$, and \emph{diexact} when both the above squares are exact for all $f$.
\end{defn}
\begin{rem}\label{rem:exact-mu}
Let $\ca K$ be a 2-category with comma objects. To say that the multiplication $\mu^T$ for a 2-monad $(\ca K,T)$ is an exact 2-natural transformation in the sense of Definition \ref{def:exact-nat}, is to say that every free $T$-algebra morphism (that is any strict morphism of the form $Tf$) is exact in the sense of Definition \ref{def:exact-T-morphism}.
\end{rem}
For a 2-functor $T:\ca K \to \ca L$ and an object $X \in \ca K$, we denote by $T_X : \ca K/X \to \ca L/TX$ the 2-functor given on objects by applying $T$ to morphisms into $X$. A \emph{local right adjoint} \cite{Weber-Fam2fun} $\ca K \to \ca L$ is a 2-functor $T:\ca K \to \ca L$ equipped with a left adjoint to $T_X$ for all $X \in \ca K$. When $\ca K$ has a terminal object $1$, to exhibit $T$ as a local right adjoint, it suffices to give a left adjoint to $T_1$.
\begin{defn}\label{def:opfamilial}
\cite{Weber-Fam2fun, Weber-PolynomialFunctors}
Suppose that $\ca K$ and $\ca L$ have comma objects and $\ca K$ has a terminal object $1$. 
\begin{enumerate}
\item An \emph{opfamilial 2-functor} $T : \ca K \to \ca L$ is a local right adjoint equipped with $\overline{T}_1 : \ca K \to \Algs {\Psi_{T1}}$ such that $U^{\Psi_{T1}}\overline{T}_1 = T_1$.
\item A 2-natural transformation $\phi : S \to T$ between opfamilial 2-functors is \emph{opfamilial} when its naturality squares are pullbacks, and when for all $X \in \ca K$, $\alpha$'s naturality square with respect to the unique map $t_X : X \to 1$ is a morphism of split opfibrations $(\alpha_X,\alpha_1) : St_X \to Tt_X$.
\item An \emph{opfamilial 2-monad} is a 2-monad whose underlying endo-2-functor, unit and multiplication are opfamilial.
\end{enumerate}
\end{defn}
By Proposition 4.3.3 and Lemma 4.3.5 of \cite{Weber-PolynomialFunctors}, opfamilial 2-functors and 2-natural transformations are those 2-functors and 2-natural transformations which are compatible with the theory of opfibrations. In particular, if $f : A \to B$ has the structure of a split opfibration in $\ca K$, and $S$, $T : \ca K \to \ca L$ and $\phi : S \to T$ are opfamilial, then $Sf$ and $Tf$ have the structure of split opfibrations in $\ca L$, and the naturality square of $\phi$ at $f$ gives a morphism $(\phi_A,\phi_B) : Sf \to Tf$ of split opfibrations. Dually \emph{familial} 2-functors are those that are compatible with the theory of fibrations. By Proposition 7.11 of \cite{Weber-Fam2fun} opfamilial 2-functors $T$ such that $T1$ is groupoidal{\footnotemark{\footnotetext{Recall that an object $X$ of a 2-category is \emph{groupoidal} (resp. \emph{discrete}) when for all $Y \in \ca K$, $\ca K(Y,X)$ is a groupoid (resp. discrete).}} are also familial, and by Theorem 7.12 of \cite{Weber-Fam2fun} such 2-functors are particularly well-behaved since they preserve groupoidal objects, and comma objects up to equivalence.

The result of applying $\PFun{\Cat} : \Polyc{\Cat} \to \TwoCAT$ to
\[ \xygraph{{I}="p0" [r] {E}="p1" [r] {B}="p2" [r] {J}="p3" "p0":@{<-}"p1"^-{s}:"p2"^-{p}:"p3"^-{t}} \]
is opfamilial (resp. familial) when $I$ is discrete, $p$ has the structure of a split fibration (resp. split opfibration), and $t$ has the structure of a split opfibration (resp. split fibration), by \cite{Weber-PolynomialFunctors} Theorem 4.4.5. Note also that the value at $1$ of the associated polynomial functor is just $t : B \to J$ as an object of $\Cat/J$. Theorem 4.4.5 of \cite{Weber-PolynomialFunctors} also provides sufficient conditions on a morphism of polynomial monads in $\Cat$ to give rise to a familial or opfamilial 2-natural transformation. As an application, one finds that the 2-monads $\MCMnd$, $\SMCMnd$ and $\BMCMnd$ are familial and opfamilial, and $\FPMnd$ is opfamilial but not familial. See Section 5 of \cite{Weber-PolynomialFunctors} for further discussion.
\begin{prop}\label{prop:diexact-bipb-natsq}
Let $S$ and $T: \ca A \to \ca B$ be 2-functors between finitely complete 2-categories and $\phi:S \to T$ be a 2-natural transformation between them. If
\begin{enumerate}
\item $T$ is opfamilial,
\item $T1$ is groupoidal, and
\item $\phi$'s naturality squares are pullbacks
\end{enumerate}
then $\phi$'s naturality squares are also bipullbacks and $\phi$ is diexact.
\end{prop}
\begin{proof}
By Theorem 6.2 of \cite{Weber-Fam2fun} $T$ preserves isofibrations. For any $X \in \ca A$ the unique morphism $t_X:X \to 1$ is an isofibration. Since the square on the left
\[ \xygraph{{\xybox{\xygraph{!{0;(1.25,0):(0,.8)::} {SX}="tl" [r] {S1}="tr" [d] {T1}="br" [l] {TX}="bl" "tl":"tr"^-{St_X}:"br"^-{\phi_1}:@{<-}"bl"^-{Tt_X}:@{<-}"tl"^-{\phi_X}}}} [r(4)]
{\xybox{\xygraph{!{0;(1.25,0):(0,.8)::} {SX}="p1" [r] {SY}="p2" [r] {S1}="p3" [d] {T1}="p4" [l] {TY}="p5" [l] {TX}="p6"
"p1":"p2"^-{Sf}:"p3"^-{St_Y}:"p4"^-{\phi_1}:@{<-}"p5"^-{Tt_Y}:@{<-}"p6"^-{Tf}:@{<-}"p1"^-{\phi_X} "p2":"p5"^{\phi_Y}}}}} \]
is a pullback and $Tt_X$ is an isofibration, this square is also a bipullback by \cite{Weber-Fam2fun} Example 3.9. For a general morphism $f:X \to Y$ one uses the cancellability of bipullbacks (see \cite{Weber-Fam2fun} Proposition 3.10) in the context of the diagram on the right in the previous display. Since $T1$ groupoidal, the component $\phi_1 : S1 \to T1$ is a bi-fibration and a bi-opfibration. Since bi-fibrations and bi-opfibrations are stable under bipullback, and the square on the left in the previous display is a bipullback, any component $\phi_X$ of $\phi$ is a bi-fibration and a bi-opfibration. Thus the naturality squares of $\phi$ are exact in both possible senses (of Convention \ref{conv:commuting-exactness}) by Proposition \ref{prop:exact-pullbacks}(\ref{propcase:ex-bipb}).
\end{proof}
In many situations when we wish to apply Theorem \ref{thm:main}, a further condition is satisfied, namely that $B_T$ is a groupoid. In that case by the following corollary, $\mu^T$ is diexact. This applies in particular when $T = \MCMnd$, $\SMCMnd$ or $\BMCMnd$.
\begin{cor}\label{cor:mu^T-diexact-in-main-thm}
If $(\Cat/I,T)$ is a polynomial 2-monad in which
\begin{enumerate}
\item $I$ is discrete,
\item $B_T$ is a groupoid, and
\item $p_T$ has the structure of a split fibration,
\end{enumerate}
then $\mu^T$ is diexact.
\end{cor}
We shall now see that Proposition \ref{prop:diexact-bipb-natsq} can also be used to transfer algebraic cocompleteness across an adjunction of 2-monads in the following way. Recall that an adjunction of 2-monads as on the left
\[ \begin{array}{lccr} {F : (\ca L,S) \longrightarrow (\ca K,T)} &&&
{SF^*A \xrightarrow{F^l_A} F^*TA \xrightarrow{F^*a} F^*A} \end{array} \]
as recalled in Section \ref{ssec:internal-algebras}, includes the data of an underlying adjunction $F_! \ladj F^* : \ca K \to \ca L$, $F^l : SF^* \to F^*T$ and $F^c : F_!S \to TF_!$, with $F^l$ and $F^c$ mates under $F_! \ladj F^*$ and compatible with the 2-monad structures of $S$ and $T$. Recall moreover that given a pseudo $T$-algebra $(A,a)$, the pseudo $S$-algebra $\overline{F}(A,a)$ has underlying object $F^*A$, and action given by the composite on the right in the previous display. Finally recall from Definition \ref{def:algebraic-cocompleteness}, that $(A,a)$ is algebraically complete with respect to $f : I \to J$ in $\ca K$, when $A$ admits all pointwise left extensions along $f$ in $\ca K$, and moreover these are compatible with $A$'s pseudo algebra structure. Proposition \ref{prop:transfer-algebraic-cocompleteness} below, says that if the adjunction of 2-monads $F$ is nice enough, then for $g : K \to L$ in $\ca L$, the algebraic cocompleteness of $A$ with respect to $F_!g$ implies the algebraic cocompleteness of $\overline{F}A$ with respect to $g$.

As recalled in Section \ref{ssec:MainTheorem}, $F : (\Cat/I,S) \to (\Cat/J,T)$ is a polynomial adjunction of 2-monads if it arises from a morphism of polynomial monads over $\Cat$, that is to say, if it the result of applying $\PFun{\Cat} : \Polyc{\Cat} \to \TwoCAT$ to
\[ \xygraph{!{0;(1.5,0):(0,.6667)::} {I}="p0" [r] {E_S}="p1" [r] {B_S}="p2" [r] {I}="p3" [d] {J.}="p4" [l] {B_T}="p5" [l] {E_T}="p6" [l] {J}="p7" "p0":@{<-}"p1"^-{s_S}:"p2"^-{p_S}:"p3"^-{t_S}:"p4"^-{f}:@{<-}"p5"^-{t_T}:@{<-}"p6"^-{p_T}:"p7"^-{s_T}:@{<-}"p0"^-{f} "p1":"p6"_-{F_2} "p2":"p5"^-{F_1} "p1":@{}"p5"|-{\tn{pb}}} \]
\begin{prop}\label{prop:transfer-algebraic-cocompleteness}
Let $F : (\Cat/I,S) \to (\Cat/J,T)$ be a polynomial adjunction of 2-monads, $A$ be a pseudo $T$-algebra, and $g : X \to Y$ be a morphism of $\Cat/I$. Suppose that
\begin{enumerate}
\item $I$ and $J$ are discrete,
\item $B_T$ is a groupoid, and
\item $p_T$ has the structure of a split fibration.
\end{enumerate}
If $A$ is algebraically cocomplete relative to $F_!g$, then $\overline{F}A$ is algebraically cocomplete relative to $g$.
\end{prop}
\begin{proof}
The 2-natural transformation $F^c : F_!S \to TF_!$ is in the image of $\PFun{\Cat}$ and so its naturality squares are pullbacks. By Theorem 4.4.5 of \cite{Weber-OpPoly2Mnd}, $T$ and $F_! = \Sigma_f$ are opfamilial, and thus so is $TF_!$. Moreover $TF_!1$ is the result of applying $T$ to $F_!1 = f \in \Cat/J$, which is discrete in $\Cat/J$ since $I$ is discrete. Since $B_T$ is a groupoid, $T1 = t_T \in \Cat/J$ is groupoidal, thus $T$ preserves groupoidal objects, and so $TF_!1$ is groupoidal. Thus by Proposition \ref{prop:diexact-bipb-natsq}, the naturality squares of $F^c$ are diexact bipullback squares.

To see that the underlying object $F^*A$ in $\Cat/I$ of $\overline{F}A$ admits all left Kan extensions along $g$, we consider $h : X \to F^*A$ and denote by $\overline{h} : F_!X \to A$ its mate under $F_! \ladj F^*$. Then mateship gives a bijection between 2-cells $\phi$ and $\overline{\phi}$ as in
\[ \xygraph{{\xybox{\xygraph{{X}="p0" [r(2)] {Y}="p1" [dl] {F^*A}="p2" "p0":"p1"^-{g}:"p2"^-{k}:@{<-}"p0"^-{h} "p0" [d(.5)r(.85)] :@{=>}[r(.3)]^-{\phi}}}}
[r(4)]
{\xybox{\xygraph{{F_!X}="p0" [r(2)] {F_!Y}="p1" [dl] {A}="p2" "p0":"p1"^-{F_!g}:"p2"^-{\overline{k}}:@{<-}"p0"^-{\overline{h}} "p0" [d(.5)r(.85)] :@{=>}[r(.3)]^-{\overline{\phi}}}}}} \]
and under this bijection, $\phi$ exhibits $k$ as a left Kan extension of $h$ along $g$ iff $\overline{\phi}$ exhibits $\overline{k}$ as a left Kan extension of $\overline{h}$ along $F_!g$. In our situation $F_! = \Sigma_f$ preserves comma objects, and so by the proof of Theorem 7.4 of \cite{Weber-2Toposes}, $\phi$ is a pointwise left Kan extension whenever $\overline{\phi}$ is. Thus the pointwise left Kan extension of $h$ along $g$ is computed by taking the pointwise left Kan extension of $\overline{h}$ along $F_!g$, which exists by the hypothesis on $A$, and then taking its mate in the manner just described.

We must now verify that pointwise left Kan extensions into $F^*A$ constructed in this way are compatible, in the sense of Definition \ref{def:algebraic-cocompleteness}, with $\overline{F}A$'s pseudo $S$-algebra structure. Since $F^l$ and $F^c$ are mates via $F_! \ladj F^*$, it is straight forward to verify that
\[ \xygraph{{\xybox{\xygraph{{SX}="p0" [r(2)] {SY}="p1" [dl] {SF^*A}="p2" [d] {F^*A}="p3" "p0":"p1"^-{g}:"p2"^-{k}:@{<-}"p0"^-{h} "p2":"p3"^-{F^*(a)F^l_A}
"p0" [d(.5)r(.85)] :@{=>}[r(.3)]^-{S\phi}}}}
[r(4)]
{\xybox{\xygraph{{TF_!X}="p0" [r(2)] {TF_!Y}="p1" [dl] {TA}="p2" [d] {A}="p3" "p0":"p1"^-{TF_!g}:"p2"^-{T\overline{k}}:@{<-}"p0"^-{T\overline{h}} "p2":"p3"^-{a}
"p0" :@{<-}[u] {F_!SX}^-{F^c_X} :[r(2)] {F_!SY}^-{F_!Sg} :"p1"^-{F^c_Y}
"p0" [d(.5)r(.85)] :@{=>}[r(.3)]^-{T\overline{\phi}}}}}} \]
are mates via $F_! \ladj F^*$. The composite on the right is a pointwise left Kan extension since $A$ is algebraically cocomplete and the naturality square for $F^c$ at the top is exact. Thus its mate, the composite on the left, is a pointwise left Kan extension.
\end{proof}
\begin{exams}\label{exams:alg-cocompleteness-V_bullet}
As recalled in Section \ref{ssec:MainTheorem}, and explained in \cite{Weber-OpPoly2Mnd}, an operad $T$ with set of colours $I$ determines an adjunction $\tn{Ar}_T : (\Cat/I,T) \to (\Cat,\SMCMnd)$ of 2-monads. Given a symmetric monoidal category $\ca V$, $\tn{Ar}_T^*\ca V \in \Cat/I$ is, as an $I$-indexed family of categories, constant at $\ca V$, and the $T$-algebra structure of $\overline{Ar}_T\ca V$ is obtained by using the symmetric monoidal structure of $\ca V$. This was described explicitly in Example 4.6 of \cite{Weber-OpPoly2Mnd}, where $\overline{Ar}_T\ca V$ was denoted as $\ca V_{\bullet}$. By Propositions \ref{prop:algcocomp-Sm-or-Bm} and \ref{prop:transfer-algebraic-cocompleteness}, if $\ca V$ is cocomplete and its tensor product preserves colimits in each variable, then $\ca V_{\bullet}$ is algebraically cocomplete as a pseudo $T$-algebra, with respect to all functors over $I$ between small categories. The same is true more generally when $\tn{Ar}_T$ is replaced by any polynomial adjunction of 2-monads $(\Cat/I,T) \to (\Cat,\SMCMnd)$ with $I$ discrete, and thus applies also to the case of $\Cat$-operads. Replacing $\SMCMnd$ by $\MCMnd$ or $\BMCMnd$, one obtains the analogous results for non-symmetric and braided operads.
\end{exams}

\section{Exact squares via codescent}
\label{sec:exact-via-codescent}

In this section and the next we explain why, in the context of Theorem \ref{thm:main}, $T^G : T^R \to T^S$ is exact. The internal algebra classifiers $T^R$ and $T^S$ are computed as codescent objects of crossed internal categories by \cite{Weber-CodescCrIntCat}. We recall the crossed double categories of \cite{Weber-CodescCrIntCat} and their codescent objects in Section \ref{ssec:codesc-crossed-dblcat}. The main result of this section is Theorem \ref{thm:hard-exactness}, which is the combinatorial reason for $T^G$'s exactness. This result is outside of any monad-theoretic context, and provides conditions on a pullback square $\ca S$ of crossed double categories which ensure that the process of taking codescent sends $\ca S$ to an exact square. Then in Sections \ref{ssec:int-alg-class-codesc} and \ref{ssec:alg-square-before-codescent} we identify the monad-theoretic context to which Theorem \ref{thm:hard-exactness} may be applied to give the proof of Theorem \ref{thm:main}.

Sections \ref{ssec:pi0-exactness}-\ref{ssec:exactness-of-internalisations} are concerned with the proof of Theorem \ref{thm:hard-exactness}. In Section \ref{ssec:pi0-exactness} the notion of $\pi_0$-exact square of 2-categories is defined, and the conclusion of Theorem \ref{thm:hard-exactness} is reformulated in such terms. In Corollary \ref{cor:pi0-exact-via-2-coends} the property of $\pi_0$-exactness is described more explicitly. This involves a 2-coend of a certain type, and in Section \ref{ssec:lax-coends}, this 2-coend is replaced by a lax coend which is easier to analyse. From this point, we become interested in how to compute such lax coends. At the end of Section \ref{ssec:lax-coends}, the weight for lax coends is described in terms of a 2-categorical weight $\ca H$, and the relevant $\ca H$-weighted colimits are then computed in Section \ref{ssec:lax-wedges}. In Corollary \ref{cor:lax-coend-formula} we obtain an explicit formula for the lax coends of interest using these results. Corollaries \ref{cor:pi0-exact-via-2-coends} and \ref{cor:lax-coend-formula} together result in a combinatorial characterisation of $\pi_0$-exact squares in Proposition \ref{prop:pi0-exact-combinatorial-characterisation}, which generalises Guitart's explicit characterisation \cite{Guitart-ExactSquares} of exact squares in $\Cat$. This is then applied in the proof of Theorem \ref{thm:hard-exactness} in Section \ref{ssec:exactness-of-internalisations}.

In Section \ref{ssec:int-alg-class-codesc} we describe $T^G$ as the result of taking codescent of a morphism of simplicial $T$-algebras which we describe in Construction \ref{const:simplicial-morphism-underlying-TG}. Then in Section \ref{ssec:alg-square-before-codescent}, we finally exhibit $T^G$'s exactness by applying Theorem \ref{thm:hard-exactness} to this situation.

\subsection{Codescent for crossed double categories}
\label{ssec:codesc-crossed-dblcat}
Denoting by $\delta : \Delta \to \Cat$ the inclusion obtained by regarding non-empty ordinals $[n] = \{0 < ...< n\}$ as categories $0 \to ... \to n$, the \emph{codescent object} of a simplicial object $X : \Delta^{\op} \to \ca K$ in a 2-category $\ca K$ is defined to be the colimit of $X$ weighted by $\delta$ in the sense of $\Cat$-enriched category theory \cite{Kelly-EnrichedCatsBook}. For this type of colimit the corresponding notion of \emph{cocone} for $X$ with vertex $Z$ amounts to a pair $(f_0,f_1)$, where $f_0:X_0 \to Z$ and $f_1:f_0d_1 \to f_0d_0$ are in $\ca K$, and satisfy $f_1s_0=1_{f_0}$ and $(f_1d_0)(f_1d_2)=f_1d_1$. As such, a codescent object for $X$ consists of an object $\tn{CoDesc}(X)$ of $\ca K$ and a cocone $(q_0,q_1)$ with vertex $\tn{CoDesc}(X)$ universal in the evident sense recalled in detail in Section 4.2 of \cite{Weber-CodescCrIntCat}.

Recall that a simplicial object $X : \Delta^{\op} \to \ca K$ in $\ca K$ is an \emph{internal category} when for all $n \in \N$ the square
\[ \xygraph{!{0;(1.5,0):(0,.6667)::} {X_{n+2}}="tl" [r] {X_{n+1}}="tr" [d] {X_n}="br" [l] {X_{n+1}}="bl" "tl":"tr"^-{d_{n+2}}:"br"^-{d_0}:@{<-}"bl"^-{d_{n+1}}:@{<-}"tl"^-{d_0}} \]
is a pullback. A category object $X$ in $\Cat$ is commonly known as a \emph{double category}. Following the conventions of \cite{Weber-CodescCrIntCat}, in elementary terms such an $X$ consists of (1) \emph{objects} -- which are the objects of $X_0$, (2) \emph{vertical arrows} -- which are the arrows of $X_0$, (3) \emph{horizontal arrows} -- which are the objects of $X_1$, and (4) \emph{squares} -- which are the arrows of $X_1$. In addition one has compositions of vertical arrows, compositions of horizontal arrows, and both vertical and horizontal composition of squares. Given categories $X$ and $Y$ internal to $\ca K$, an \emph{internal functor} between them is a morphism $f : X \to Y$ in $[\Delta^{\op},\ca K]$. When $\ca K$ is $\Cat$ these are usually referred to as \emph{double functors} involving assignations at the level of objects, vertical arrows, horizontal arrows and squares, compatible in the evident way with the compositions listed above.

We recall now the notions of crossed internal category and crossed internal functor from \cite{Weber-CodescCrIntCat}. When $\ca K$ has comma objects and pullbacks, a \emph{crossed internal category} is an internal category $X : \Delta^{\op} \to \ca K$, together with the structure of a split opfibration on $d_0:X_1 \to X_0$ such that 
\[ \xygraph{!{0;(1.5,0):(0,.6667)::} {X_0}="p1" [r] {X_1}="p2" [r] {X_2}="p3" [dl] {X_0}="p4" "p1":"p2"^-{s_0}:@{<-}"p3"^-{d_1}:"p4"^-{d_0^2}:@{<-}"p1"^-{1} "p2":"p4"^{d_0}} \]
are morphisms of split opfibrations over $X_0$. When $\ca K = \Cat$ such an $X$ is called a \emph{crossed double category}. The main extra structure one has in a crossed double category $X$ is that for each pair of arrows $(h,v)$ as on the left
\[ \xygraph{{\xybox{\xygraph{{x}="p0" [r] {y}="p1" [d] {z}="p2" "p0":"p1"^-{h}:"p2"^-{v}}}}
[r(3)]
{\xybox{\xygraph{{x}="p0" [r] {y}="p1" [d] {z}="p2" [l] {w}="p3" "p0":"p1"^-{h}:"p2"^-{v}:@{<-}"p3"^-{\rho_{h,v}}:@{<-}"p0"^-{\lambda_{h,v}}:@{}"p2"|-*{\kappa_{h,v}}}}}} \]
one has distinguished squares, called \emph{chosen opcartesian squares} as on the right satisfying a universal property, coming from the fact that these squares are opcartesian arrows for the opfibration $d_0 : X_1 \to X_0$, and these squares are closed under vertical and horizontal composition. A full unpacking of this notion in the case $\ca K = \Cat$ is given in Section 5.1 of \cite{Weber-CodescCrIntCat}.

Let $X$ and $Y$ be crossed internal categories in a finitely complete 2-category $\ca K$. A \emph{crossed internal functor} $f:X \to Y$ is an internal functor such that the square
\[ \xygraph{!{0;(1.5,0):(0,.6667)::} {X_1}="tl" [r] {X_0}="tr" [d] {Y_0}="br" [l] {Y_1}="bl" "tl":"tr"^-{d_0}:"br"^-{f_0}:@{<-}"bl"^-{d_0}:@{<-}"tl"^-{f_1}} \]
is a morphism from $d_0:X_1 \to X_0$ to $d_0:Y_1 \to Y_0$ of split opfibrations. When $\ca K = \Cat$ we shall say that $f$ is a \emph{crossed double functor}. In elementary terms, crossed double functors are double functors that preserve chosen opcartesian squares. We denote by $\CrIntCat {\ca K}$ the category of crossed internal categories in $\ca K$ and crossed internal functors between them.

In \cite{Weber-CodescCrIntCat} the computation of codescent objects of crossed double categories was understood, and we recall the relevant details now. For $X$ a crossed double category, one computes $\tn{CoDesc}(X)$ by first constructing a 2-category $\tn{Cnr}(X)$, and then applying $\pi_0$ to the homs of this 2-category. The 2-category $\tn{Cnr}(X)$ is defined in elementary terms as follows. An object is an object of the double category $X$. An arrow $x \to y$ is a pair $(f,g)$ where $f$ is a vertical arrow and $g$ is a horizontal arrow as on the left in
\[ \xygraph{{\xybox{\xygraph{{x}="l" [d] {z}="m" [r] {y}="r" "l":"m"^-{f}:"r"^-{g}}}}
[r(3)]
{\xybox{\xygraph{{x}="ttl" [d] {a}="tl" [r] {y}="tr" [d] {b}="br" [l] {c}="bl" [r(2)] {z}="brr"
"ttl":"tl"_{f}:"tr"^-{g}:"br"^-{h}:@{<-}"bl"^-{\rho_{g,h}}:@{<-}"tl"^-{\lambda_{g,h}}:@{}"br"|-*{\kappa_{g,h}}:"brr"_-{k}}}}
[r(3)]
{\xybox{\xygraph{{x}="ttl" [d] {z_1}="tl" [r] {y}="tr" [d] {y}="br" [l] {z_2}="bl" "ttl":"tl"_-{f}:"tr"^-{g}:"br"^-{1_y}:@{<-}"bl"^-{k}:@{<-}"tl"^-{\alpha} "tl":@{}"br"|-*{\beta}}}}} \]
and is called a \emph{corner} from $x$ to $y$. Such an $(f,g)$ is an identity in $\Cnr(X)$ when $f$ and $g$ are identities. The composite of $(f,g):x \to y$ and $(h,k):y \to z$ in $\Cnr(X)$ is defined to be $(\lambda_{g,h}f,k\rho_{g,h})$ as in the middle of the previous display. Given $(f,g)$ and $(h,k):x \to y$, a 2-cell $(f,g) \to (h,k)$ in $\Cnr(X)$ is a pair $(\alpha,\beta)$ where $\alpha$ is a vertical arrow and $\beta$ is a square as on the right in the previous display, such that $\alpha f=h$. Vertical composition of 2-cells in $\Cnr(X)$ is given in the evident manner by vertical composition in $X$. The assignation $X \mapsto \Cnr(X)$ is the object map of a functor as on the left in
\[ \begin{array}{lccr} {\Cnr : \CrIntCat{\Cat} \longrightarrow \TwoCat} &&& {\pi_{0*} : \TwoCat \to \Cat} \end{array} \]
and the functor on the right applies $\pi_0$ to the homs of a 2-category (leaving the objects fixed). By Corollary 5.4.5 of \cite{Weber-CodescCrIntCat} one has
\begin{thm}\label{thm:codesc-crossed-int-cat}
\cite{Weber-CodescCrIntCat} The functor
\[ \CoDesc : \CrIntCat{\Cat} \longrightarrow \Cat \]
factors as $\CoDesc = \pi_{0*}\Cnr$.
\end{thm}
In Theorem \ref{thm:hard-exactness}, the central result of this section, conditions on a pullback square in $\CrIntCat{\Cat}$ are exhibited, which ensure that $\CoDesc$ sends it to an exact square in $\Cat$. To understand what pullbacks in $\CrIntCat{\Cat}$ amount to, we have
\begin{lem}\label{lem:pbs-CrIntCat}
The functor
\[ \begin{array}{lccr} {\CrIntCat{\Cat} \longrightarrow \Set \times \Set \times \Set \times \Set} &&& {X \mapsto (X_{00},X_{01},X_{10},X_{11})} \end{array} \]
where $X_{00}$, $X_{01}$, $X_{10}$ and $X_{11}$ are the sets of objects, vertical arrows, horizontal arrows and squares of $X$ respectively; preserves and reflects limits.
\end{lem}
\begin{proof}
As a sub-2-category of $\tn{Cat}(\ca K)$, $\CrIntCat{\ca K}$ has any limits that $\ca K$ has, and these are computed componentwise, since split opfibrations and their morphisms are expressed internally to $\ca K$ using limits. Thus the inclusion of categories $\CrIntCat{\ca K} \hookrightarrow [\Delta^{\op},\ca K]$ preserves and reflects limits. When $\ca K = \Cat$, using the nerve functor one has an inclusion of $[\Delta^{\op},\ca K]$ into the category of bisimplicial sets, which preserves and reflects limits. Since for a double category $X$, the sets $X_{mn}$ participating in its associated bisimplicial set can be reconstructed as limits of $X_{00}$, $X_{01}$, $X_{10}$ and $X_{11}$, the result follows.
\end{proof}
\begin{defn}\label{defn:extra-conditions-TG-alg-square}
Let $\ca K$ be a 2-category with comma objects and pullbacks, $X,Y:\Delta^{\op} \to \ca K$ be category objects, and $f : X \to Y$ be an internal functor.
\begin{enumerate}
\item $f$ is a \emph{discrete fibration} when the square
\[ \xygraph{{X_1}="p0" [r] {Y_1}="p1" [d] {Y_0}="p2" [l] {X_0}="p3" "p0":"p1"^-{f_1}:"p2"^-{d_0}:@{<-}"p3"^-{f_0}:@{<-}"p0"^-{d_0}} \]
is a pullback.
\item $f$ is an \emph{objectwise opfibration} when the morphism $f_0 : X_0 \to Y_0$ is an opfibration.
\end{enumerate}
\end{defn}
\begin{thm}\label{thm:hard-exactness}
Suppose that $\ca S$ is a pullback square
\[ \xygraph{{P}="p0" [r] {B}="p1" [d] {C}="p2" [l] {A}="p3" "p0":"p1"^-{}:"p2"^-{g}:@{<-}"p3"^-{f}:@{<-}"p0"^-{}:@{}"p2"|-{\tn{pb}}} \]
in $\CrIntCat {\Cat}$ in which $g$ is a discrete fibration and $f$ is an objectwise opfibration. Then $\CoDesc(\ca S)$ is exact.
\end{thm}
The proof of Theorem \ref{thm:hard-exactness} occupies the rest of Section \ref{sec:exact-via-codescent}. Notice that the hypotheses on $\ca S$ are a double categorical mixture of both types of hypotheses resulting in exact pullback squares in Proposition \ref{prop:exact-pullbacks}(\ref{propcase:ex-pb}).

These hypotheses on $\ca S$ are combinatorial. In elementary double categorical terms, the condition that $g$ be a discrete fibration amounts to the following two conditions. First, the underlying functor between the underlying horizontal categories is a discrete fibration, that is to say, given $b \in B$ and a horizontal arrow $h:c \to gb$ in $C$, there is a unique horizontal arrow $k:b_2 \to b$ in $B$ such that $gk = h$. Second, one has a unique horizontal lifting property for squares. This says that for any square in $C$ as on the left in
\[ \xygraph{{\xybox{\xygraph{{c_1}="p0" [r] {gb_1}="p1" [d] {gb_2}="p2" [l] {c_2}="p3" "p0":"p1"^-{h_1}:"p2"^-{gv_2}:@{<-}"p3"^-{h_2}:@{<-}"p0"^-{v_1}:@{}"p2"|-*{\alpha}}}}
[r(3)]
{\xybox{\xygraph{{b_3}="p0" [r] {b_1}="p1" [d] {b_2}="p2" [l] {b_4}="p3" "p0":"p1"^-{k_1}:"p2"^-{v_2}:@{<-}"p3"^-{k_2}:@{<-}"p0"^-{v_3}:@{}"p2"|-*{\beta}}}}} \]
there is a unique square in $B$ as on the right such that $g\beta = \alpha$. In particular in this last situation, $k_1$ and $k_2$ are the unique lifts of $h_1$ and $h_2$. On the other hand, the hypothesis on $f$ is about lifting vertical arrows.

\subsection{$\pi_0$-exactness}
\label{ssec:pi0-exactness}
As we saw in Lemma \ref{lem:combinatorial-exactness-Cat}, one way to express combinatorially what exactness amounts in $\Cat$, is that a lax square as on the left
\[ \xygraph{{\xybox{\xygraph{!{0;(1.5,0):(0,.6667)::} {P}="p0" [r] {B}="p1" [d] {C}="p2" [l] {A}="p3" "p0":"p1"^-{q}:"p2"^-{g}:@{<-}"p3"^-{f}:@{<-}"p0"^-{p} "p0" [d(.55)r(.4)] :@{=>}[r(.2)]^-{\phi}}}}
[r(5.5)]
{\xybox{\xygraph{{\overline{\phi}_{a,b} : (q \downarrow b) \times_P (a \downarrow p) \longrightarrow C(fa,gb)} [d(.4)] {(x,\beta:qx \to b,\alpha:a \to px) \mapsto g\beta \comp \phi_x \comp f\alpha}}}}} \]
is exact iff for all $a \in A$ and $b \in B$, the functor $\pi_0:\Cat \to \Set$ inverts the functor $\overline{\phi}_{a,b}$ described on the right in the previous display. In the context of Theorem \ref{thm:hard-exactness}, $P$, $A$, $B$ and $C$ are all of the form $\tn{CoDesc}(X)$ for some crossed double category $X$.
\begin{defn}\label{def:pi0-exact}
A lax square $\ca S$
\[ \xygraph{!{0;(1.5,0):(0,.6667)::} {P}="p0" [r] {B}="p1" [d] {C}="p2" [l] {A}="p3" "p0":"p1"^-{q}:"p2"^-{g}:@{<-}"p3"^-{f}:@{<-}"p0"^-{p} "p0" [d(.55)r(.4)] :@{=>}[r(.2)]^-{\phi}} \]
in $\TwoCat$ is \emph{$\pi_0$-exact} when $\pi_{0*}(\ca S)$ is an exact square in $\Cat$.
\end{defn}
\begin{rem}\label{rem:exactness-via-Cnr}
By the way in which codescent objects of crossed double categories are computed, as recalled in Theorem \ref{thm:codesc-crossed-int-cat}, to say that $\CoDesc(\ca S)$ is exact in the context of Theorem \ref{thm:hard-exactness}, is to say that $\Cnr(\ca S)$ is $\pi_0$-exact.
\end{rem}
\noindent We shall achieve an explicit characterisation of $\pi_0$-exact squares of 2-categories in Proposition \ref{prop:pi0-exact-combinatorial-characterisation} below, from which Theorem \ref{thm:hard-exactness} will follow. In this section we obtain the analogue of Lemma \ref{lem:combinatorial-exactness-Cat} for $\pi_0$-exact squares, in Corollary \ref{cor:pi0-exact-via-2-coends}.

As mentioned in Section \ref{sec:exact-in-Cat}, Definition \ref{def:Guitart-exactness-in-Cat} of an exact square in $\Cat$ can be generalised to the setting of $\ca V$-categories, where $\ca V$ is nice enough (symmetric monoidal closed cocomplete), and further still to the setting of proarrow equipments in the sense of \cite{Wood-Proarrows-I, Wood-Proarrows-II} as in \cite{Koudenburg-Thesis, MelliesTabareau-TAlgTheoriesKan}. That is, a lax square in $\Enrich {\ca V}$ as on the left in
\[ \xygraph{{\xybox{\xygraph{!{0;(1.5,0):(0,.6667)::} {P}="p0" [r] {B}="p1" [d] {C}="p2" [l] {A}="p3" "p0":"p1"^-{q}:"p2"^-{g}:@{<-}"p3"^-{f}:@{<-}"p0"^-{p} "p0" [d(.55)r(.4)] :@{=>}[r(.2)]^-{\phi}}}}
[r(4)]
{\xybox{\xygraph{!{0;(1.5,0):(0,.6667)::} {P}="p0" [r] {B}="p1" [d] {C}="p2" [l] {A}="p3" "p0":"p1"^-{B(q,1)}:@{<-}"p2"^-{C(1,g)}:@{<-}"p3"^-{C(f,1)}:"p0"^-{A(1,p)} "p0" [d(.55)r(.4)] :@{=>}[r(.2)]^-{\tilde{\phi}}}}}} \]
is \emph{$\ca V$-exact} when the induced 2-cell $\tilde{\phi}$ on the right in $\Mod {\ca V}$ is invertible. The cases of interest for us is when $\ca V$ is either $\Set$ or $\Cat$, in both cases with the cartesian monoidal structure, and in these cases the bicategories $\Mod {\ca V}$ are denoted $\Prof$ (as above) and $\TwoProf$ respectively. In the case $\ca V = \Cat$ this notion of exactness is, as we shall see, more restrictive than the $\pi_0$-exactness of Definition \ref{def:pi0-exact}.

Given a 2-profunctor $F:A \to B$ between 2-categories, we define the profunctor $\pi_{0*}F:\pi_{0*}A \to \pi_{0*}B$ by $(\pi_{0*}F)(a,b) = \pi_0(F(a,b))$. Clearly given a 2-functor $f:A \to B$, one has
\[ \begin{array}{lccr} {\pi_{0*}(B(f,1)) = (\pi_{0*}B)(\pi_{0*}f,1)} &&&
{\pi_{0*}(B(1,f)) = (\pi_{0*}B)(1,\pi_{0*}f).} \end{array} \]
With the obvious extension of $\pi_{0*}$ to 2-cells in $\TwoProf$ we have extended the definition of $\pi_{0*}$ to the level of modules, modulo verifying that this extension is compatible with module composition.

To establish this compatibility some remarks regarding 2-coends are in order. Given a 2-functor $T:A^{\op} \times A \to B$ the 2-coend $\int^{a \in A} T(a,a)$ is by definition \cite{Kelly-EnrichedCatsBook} the weighted colimit $\tn{col}(\tn{Hom}_{A^{\op}},T)$, which is to say that it is defined by isomorphisms
\[ B(\int^{a \in A} T(a,a),b) \iso [A \times A^{\op}](\tn{Hom}_{A^{\op}},B(T,b)) \]
2-natural in $b$. Thus a 2-dinatural transformation for $T$ with vertex $b$ is by definition a 2-natural transformation $\phi:\tn{Hom}_{A^{\op}} \to B(T,b)$, which amounts to giving components $\phi_a : T(a,a) \to b$ satisfying the dinaturality condition familiar from the 1-dimensional notion of ``dinatural transformation'', together with a 2-dimensional condition which says that given $\alpha_1,\alpha_2 : a_2 \to a_1$ and $\beta:\alpha_1 \to \alpha_2$ in $A$, one has $\phi_{a_1}T(1,\beta) = \phi_{a_2}T(\beta,1)$. For the sake of the following lemma, we denote by $\tn{obj}$ the functor $\Cat \to \Set$ which sends a category to its set of objects, and by $\tn{obj}_* : \TwoCat \to \Cat$ the 2-functor which sends a 2-category $X$ to its underlying category, which amounts to applying $\tn{obj}$ to the homs of $X$.
\begin{lem}\label{lem:coend-obs-for-pi0star}
Suppose that $T:A^{\op} \times A \to \Cat$ is a 2-functor such that for all $a,b \in A$, $T(a,b)$ is in fact discrete. Then
\[ \int^{a \in A} T(a,a) = \int^{a \in \tn{obj}_*A} T(a,a) = \int^{a \in \pi_{0*}A} T(a,a). \]
\end{lem}
\begin{proof}
Since $\Cat$ admits all cotensors, it suffices to show that the types of dinatural transformations defining each of these coends turn out to be the same in this case. A $1$-dinatural transformation for $T$ is automatically a $2$-dinatural transformation since $T$ is discretely valued, and so one has the first equality. For any 2-cell $\alpha:f \to g$ in $A$, $T(\alpha,1)$ and $T(1,\alpha)$ will be identities, and so $T(f,1) = T(g,1)$ and $T(1,f) = T(1,g)$, whence the second equality.
\end{proof}
Note that the adjunction $\pi_0 \ladj d$ is in fact a 2-adjunction, when $\Set$ is regarded as a locally discrete 2-category. Hence given $2$-profunctors $F:A \to B$ and $G:B \to C$ one has canonical natural isomorphisms
\[ \begin{array}{rcl} {\pi_{0*}(G \comp F)(a,c)}
& {\iso} & {\displaystyle \int^{b \in B} \pi_0G(b,c) \times \pi_0F(a,b)} \\
& {\iso} & {\displaystyle \int^{b \in \pi_{0*}B} \pi_0G(b,c) \times \pi_0F(a,b)} \\
& {=} & {(\pi_{0*}G \comp \pi_{0*}F)(a,c)} \end{array} \]
the first of which follows since $\pi_0$ preserves 2-colimits and finite products, and the second follows by Lemma \ref{lem:coend-obs-for-pi0star}. Thus we have shown
\begin{cor}\label{cor:extend-pi0-star}
The 2-functor $\pi_{0*}:\TwoCat \to \Cat$ extends to a homomorphism of bicategories $\pi_{0*}:\TwoProf \to \Prof$ compatibly with the inclusions, that is, such that
\[ \xygraph{{\xybox{\xygraph{!{0;(2,0):(0,.5)::} {\TwoCat^{\co}}="p0" [r] {\Cat^{\co}}="p1" [d] {\Prof}="p2" [l] {\TwoProf}="p3" "p0":"p1"^-{\pi_{0*}^{\co}}:"p2"^-{}:@{<-}"p3"^-{\pi_{0*}}:@{<-}"p0"^-{}:@{}|-{=}"p2"}}}
[r(5)]
{\xybox{\xygraph{!{0;(2,0):(0,.5)::} {\TwoCat^{\op}}="p0" [r] {\Cat^{\op}}="p1" [d] {\Prof}="p2" [l] {\TwoProf}="p3" "p0":"p1"^-{\pi_{0*}^{\op}}:"p2"^-{}:@{<-}"p3"^-{\pi_{0*}}:@{<-}"p0"^-{}:@{}|-{=}"p2"}}}} \]
\end{cor}
Thus in particular $\pi_{0*}$ sends $\Cat$-exact squares to exact squares in $\Cat$, but more importantly, one has the immediate
\begin{cor}\label{cor:pi0-exact-via-2-coends}
A lax square
\[ \xygraph{!{0;(1.5,0):(0,.6667)::} {P}="p0" [r] {B}="p1" [d] {C}="p2" [l] {A}="p3" "p0":"p1"^-{q}:"p2"^-{g}:@{<-}"p3"^-{f}:@{<-}"p0"^-{p} "p0" [d(.55)r(.4)] :@{=>}[r(.2)]^-{\phi}} \]
in $\TwoCat$ is $\pi_0$-exact iff for all $a \in A$ and $b \in B$, the functor $\pi_0:\Cat \to \Set$ inverts the functor
\[ \tilde{\phi}_{a,b} : \int^{x \in P} B(qx,b) \times A(a,px) \longrightarrow C(fa,gb) \]
defined by composition with the components of $\phi$.
\end{cor}

\subsection{Lax coends}
\label{ssec:lax-coends}
As with many other colimits in $\Cat$, 2-coends such as that which describes the domain of $\tilde{\phi}_{a,b}$, are difficult to compute in general. This is because it is possible that some quotienting will occur at the level of objects, causing ``new'' composable sequences to arise in the colimit, which then make it hard to keep track of all of the freely added composites which must then also appear. However, since by Corollary \ref{cor:pi0-exact-via-2-coends} we are only concerned with the value of such 2-coends ``up to functors inverted by $\pi_0$'', it turns out that for the purposes of characterising $\pi_0$-exact squares, it suffices to consider ``lax coends'' which turn out to be a lot easier to compute. In this section we describe these lax coends, and explain why knowing them is sufficient for our purposes. Moreover we compute the weight governing lax coends as a codescent object, which in Section \ref{ssec:lax-wedges} will enable us to compute the lax coends of interest to us.

For the remainder of this section and the next, let
\[ \begin{array}{lccr} {S:P^{\op} \to \Cat} &&& {T:P \to \Cat} \end{array} \]
be 2-functors and denote by $S \times T$ the 2-functor whose effect on objects is given by $(x,y) \mapsto Sx \times Ty$. For the sake of brevity denote by $H : P \times P^{\op} \to \Cat$ the 2-functor we denoted above as $\tn{Hom}_{P^{\op}}$, whose effect on objects is $(x,y) \mapsto P(y,x)$. We wish to understand
\[ \begin{array}{rcl} {\displaystyle \int^{x \in P} Sx \times Tx} & = & {\tn{col}(H,S \times T)} \end{array} \]
up to a functor inverted by $\pi_0$. By definition the weight $H$ is an object of $[P \times P^{\op},\Cat]$. We shall regard this 2-category as the 2-category of strict algebras and strict morphisms of a 2-monad $L$ to be defined below, so that one can consider another weight $H^{\dagger}_L$, where $(-)^{\dagger}_L$ is the left adjoint to the inclusion $J_L : \Algs L \to \Algl L$.

The 2-monad $L$ is essentially a special case of that described in Section 6.6 of \cite{BWellKellyPower-2DMndThy}. Regard the set $\tn{ob}(P)$ of objects of $P$ as a discrete 2-category, and then left extension and restriction along the inclusion $\tn{ob}(P) \to P$ gives a 2-monad on $[\tn{ob}(P), [P^{\op},\Cat]]$. The 2-category of strict algebras and strict maps may be identified with $[P, [P^{\op},\Cat]]$, and lax morphisms $F \to G$ may be identified with lax natural transformations $F \to G$. Our 2-monad $L$ is exactly this, except that we regard the underlying 2-category and the 2-category of strict algebras and strict maps as
\[ \begin{array}{lccr} {[\tn{ob}(P) \times P^{\op},\Cat]} &&& {[P \times P^{\op},\Cat]} \end{array} \]
respectively. In these terms lax morphisms $F \to G$ may be identified with lax natural transformations $F \to G$ which are strictly natural in the second variable.

Explicitly $L$ is given on objects by
\[ LX(x,y) = \coprod_{z \in \tn{ob}(P)} X(z,y) \times P(z,x) \]
and following \cite{BWellKellyPower-2DMndThy} one may exhibit the rest of the monad structure. Since coproducts in $\Cat$ commute with connected limits and $\Cat$ is cartesian closed, the above formula exhibits $L$ as connected limit preserving. Since in $\Cat$ coproducts of pullback squares are pullbacks, and squares of the form
\[ \xygraph{!{0;(2,0):(0,.5)::} {A \times B}="p0" [r] {A \times D}="p1" [d] {C \times D}="p2" [l] {C \times B}="p3" "p0":"p1"^-{1_A \times g}:"p2"^-{f \times 1_D}:@{<-}"p3"^-{1_C \times g}:@{<-}"p0"^-{f \times 1_B}} \]
are pullbacks, the unit and multiplication of $L$ may be exhibited as cartesian. Thus $L$ is a cartesian 2-monad. Moreover note that since $[P \times P^{\op},\Cat]$ is cocomplete, the codescent objects necessary for the description of $(-)^{\dagger}_L$ exist and we can make
\begin{defn}\label{def:lax-coend}
Let $\ca K$ be a 2-category and $F:P^{\op} \times P \to \ca K$ be a 2-functor. Then the colimit of $F$ weighted by $H^{\dagger}_L$ is called the \emph{lax coend} of $F$.
\end{defn}
The counit of the adjunction $(-)^{\dagger}_L \ladj J_L$ gives us a 2-natural transformation $E:H^{\dagger}_L \to H$. Since computing weighted colimits is functorial in the weight, one has a ``comparison functor''
\[ \tn{col}(E,S \times T) : \tn{col}(H^{\dagger}_L,S \times T) \longrightarrow \tn{col}(H,S \times T) \]
between the lax and strict coends of interest. The unit of $(-)^{\dagger}_L \ladj J_L$ gives us $N:H \to H^{\dagger}_L$ in $\Algl L$, and by \cite{LackShul-Enhanced} Lemma 2.5 one has an adjunction $E \ladj N$ in $\Algl L$ with identity counit. As we shall now see, the existence of this adjunction at the level of weights enables us to verify that the above comparison between lax and strict coends is inverted by $\pi_0$.
\begin{prop}\label{prop:weight-result-for-pi0-invertibiility}
Let $A$ be a small 2-category, and $F:A \to \Cat$ and $I:A^{\op} \to \Cat$ be 2-functors. Then one has isomorphisms
\[ \pi_0\tn{col}(I,F) \iso \tn{col}(I,\pi_0F) \iso \tn{col}(d\pi_0I,\pi_0F) \]
2-natural in $I$ and $F$.
\end{prop}
\begin{proof}
The first isomorphism follows since $\pi_0 \ladj d$ is a 2-adjunction when one regards $\Set$ as a locally discrete 2-category. Consistent with our notation, we write $d_*\Set$ for the category of $\Set$ so regarded. For the second isomorphism we have the following sequence of natural isomorphisms
\[ \begin{array}{lll} {(d_*\Set)(\tn{col}(d\pi_0I,\pi_0F),X)} & {\iso} & {[A^{\op},\Cat](d\pi_0I,(d_*\Set)(\pi_0F,X))} \\
& {\iso} & {[A^{\op},\Cat](d\pi_0I,d(\Set(\pi_0F,X)))} \\
& {\iso} & {[A^{\op},\Set](\pi_0I,\Set(\pi_0F,X))} \\
& {\iso} & {[A^{\op},\Cat](I,d(\Set(\pi_0F,X)))} \\
& {\iso} & {[A^{\op},\Cat](I,(d_*\Set)(\pi_0F,X))} \\
& {\iso} & {(d_*\Set)(\tn{col}(I,\pi_0F),X)} \end{array} \]
coming from the definitions of weighted colimit and of $d_*$, $\pi_0 \ladj d$ and the 2-fully faithfulness of $d$.
\end{proof}
Since the 2-functor $\pi_0 \comp (-) : [P \times P^{\op},\Cat] \to [P \times P^{\op},\Set]$ given by composition with $\pi_0$ factors through the inclusion $J_L$, the existence of the adjunction $E \ladj N$ in $\Algl L$ ensures that $\pi_0 \comp (-)$ inverts $E$. Since the composite isomorphism of Proposition \ref{prop:weight-result-for-pi0-invertibiility} is natural in the weight, $\pi_0\tn{col}(E,S \times T)$ is an isomorphism since $d\pi_0E$ is. Thus we have proved
\begin{cor}\label{cor:pi0-inverts-laxcoend->coend}
The functor $\tn{col}(E,S \times T)$ is inverted by $\pi_0$.
\end{cor}
\noindent and so for the purposes of characterising $\pi_0$-exact squares, lax coends are as good as strict ones. The remainder of this section and the next is devoted to the computation of $\tn{col}(H^{\dagger}_L,S \times T)$, which is ultimately achieved in Corollary \ref{cor:lax-coend-formula} below.

As is well-known \cite{Bourke-Thesis, Lack-Codescent, Weber-CodescCrIntCat}, the weight $H^{\dagger}_L$ can be computed as the codescent object of $\ca R_LH$, where $\ca R_LH$ is the simplicial object in $[P \times P^{\op},\Cat]$ whose codescent-relevant parts are
\begin{equation}\label{eq:simp-object-for-H-dagger}
\xygraph{!{0;(2,0):(0,1)::} {L^3H}="p0" [r] {L^2H}="p1" [r] {LH}="p2"
"p2":"p1"|-{L\eta^L_H} "p1":@<1.5ex>"p2"^-{\mu^L_H} "p1":@<-1.5ex>"p2"_-{Lh} "p0":@<1.5ex>"p1"^-{\mu^L_{LH}} "p0":"p1"|-{L\mu^L_H} "p0":@<-1.5ex>"p1"_-{L^2h}}
\end{equation}
where $h:LH \to H$ is the strict $L$-algebra action for $H:P \times P^{\op} \to \Cat$, which in this case encodes its functoriality data in the first variable.

Let us unpack (\ref{eq:simp-object-for-H-dagger}) in more elementary terms. For $x,y \in P$ denote by $s_n(y,x)$ the set of sequences of objects of $P$ of length $(n+2)$ starting from $y$ and finishing at $x$. We denote a typical $z \in s_n(y,x)$ as $(z_1,...,z_n)$ and let $z_0=y$ and $z_{n+1}=x$. By the definitions of $L$ and $H$, for $x,y \in P$ one has
\[ L^nH(x,y) = \coprod\limits_{z \in s_n(y,x)} \prod_{i=0}^n P(z_i,z_{i+1}). \]
Thus one has following elementary description of the category $L^nH(x,y)$. An object is a functor $p:[n+1] \to P$ such that $p0 = y$ and $p(n+1) = x$. Such a $p$ is clearly a path of length $(n+1)$ from $y$ to $x$, we denote by $p_i:p(i-1) \to pi$ the $i$-th arrow in this path, and when convenient we shall denote $p$ also as the $(n+1)$-tuple $(p_1,...,p_{n+1})$. A morphism $\phi:p \to r$ in $L^nH(x,y)$ may be identified as an icon $p \to r$ in the sense of \cite{Lack-Icons}. Recall that an icon is a lax natural transformation whose 1-cell components are identities, and so to give such a $\phi$ is to give 2-cells $\phi_i:p_i \to r_i$ for $1 \leq i \leq n+1$.

In these terms the components of the morphisms appearing in (\ref{eq:simp-object-for-H-dagger}) are given on objects as follows. The effect of $\mu^L_{H,x,y}$ and $Lh_{x,y}$ on a path $(p_1,p_2,p_3)$ of length $3$ is $(p_1,p_3p_2)$ and $(p_2p_1,p_3)$ respectively, and $L\eta^L_{H,x,y}(p_1,p_2) = (p_1,1_{p1},p_2)$. The effect of $\mu^L_{LH,x,y}$, $L\mu^L_{H,x,y}$ and $L^2h_{x,y}$ on $(p_1,p_2,p_3,p_4)$ is $(p_1,p_2,p_4p_3)$, $(p_1,p_3p_2,p_4)$ and $(p_2p_1,p_3,p_4)$ respectively.

Since $L$ is a cartesian 2-monad (\ref{eq:simp-object-for-H-dagger}) is a category object (see for instance Proposition 4.4.1 of \cite{Weber-CodescCrIntCat}). However at this generality there is no reason why (\ref{eq:simp-object-for-H-dagger}) should be a crossed internal category. However despite the fact that the methods of \cite{Weber-CodescCrIntCat} do not apply here, we are nevertheless able to compute this codescent object. We do this by exhibiting a codescent cocone directly.
\begin{const}\label{const:lax-wedge-weight}
We now describe a 2-functor
\[\ca H : P \times P^{\op} \longrightarrow \TwoCat.\]
For $x,y \in P$ the 2-category $\ca H(x,y)$ is defined as follows. An object is an object of $LH(x,y)$, that is to say, a path of length 2 in $P$ from $y$ to $x$. A morphism $p \to r$ is a triple $(f,f_1,f_2)$ where $f:p1 \to r1$, $f_1:fp_1 \to r_1$ and $f_2:r_2f \to p_2$ as in
\[ \xygraph{!{0;(1.5,0):(0,.4)::} {y}="p0" [ur] {p1}="p1" [dr] {x.}="p2" [dl] {r1}="p3" "p0":"p1"^-{p_1}:"p2"^-{p_2}:@{<-}"p3"^-{r_2}:@{<-}"p0"^-{r_1} "p1":"p3"^-{f}
"p0" [u(.25)r(.6)] :@{=>}[d(.5)]^{f_1} "p2" [u(.25)l(.65)] :@{=>}[d(.5)]^{f_2}} \]
A 2-cell $(f,f_1,f_2) \to (g,g_1,g_2)$ is a 2-cell $\alpha:f \to g$ in $P$ such that $g_1(\alpha p_1)=f_1$ and $(r_2\alpha)f_2=g_2$. The 2-categorical compositions for $\ca H(x,y)$ are inherited in the evident way from those of $P$, and this construction is 2-functorial in $x$ and $y$.
\end{const}
The vertex of the codescent cocone we are in the process of describing is $\pi_{0*}\ca H$. We will describe the rest of the data as the effect of post-composition with $\pi_{0*}$ on $q_0:d_*LH \to \ca H$ and $q_1:q_0\mu^L_H \to q_0Lh$, noting that $\pi_{0*}d_*$ is the identity. Now $(q_0,q_1)$ will almost be a codescent cocone in $[P \times P^{\op},\TwoCat]$, except that the components of $q_1$ are lax natural transformations. To clarify what sort of entity $q_1$ really is, some preliminary remarks are in order.

For 2-categories $A$ and $B$, $[A,B]$ denotes the 2-category of 2-functors from $A$ to $B$, 2-natural transformations and modifications, and $[A,B]_{\tn{l}}$ denotes the 2-category of 2-functors from $A$ to $B$, lax natural transformations and modifications. Note that the assignation $(A,B) \mapsto [A,B]_{\tn{l}}$ is 2-functorial in $A$ and $B$, in fact this is part of a well-known closed structure on $\TwoCat$.  Given a small 2-category $A$ and 2-functors $X$ and  $Y:A \to \TwoCat$, we define the 2-category
\[ [A,\TwoCat](X,Y)_{\tn{l}} = \int_{a \in A} [Xa,Ya]_{\tn{l}} \]
in which the end on the right hand side is taken in the $\TwoCat$-enriched sense, where the tensor product on $\TwoCat$ is taken to be cartesian product. An object of this 2-category consists of 2-functors $F_a : Xa \to Ya$ for each $a \in A$, 2-naturally in $a$. A morphism $\phi:F \to G$ consists of lax natural transformations $\phi_a:F_a \to G_a$, 2-naturally in $a$. This naturality is in the evident sense, given that lax natural transformations can be horizontally composed with 2-functors, strict 2-natural transformations and modifications thereof. We call such a $\phi$ a \emph{lax modification}. Composition with $\pi_{0*}$ gives a 2-functor
\[ \pi_{0*} \comp (-) : [A,\TwoCat](X,Y)_{\tn{l}} \longrightarrow [A,\Cat](\pi_{0*}X,\pi_{0*}Y) \]
whose codomain is locally discrete. In particular $\pi_{0*} \comp (-)$ sends lax modifications to modifications.
\begin{const}\label{const:lax-codescent-cocone}
The 2-natural transformation on the left
\[ \begin{array}{lccr} {q_0 : d_*LH \longrightarrow \ca H} &&& {q_1:q_0\mu^L_H \to q_0Lh} \end{array} \]
and the lax modification on the right will now be described. For $x,y \in P$ define the 2-functor $q_{0,x,y} : d_*LH(x,y) \to \ca H(x,y)$ to be the identity on objects, and for $(f_1,f_2):(p_1,p_2) \to (r_1,r_2)$ in $LH(x,y)$, define $q_{0,x,y}(f_1,f_2) = (1_z,f_1,f_2)$, where $z = p1 = r1$. The 2-functors $q_{0,x,y}$ are easily seen to be be 2-natural in $x$ and $y$. The lax modification $q_1$ is defined as follows. For $x,y \in P$ and $p = (p_1,p_2,p_3)$ in $L^2H(x,y)$, we take the corresponding 1-cell component $(q_{1,x,y})_p$ to be $(p_2,\id,\id):(p_1,p_3p_2) \to (p_2p_1,p_3)$. The lax naturality 2-cell $(q_{1,x,y})_f$ with respect to $f:p \to r$ in $L^2H(x,y)$ is given by
\[ \xygraph{!{0;(3,0):(0,.3333)::} {(p_1,p_3p_2)}="p0" [r] {(r_1,r_3r_2)}="p1" [d] {(r_2r_1,r_3)}="p2" [l] {(p_2p_1,p_3)}="p3" "p0":"p1"^-{(1_{z_1},f_1,f_3 \cdot f_2)}:"p2"^-{(r_2,\id,\id)}:@{<-}"p3"^-{(1_{z_2},f_2 \cdot f_1,f_3)}:@{<-}"p0"^-{(p_2,\id,\id)} "p0" [d(.55)r(.42)] :@{=>}[r(.16)]^-{f_2}} \]
where $z_1 = p1 = r1$ and $z_2 = p2 = r2$, and ``$\cdot$'' denotes horizontal composition of 2-cells in $P$. The 2-naturality of the $q_{1,x,y}$ in $x$ and $y$ follows easily from the 2-category structure of $P$.
\end{const}
Note also that in the context of Construction \ref{const:lax-codescent-cocone} the equations
\[ \begin{array}{lccr} {q_1(L\eta^L_H) = \id} &&& {(q_1L^2h)(q_1\mu^L_{LH}) = q_1L\mu^L_H} \end{array} \]
also follow easily from the 2-category structure of $P$. Thus $(\pi_{0*}q_0,\pi_{0*}q_1)$ is a cocone for the simplicial object $\ca R_LH$.
\begin{prop}\label{prop:H-dagger}
$(\pi_{0*}q_0,\pi_{0*}q_1)$ is a codescent cocone which exhibits
\[ H^{\dagger}_L = \pi_{0*}\ca H. \]
\end{prop}
\begin{proof}
For $x,y \in P$ it suffices to show that $(\pi_{0*}q_{0,x,y},\pi_{0*}q_{1,x,y})$ is a codescent cocone in $\Cat$. For any 2-category $A$, let us denote by $\tn{LCD}(A)$ the set of pairs $(\phi_0,\phi_1)$ where $\phi_0:d_*LH(x,y) \to A$ is a 2-functor and $\phi_1:\phi_0\mu^L_H \to \phi_0Lh_{x,y}$ is a lax natural transformation, such that
\[ \begin{array}{lccr} {\phi_1L\eta^L_{H,x,y} = \id} &&& {(\phi_1L^2h_{x,y})(\phi_1\mu^L_{LH,x,y}) = \phi_1L\mu^L_{H,x,y}.} \end{array} \]
Precomposition with $(q_{0,x,y}, q_{1,x,y})$ gives a function
\[ (-) \comp (q_{0,x,y}, q_{1,x,y}) : \tn{ob}[\ca H(x,y),A] \longrightarrow \tn{LCD}(A).  \]
When $A$ is locally discrete, that is, of the form $d_*B$ for some category $B$, $\tn{LCD}(A)$ is in bijection with the set of cocones for $(\ca R_LH)(x,y)$ with vertex $B$, and under this correspondence, composition with $(q_{0,x,y}, q_{1,x,y})$ is identified with composition with $(\pi_{0*}q_{0,x,y}, \pi_{0*}q_{1,x,y})$. Thus it suffices to show that $(-) \comp (q_{0,x,y}, q_{1,x,y})$ is bijective when $A$ is locally discrete. So we suppose that we are given $(\phi_0,\phi_1)$ as above with $A$ locally discrete, and we must exhibit $\overline{\phi}:\ca H(x,y) \to A$ unique such that $\overline{\phi}q_{0,x,y} = \phi_0$ and $\overline{\phi}q_{1,x,y} = \phi_1$.

The definition $\overline{\phi}(p_1,p_2) = \phi_0(p_1,p_2)$ of $\overline{\phi}$ on objects is forced by the equation $\overline{\phi}q_{0,x,y} = \phi_0$ on objects. Note that any arrow $(f,f_1,f_2):(p_1,p_2) \to (r_1,r_2)$ in $\ca H(x,y)$ can be factored in the following way
\[ (p_1,p_2) \xrightarrow{(1_{p1},\id,f_2)} (p_1,r_2f) \xrightarrow{(f,\id,\id)} (fp_1,r_2) \xrightarrow{(1_{r1},f_1,\id)} (r_1,r_2) \]
and so the equations $\overline{\phi}q_{0,x,y} = \phi_0$ and $\overline{\phi}q_{1,x,y} = \phi_1$ force the definition
\[ \overline{\phi}(f,f_1,f_2) = \phi_0(f_1,\id)(\phi_{1})_{(p_1,f,r_2)}\phi_0(\id,f_2) \]
of $\overline{\phi}$ on 1-cells. Any 2-cell $\alpha:(f,f_1,f_2) \to (g,g_1,g_2)$ in $\ca H(x,y)$ can be factored as
\[ \xygraph{!{0;(2.5,0):(0,.4)::} {(p_1,p_2)}="p0" [r] {(p_1,r_2f)}="p1" [r] {(fp_1,r_2)}="p2" [dr] {(r_1,r_2)}="p3" [l] {(gp_1,r_2)}="p4" [l] {(p_1,r_2g)}="p5" "p0":"p1"^-{(1_{p1},\id,f_2)}:"p2"^-{(f,\id,\id)}:@/^{1pc}/"p3"^-{(1_{r1},f_1,\id)}:@{<-}"p4"^-{(1_{r1},g_1,\id)}:@{<-}"p5"^-{(g,\id,\id)}:@/^{1pc}/@{<-}"p0"^-{(1_{p1},\id,g_2)} "p1":"p5"|-{(1_{p1},\id,r_2\alpha)} "p2":"p4"|-{(1_{r1},\alpha p_1,\id)}
"p1" [d(.35)r(.45)] :@{=>}[d(.3)]^{\alpha}
"p0" [d(.5)r(.55)] {\scriptstyle =} "p2" [d(.5)r(.45)] {\scriptstyle =}} \]
and we observe that the lax square in the middle of this last diagram is just the lax naturality 2-cell $(q_{1,x,y})_{(\id,\alpha,\id)}$. Thus the equation $\overline{\phi}q_{1,x,y} = \phi_1$ forces us to define $\overline{\phi}(\alpha)$ to be the composite 2-cell
\[ \xygraph{!{0;(2.5,0):(0,.4)::} {\phi_0(p_1,p_2)}="p0" [r] {\phi_0(p_1,r_2f)}="p1" [r] {\phi_0(fp_1,r_2)}="p2" [dr] {\phi_0(r_1,r_2)}="p3" [l] {\phi_0(gp_1,r_2)}="p4" [l] {\phi_0(p_1,r_2g)}="p5" "p0":"p1"^-{\phi_0(\id,f_2)}:"p2"^-{(\phi_1)_{(p_1,f,r_2)}}:@{}"p3":@{<-}"p4"^-{\phi_0(g_1,\id)}:@{<-}"p5"^-{(\phi_1)_{(p_1,g,r_2)}}:@{}"p0"
"p1":"p5"_-{\phi_0(\id,r_2\alpha)} "p2":"p4"^-{\alpha p_1,\id)}
"p1" [d(.35)r(.35)] :@{=>}[d(.3)]^{(\phi_1)_{(\id,\alpha,\id)}}} \]
in $X$, which since $X$ is locally discrete, is an identity. It suffices to show that with these assignations, $\overline{\phi}$ respects the composition of 1-cells.

Given the 3-fold factorisation of arrows of $\ca H(x,y)$ described above, it suffices to verify that $\overline{\phi}$ is functorial with respect to morphisms of the form
\[ \begin{array}{lcccccr} {\tn{(i)} \,\, (f,\id,\id)} &&& {\tn{(ii)} \,\, (1,f_1,\id)} &&& {\tn{(iii)} \,\, (1,\id,f_2)} \end{array} \]
separately, and moreover given $(f,f_1,f_2):p \to r$ and $(g,g_1,g_2):r \to s$ in $\ca H(x,y)$, that
\begin{itemize}
\item[(iv)] $\overline{\phi}(1_{r1},\id_{r_1},g_2)\overline{\phi}(1_{r1},f_1,\id_{r_2}) = \overline{\phi}(1_{r1},f_1,\id_{r_2})\overline{\phi}(1_{r1},\id_{r_1},g_2)$
\item[(v)] $\overline{\phi}(1_{r1},\id_{fp_1},g_2)\overline{\phi}(f,\id_{fp_1},\id_{r_2f}) = \overline{\phi}(f,\id_{fp_1},\id_{r_2gf})\overline{\phi}(1_{p1},\id_{p_1},g_2f)$
\item[(vi)] $\overline{\phi}(g,\id_{gr_1},\id_{s_2g})\overline{\phi}(1_{r_1},f_1,\id_{s_2g}) = \overline{\phi}(1_{s1},gf_1,\id_{s_2})\overline{\phi}(g,\id_{gfp_1},\id_{s_2g})$
\end{itemize}
so that $\overline{\phi}$ respects how these 3 classes of morphisms interact. 

Functoriality in the case (i) follows by the lax naturality axioms $\phi_1$, and in the cases (ii) and (iii) by the functoriality of $\phi_0$. The calculation
\[ \begin{array}{rllll}
{\overline{\phi}(1,\id,g_2)\overline{\phi}(1,f_1,\id)}
& = & {\phi_0(\id,g_2)\phi_0(f_1,\id)}
& = & {\phi_0(f_1,g_2)} \\ {}
& = & {\phi_0(f_1,\id)\phi_0(\id,g_2)}
& = & {\overline{\phi}(1,f_1,\id)\overline{\phi}(1,\id,g_2)}  \end{array} \]
establishes (iv). As for (v) and (vi) the lax naturality 2-cells $(\phi_1)_{(g_2,\id,\id)}$ and $(\phi_1)_{(\id,\id,f_1)}$ give 2-cells between opposing sides of these equations, which since $X$ is locally discrete, are identities.
\end{proof}

\subsection{Lax wedges}
\label{ssec:lax-wedges}
The formula $H^{\dagger}_L = \pi_{0*}\ca H$ of Proposition \ref{prop:H-dagger} tells us that we know how to compute lax coends if we know how to compute the corresponding 2-categorical colimit weighted by $\ca H$, by the following result.
\begin{prop}\label{prop:weights-of-the-form-pi0star-2catwt}
Let $A$ be a 2-category, and $I:A^{\op} \to \TwoCat$ and $F:A \to \Cat$ be 2-functors. Then $\tn{col}(\pi_{0*}I,F) = \pi_{0*}\tn{col}(I,d_*F)$.
\end{prop}
\begin{proof}
One has natural isomorphisms
\[ \begin{array}{lll}
{\Cat(\tn{col}(\pi_{0*}I,F),X)}  & {\iso} & {[A^{\op},\Cat](\pi_{0*}I,\Cat(F,X))} \\
{}  & {\iso} & {[A^{\op},\TwoCat](I,d_*\Cat(F,X))} \\
{}  & {\iso} & {[A^{\op},\TwoCat](I,\TwoCat(d_*F,d_*X))} \\
{}  & {\iso} & {\TwoCat(\tn{col}(I,d_*F),d_*X)} \\
{}  & {\iso} & {\Cat(\pi_{0*}\tn{col}(I,d_*F),X)}
\end{array} \]
because of the definition of weighted colimit, the adjunction $\pi_{0*} \ladj d_*$, and since $d_*$ is 2-fully faithful.
\end{proof}
\noindent Thus the problem of understanding the lax coend $\tn{col}(H^{\dagger}_L,S \times T)$ comes down to that of understanding the weighted colimit $\tn{col}(\ca H,d_*S \times d_*T)$, which is a colimit in the setting of $\TwoCat$-enriched category theory. We now turn to an analysis of these.

By definition an $\ca H$-cocone for $F:P^{\op} \times P \to \TwoCat$ with vertex $X \in \TwoCat$ consists of 2-functors
\[ \phi_{x,y} : \ca H(x,y) \longrightarrow \TwoCat(F(x,y),X) \]
which are 2-natural in $x$ and $y$. In this section we shall give a minimalistic combinatorial description of the data contained in an $\ca H$-cocone, and using this, exhibit the universal such in the case where $F = d_*S \times d_*T$.
\begin{defn}\label{def:lax-wedge}
A \emph{lax wedge} $\psi$ for $F$ with vertex $X$ consists of
\begin{itemize}
\item $\forall x \in P$, a 2-functor $\psi_x:F(x,x) \to X$.
\item $\forall f:x \to y \in P$, a 2-natural transformation $\psi_f:\psi_xF(f,1) \to \psi_yF(1,f)$.
\item $\forall$ 2-cells $\alpha:f \to g$ in $P$, a modification
\[ \xygraph{!{0;(2.5,0):(0,.4)::} {\psi_xF(f,1)}="p0" [r] {\psi_xF(g,1)}="p1" [d] {\psi_yF(1,g)}="p2" [l] {\psi_yF(1,f)}="p3" "p0":"p1"^-{\psi_xF(\alpha,\id)}:"p2"^-{\psi_g}:@{<-}"p3"^-{\psi_yF(\id,\alpha)}:@{<-}"p0"^-{\psi_f} "p0" [d(.55)r(.43)] :@{=>}[r(.14)]^-{\psi_{\alpha}}} \]
\end{itemize}
subject to the unit, 1-cell composition, 2-cell vertical composition and 2-cell horizonal composition axioms. The unit axioms say that $\forall x \in P$, $\psi_{1_x} = 1_{\psi_x}$, and $\forall f:x \to y \in P$, $\psi_{\id_f} = \id_{\psi_f}$. The 1-cell composition axiom says that given $f:x \to y$ and $g:y \to z$ in $P$, $\psi_{gf} = (\psi_{g}F(1,f))(\psi_{f}F(g,1))$. The 2-cell vertical composition axiom says that given $\alpha$ and $\beta$ in $P$ as on the left in
\[ \xygraph{{\xybox{\xygraph{{x}="p0" [r(1.25)] {y}="p1" "p0":@/^{2pc}/"p1"^-{f}|-{}="t" "p0":"p1"|-{g}="m" "p0":@/_{2pc}/"p1"_-{h}|{}="b"
"t":@{}"m"|(.25){}="d1"|(.75){}="c1" "m":@{}"b"|(.25){}="d2"|(.75){}="c2"
"d1":@{=>}"c1"^-{\alpha} "d2":@{=>}"c2"^-{\beta}}}}
[r(6)]
{\xybox{\xygraph{!{0;(2.5,0):(0,.4)::}
{\psi_xF(f,1)}="p0" [r] {\psi_xF(g,1)}="p1" [r] {\psi_xF(h,1)}="p2" [d] {\psi_yF(1,h)}="p3" [l] {\psi_yF(1,g)}="p4" [l] {\psi_yF(1,f)}="p5" "p0":"p1"^-{\psi_xF(\alpha,\id)}:"p2"^-{\psi_xF(\beta,\id)}:"p3"^-{
\psi_h}:@{<-}"p4"^-{\psi_yF(\id,\beta)}:@{<-}"p5"^-{\psi_yF(\id,\alpha)}:@{<-}"p0"^-{\psi_f} "p1":"p4"^{\psi_g}
"p0" [d(.55)r(.43)] :@{=>}[r(.14)]^-{\psi_{\alpha}}
"p1" [d(.55)r(.43)] :@{=>}[r(.14)]^-{\psi_{\beta}}}}}} \]
the composite on the right in the previous display equals $\psi_{\beta\alpha}$. The 2-cell horizontal composition axiom says that given $\alpha$ and $\beta$ in $P$ as on the left in
\[ \xygraph{{\xybox{\xygraph{{x}="p0" [r] {y}="p1" [r] {z}="p2"
"p0":@/^{1pc}/"p1"^-{f}|-{}="tl":@/^{1pc}/"p2"^-{g}|-{}="tr"
"p0":@/_{1pc}/"p1"_-{h}|-{}="bl":@/_{1pc}/"p2"_-{k}|-{}="br"
"tl":@{}"bl"|(.25){}="d1"|(.75){}="c1" "tr":@{}"br"|(.25){}="d2"|(.75){}="c2"
"d1":@{=>}"c1"^-{\alpha} "d2":@{=>}"c2"^-{\beta}}}}
[r(6)]
{\xybox{\xygraph{!{0;(2.5,0):(0,.4)::} {\psi_xF(gf,1)}="p0" [r] {\psi_yF(g,f)}="p1" [r] {\psi_zF(1,gf)}="p2" [d] {\psi_zF(1,kh)}="p3" [l] {\psi_yF(k,h)}="p4" [l] {\psi_xF(kh,1)}="p5" "p0":"p1"^-{\psi_fF(g,1)}:"p2"^-{\psi_gF(1,f)}:"p3"^-{\psi_zF(1,\beta\cdot\alpha)}:@{<-}"p4"^-{\psi_kF(1,h)}:@{<-}"p5"^-{\psi_hF(k,1)}:@{<-}"p0"^-{\psi_xF(\beta\cdot\alpha,\id)} "p1":"p4"|-{\psi_yF(\beta,\alpha)}
"p0" [d(.35)r(.2)] :@{=>}[d(.3)]^-{\psi_{\alpha}F(\beta,\id)}
"p1" [d(.35)r(.35)] :@{=>}[d(.3)]^-{\psi_{\beta}F(\id,\alpha)}}}}} \]
the composite on the right in the previous display equals $\psi_{\beta\cdot\alpha}$.
\end{defn}
Let $\phi:\ca H \to \TwoCat(F,X)$ be an $\ca H$-cocone as above. For $x \in P$, $(1_x,1_x)$ is a path in $P$ of length 2 from $x$ to itself, and thus an object of $\ca H(x,x)$. We define $\overline{\phi}_x = \phi_{x,x}(1_x,1_x)$ so that by definition, $\overline{\phi}_x$ is a 2-functor $F(x,x) \to X$. Let $f:x \to y$ be in $P$. Since
\[ \begin{array}{lccr} {(1_x,f) = \ca H(f,1_x)(1_x,1_x)} &&& {(f,1_y) = \ca H(1_y,f)(1_y,1_y)} \end{array} \]
by the naturality of $\phi$ one has
\[ \begin{array}{lccr} {\phi_{y,x}(1_x,f) = \overline{\phi}_xF(f,1)} &&& {\phi_{y,x}(f,1_y) = \overline{\phi}_yF(1,f).} \end{array} \]
Moreover one has a morphism $(f,\id_f,\id_f):(1_x,f) \to (f,1_y)$ of $\ca H(y,x)$, and so one can define $\overline{\phi}_f = \phi_{y,x}(f,\id_f,\id_f)$, so that by definition $\overline{\phi}_f$ is a 2-natural transformation $\overline{\phi}_xF(f,1) \to \overline{\phi}_yF(1,f)$. Let $\alpha:f \to g$ be a 2-cell in $P$. Since
\[ \begin{array}{lccr} {(1_x,\id_{1_x},\alpha) = \ca H(\alpha,\id)_{(1_x,1_x)}} &&& {(1_y,\alpha,\id_{1_y}) = \ca H(\id,\alpha)_{(1_y,1_y)}} \end{array} \]
by the naturality of $\phi$ one has
\[ \begin{array}{lccr} {\phi_{y,x}(1_x,\id_{1_x},\alpha) = \overline{\phi}_xF(\alpha,\id)} &&& {\phi_{y,x}(1_y,\alpha,\id_{1_y}) = \overline{\phi}_yF(\id,\alpha).} \end{array} \]
Moreover one has a 2-cell as on the left
\[ \xygraph{{\xybox{\xygraph{!{0;(2,0):(0,.5)::} {(1_x,f)}="p0" [r] {(1_x,g)}="p1" [d] {(g,1_y)}="p2" [l] {(f,1_y)}="p3" "p0":"p1"^-{(1_x,\id,\alpha)}:"p2"^-{(g,\id,\id)}:@{<-}"p3"^-{(1_y,\alpha,\id)}:@{<-}"p0"^-{(f,\id,\id)} "p0" [d(.55)r(.4)] :@{=>}[r(.2)]^-{\alpha}}}}
[r(5)]
{\xybox{\xygraph{!{0;(2.5,0):(0,.4)::} {\overline{\phi}_xF(f,1)}="p0" [r] {\overline{\phi}_xF(g,1)}="p1" [d] {\overline{\phi}_yF(1,g)}="p2" [l] {\overline{\phi}_yF(1,f)}="p3" "p0":"p1"^-{\overline{\phi}_xF(\alpha,\id)}:"p2"^-{\overline{\phi}_g}:@{<-}"p3"^-{\overline{\phi}_yF(\id,\alpha)}:@{<-}"p0"^-{\overline{\phi}_f} "p0" [d(.55)r(.43)] :@{=>}[r(.14)]^-{\overline{\phi}_{\alpha}}}}}} \]
in $\ca H(y,x)$. We define $\overline{\phi}_{\alpha}$ to be the effect of $\phi_{y,x}$ on this 2-cell, so that by definition $\overline{\phi}_{\alpha}$ is a modification as indicated on the right in the previous display. Thus from an $\ca H$-cocone $\phi$ we have defined the data of a lax wedge $\overline{\phi}$.
\begin{lem}\label{lem:lax-wedge-from-caH-cocone}
In the manner just described, every $\ca H$-cocone $\phi$ for $F$ with vertex $X$ determines a lax wedge $\overline{\phi}$ for $F$ with vertex $X$.
\end{lem}
\begin{proof}
We must verify the lax wedge axioms for $\overline{\phi}$. The unit axioms follow from the 2-functoriality of the components of $\phi$. Given $f:x \to y$ and $g:y \to z$ in $P$, note that $(gf,\id_{gf},\id_{gf}) = (g,\id_{gf},\id_g)(f,\id_f,\id_{gf})$ and
\[ \begin{array}{lccr}
{(f,\id_f,\id_{gf}) = \ca H(g,1_x)(f,\id_f,\id_f)} &&&
{(g,\id_{gf},\id_g) = \ca H(1_z,f)(g,\id_g,\id_g)} \end{array} \]
and so the 1-cell axiom for $\overline{\phi}$ follows from the naturality of $\phi$ and the functoriality of $\phi$'s components. Given $f$, $g$ and $h:x \to y$, and $\alpha:f \to g$ and $\beta:g \to h$ in $P$, one has
\[ \xygraph{{\xybox{\xygraph{!{0;(2,0):(0,.5)::}
{(1_x,f)}="p0" [r] {(1_x,h)}="p1" [d] {(h,1_y)}="p2" [l] {(f,1_y)}="p3" "p0":"p1"^-{(1_x,\id,\beta\alpha)}:"p2"^-{(h,\id,\id)}:@{<-}"p3"^-{(1_y,\beta\alpha,\id)}:@{<-}"p0"^-{(f,\id,\id)}
"p0" [d(.55)r(.4)] :@{=>}[r(.2)]^-{\beta\alpha}}}}
[r(2.5)] {=} [r(3.5)]
{\xybox{\xygraph{!{0;(2,0):(0,.5)::}
{(1_x,f)}="p0" [r] {(1_x,g)}="p1" [r] {(1_x,h)}="p2" [d] {(h,1_y)}="p3" [l] {(g,1_y)}="p4" [l] {(f,1_y)}="p5" "p0":"p1"^-{(1_x,\id,\alpha)}:"p2"^-{(1_x,\id,\beta)}:"p3"^-{(h,\id,\id)}:@{<-}"p4"^-{(1_y,\beta,\id)}:@{<-}"p5"^-{(1_y,\alpha,\id)}:@{<-}"p0"^-{(f,\id,\id)} "p1":"p4"|-{(g,\id,\id)}
"p0" [d(.55)r(.4)] :@{=>}[r(.2)]^-{\alpha} "p1" [d(.55)r(.4)] :@{=>}[r(.2)]^-{\beta}}}}} \]
in $\ca H(y,x)$, and so the 2-cell vertical composition axiom for $\overline{\phi}$ follows from the 2-naturality of $\phi$ and the 2-functoriality of $\phi$'s components. Given $\alpha$ and $\beta$ as on the left
\[ \xygraph{{\xybox{\xygraph{{x}="p0" [r] {y}="p1" [r] {z}="p2"
"p0":@/^{1pc}/"p1"^-{f}|-{}="tl":@/^{1pc}/"p2"^-{g}|-{}="tr"
"p0":@/_{1pc}/"p1"_-{h}|-{}="bl":@/_{1pc}/"p2"_-{k}|-{}="br"
"tl":@{}"bl"|(.25){}="d1"|(.75){}="c1" "tr":@{}"br"|(.25){}="d2"|(.75){}="c2"
"d1":@{=>}"c1"^-{\alpha} "d2":@{=>}"c2"^-{\beta}}}}
[r(6)]
{\xybox{\xygraph{!{0;(2.5,0):(0,.4)::}
{(1_x,gf)}="p0" [r] {(1_x,kh)}="p1" [d] {(kh,1_z)}="p2" [l] {(gf,1_z)}="p3" "p0":"p1"^-{(1_x,\id,\beta\cdot\alpha)}:"p2"^-{(kh,\id,\id)}:@{<-}"p3"^-{(1_z,\beta\cdot\alpha,\id)}:@{<-}"p0"^-{(gf,\id,\id)}
"p0" [d(.55)r(.4)] :@{=>}[r(.2)]^-{\beta\cdot\alpha}}}}} \]
in $P$, the 2-cell in $\ca H(y,x)$ indicated on the right in the previous display is, by a straight forward calculation, the following horizontal composite
\[ \xygraph{
{\xybox{\xygraph{!{0;(2,0):(0,.5)::}
{(1_y,g)}="p0" [r] {(1_y,k)}="p1" [d] {(k,1_z)}="p2" [l] {(g,1_z)}="p3" "p0":"p1"^-{(1_y,\id,\beta)}:"p2"|-{(k,\id,\id)}:@{<-}"p3"^-{(1_z,\beta,\id)}:@{<-}"p0"|-{(g,\id,\id)}
"p0" [d(.55)r(.4)] :@{=>}[r(.2)]^-{\beta}}}}="pl"
[r(4.75)]
{\xybox{\xygraph{!{0;(2,0):(0,.5)::}
{(1_x,f)}="p0" [r] {(1_x,h)}="p1" [d] {(h,1_y)}="p2" [l] {(f,1_y)}="p3" "p0":"p1"^-{(1_x,\id,\alpha)}:"p2"|-{(h,\id,\id)}:@{<-}"p3"^-{(1_y,\alpha,\id)}:@{<-}"p0"|-{(f,\id,\id)}
"p0" [d(.55)r(.4)] :@{=>}[r(.2)]^-{\alpha}}}}="pr"
"pl" ([u(.9)l(1.5)] {}="tc1",[d(.9)l(1.5)] {}="bc1") "tc1"-@/_{.75pc}/"bc1"
"pl" ([u(.9)r(1.2)] {}="tc2",[d(.9)r(1.2)] {}="bc2") "tc2"-@/^{.75pc}/"bc2"
"pr" ([u(.9)l(1.5)] {}="tc3",[d(.9)l(1.5)] {}="bc3") "tc3"-@/_{.75pc}/"bc3"
"pr" ([u(.9)r(1.2)] {}="tc4",[d(.9)r(1.2)] {}="bc4") "tc4"-@/^{.75pc}/"bc4"
"pl" [l(2.35)] {\ca H(1_z,h)} "pr" [l(2.35)] {\ca H(g,1_x)} [l(.75)] {\cdot}
} \]
and so by the 2-naturality of $\phi$ and the 2-functoriality of $\phi$'s components, $\overline{\phi}_{\beta\cdot\alpha}$ is the horizontal composite embodied in the solid parts of
\[ \xygraph{!{0;(4,0):(0,.25)::} {\overline{\phi}_xF(gf,1)}="p0" [r] {\overline{\phi}_yF(g,f)}="p1" [r] {\overline{\phi}_zF(1,gf)}="p2" [d] {\overline{\phi}_zF(1,gh)}="p3" [d] {\overline{\phi}_zF(1,kh)}="p4" [l] {\overline{\phi}_yF(k,h)}="p5" [l] {\overline{\phi}_xF(kh,1)}="p6" [u] {\overline{\phi}_xF(gh,1)}="p7" [r] {\overline{\phi}_yF(g,h)}="p8" "p0":"p1"^-{\overline{\phi}_fF(g,1)}:@{.>}"p2"^-{\overline{\phi}_gF(1,f)}:@{.>}"p3"^-{\overline{\phi}_zF(1,g\alpha)}:"p4"^-{\overline{\phi}_zF(1,\beta h)}:@{<-}"p5"^-{\overline{\phi}_kF(1,h)}:@{<.}"p6"^-{\overline{\phi}_hF(k,1)}:@{<.}"p7"^-{\overline{\phi}_xF(\beta h,\id)}:@{<-}"p0"^-{\overline{\phi}_xF(g\alpha,\id)} "p8"(:@{<-}"p1"_-{\overline{\phi}_yF(g,\alpha)},:"p3"^-{\overline{\phi}_gF(1,h)},:"p5"_-{\overline{\phi}_yF(\beta,h)},:@{<-}"p7"^-{\overline{\phi}_hF(g,1)})
"p0" [d(.35)r(.35)] :@{=>}[d(.3)]^-{\overline{\phi}_{\alpha}F(g,1)}
"p8" [d(.35)r(.35)] :@{=>}[d(.3)]^-{\overline{\phi}_{\beta}F(1,h)}
"p6" [u(.4)r(.4)] {\scriptstyle =} "p2" [d(.4)l(.4)] {\scriptstyle =}} \]
and the naturality of $\phi$ provides the remaining commutative squares. The 2-cell horizontal composition axiom for $\overline{\phi}$ follows since the above diagram in its entirety is the required decomposition of $\overline{\phi}_{\beta\cdot\alpha}$ in terms of $\overline{\phi}_{\alpha}$ and $\overline{\phi}_{\beta}$.
\end{proof}
\begin{lem}\label{lem:lax-wedge-equals-caH-cocone}
The assignation $\phi \mapsto \overline{\phi}$ given by Lemma \ref{lem:lax-wedge-from-caH-cocone} gives a bijection between the set of $\ca H$-cocones for $F$ with vertex $X$, and the set of lax wedges for $F$ with vertex $X$, naturally in $F$ and $X$.
\end{lem}
\begin{proof}
The given assignation is clearly compatible with precomposition by 2-natural transformations $F' \to F$ and with postcomposition by 2-functors $X \to X'$, and so is natural in the required sense. Given $(p_1,p_2) \in \ca H(x,y)$, we have $(p_1,p_2)= \ca H(p_2,p_1)(1_{p1},1_{p1})$, and so by the naturality of $\phi$ we have
\begin{equation}\label{eq:phi-on-objects}
\phi_{x,y}(p_1,p_2) = \overline{\phi}_{p1}F(p_2,p_1).
\end{equation}
The 3-way factorisation of $(f,f_1,f_2):(p_1,p_2) \to (r_1,r_2)$ described in the proof of Proposition \ref{prop:H-dagger} can be written as
\[ (f,f_1,f_2) = (\ca H(\id,f_1)_{(1_{r1},1_{r1})})(\ca H(r_2,p_1)(f,\id_f,\id_f))(\ca H(f_2,\id)_{(1_{p1},1_{p1})}) \]
and so by the naturality of $\phi$ and the functoriality of $\phi$'s components one has
\begin{equation}\label{eq:phi-on-arrows}
\phi_{x,y}(f,f_1,f_2) = (\overline{\phi}_{r1}F(\id,f_1))(\overline{\phi}_fF(r_2,p_1))(\overline{\phi}_{p1}F(f_2,\id)).
\end{equation}
Given a 2-cell $\alpha:(f,f_1,f_2) \to (g,g_1,g_2)$ in $\ca H(x,y)$, its corresponding decomposition, also described in the proof of Proposition \ref{prop:H-dagger}, can be written as
\[ \xygraph{*=(0,1.8){\xybox{\xygraph{!{0;(5,0):(0,.2)::}
{\ca H(p_2,p_1)(1_{p1},1_{p1})}="p0" [r] {\ca H(r_2f,p_1)(1_{p1},1_{p1})}="p1" [r] {\ca H(r_2,fp_1)(1_{r1},1_{r1})}="p2" [d] {\ca H(r_2,r_1)(1_{r1},1_{r1})}="p3" [l] {\ca H(r_2,gp_1)(1_{r1},1_{r1})}="p4" [l] {\ca H(r_2g,p_1)(1_{p1},1_{p1})}="p5"
"p0":"p1"|-{}="a1":"p2"|-{}="a2":@{}"p3":@{<-}"p4"|-{}="a3":@{<-}"p5"|-{}="a4":@{}"p0"
"p1":"p5"|-{}="a5" "p2":"p4"|-{}="a6"
"a1" [u(.3)] {\scriptstyle \ca H(f_2,\id)_{(1_{p1},1_{p1})}}
"a2" [u(.3)] {\scriptstyle \ca H(r_2,p_1)(f,\id_f,\id_f)}
"a3" [d(.3)] {\scriptstyle \ca H(\id,g_1)_{(1_{r1},1_{r1})}}
"a4" [d(.3)] {\scriptstyle \ca H(r_2,p_1)(g,\id_g,\id_g)}
"a5" [l(.4)u(.05)] {\scriptstyle \ca H(r_2,p_1)\ca H(\alpha,\id)_{(1_{p1},1_{p1})}}
"a6" [r(.5)] {\scriptstyle \ca H(r_2,p_1)\ca H(\id,\alpha)_{(1_{r1},1_{r1})}}
"p1" [d(.35)l(.1)] :@{=>}[d(.3)]^{\ca H(r_2,p_1)\alpha}}}}} \]
and so by the 2-naturality of $\phi$ and the 2-functoriality of $\phi$'s components one has
\begin{equation}\label{eq:phi-on-2-cells}
\phi_{x,y}(\alpha) = (\overline{\phi}_{r1}F(\id,g_1))(\overline{\phi}_{\alpha}F(r_2,p_1))(\overline{\phi}_{p1}F(f_2,\id)).
\end{equation}
Thus by (\ref{eq:phi-on-objects})-(\ref{eq:phi-on-2-cells}) $\phi \mapsto \overline{\phi}$ is injective. To show that $\phi \mapsto \overline{\phi}$ is surjective it suffices to show that given the lax wedge $\overline{\phi}$, and taking (\ref{eq:phi-on-objects})-(\ref{eq:phi-on-2-cells}) as a definition of $\phi$, that the components $\phi_{x,y}$ are 2-functors and are 2-natural in $x$ and $y$, and moreover that the lax wedge corresponding to $\phi$ is $\overline{\phi}$.

This last fact follows by applying (\ref{eq:phi-on-objects}) to the cases $(p_1,p_2) = (1_x,1_x)$, (\ref{eq:phi-on-arrows}) to the cases $(f,f_1,f_2) = (f,\id_f,\id_f)$, and (\ref{eq:phi-on-2-cells}) to the cases $\alpha:(f,\alpha,\id_f) \to (g,\id_g,\alpha)$ where $\alpha:f \to g$ is a 2-cell of $P$. Two-naturality in $x$ and $y$ is obvious from the definitions (\ref{eq:phi-on-objects})-(\ref{eq:phi-on-2-cells}). The unit axioms of 2-functoriality are exactly the unit axioms for $\overline{\phi}$. That the $\phi_{x,y}$ respect horizontal composition of 1-cells, vertical composition of 2-cells and horizontal composition of 2-cells, is easily seen to be a consequence of the definitions and the 1-cell composition axiom, the 2-cell vertical composition axiom and the 2-cell horizontal composition axiom respectively, by straight forward calculations that are left to the reader.
\end{proof}
\begin{rem}\label{rem:Hdagger-cocone}
For $F:P^{\op} \times P \to \Cat$ and $X \in \Cat$, an $H^{\dagger}_L$-cocone for $F$ with vertex $X$ is the same thing as a lax wedge for $d_*F$ with vertex $d_*X$, by Propositions \ref{prop:H-dagger}, \ref{prop:weights-of-the-form-pi0star-2catwt} and \ref{lem:lax-wedge-equals-caH-cocone}, and the adjunction $\pi_{0*} \ladj d_*$. Thus such an $H^{\dagger}_L$-cocone amounts to the following data:
\begin{itemize}
\item $\forall x \in P$, a functor $\psi_x:F(x,x) \to X$.
\item $\forall f:x \to y \in P$, a natural transformation $\psi_f:\psi_xF(f,1) \to \psi_yF(1,f)$.
\end{itemize}
such that
\begin{enumerate}
\item $\forall$ 2-cells $\alpha:f \to g$ in $P$, $(\psi_yF(\id,\alpha))\psi_f = \psi_g(\psi_xF(\alpha,\id))$.
\item $\forall x \in P$, $\psi_{1_x} = 1_{\psi_x}$.
\item Given $f:x \to y$ and $g:y \to z$ in $P$, $\psi_{gf} = (\psi_{g}F(1,f))(\psi_{f}F(g,1))$.
\end{enumerate}
\end{rem}
We turn now to the task of exhibiting a universal lax wedge for $d_*S \times d_*T$ given 2-functors $S:P^{\op} \to \Cat$ and $T:P \to \Cat$. We begin by describing its vertex.

To $S$ we can associate the 2-category $1 \dbslash S$ which is described as follows. An object is a pair $(x,y)$ where $x \in P$ and $y \in Sx$. An arrow $(x_1,y_1) \to (x_2,y_2)$ is a pair $(f_1,f_2)$ where $f_1:x _1 \to x_2$ and $f_2:y_1 \to Sf_1y_2$. A 2-cell $\alpha:(f_1,f_2) \to (g_1,g_2)$ consists of a 2-cell $\alpha:f_1 \to g_1$ in $P$ such that $(S(\alpha)y_2)f_2 = g_2$. With the compositions in $1 \dbslash S$ inherited in the obvious way from $P$, $(x,y) \mapsto x$ becomes the object map of a 2-functor $1 \dbslash S \to P$.

Similarly to $T$ we can associate the 2-category $1 // T$ which is described as follows. An object is a pair $(x,z)$ where $x \in P$ and $z \in Tx$. An arrow $(x_1,z_1) \to (x_2,z_2)$ is a pair $(f_1,f_2)$ where $f_1:x _1 \to x_2$ and $f_2:Tf_1z_1 \to z_2$. A 2-cell $\alpha:(f_1,f_2) \to (g_1,g_2)$ consists of a 2-cell $\alpha:f_1 \to g_1$ in $P$ such that $f_2 = g_2(T(\alpha)z_1)$. With the compositions in $1 // T$ inherited in the obvious way from $P$, $(x,z) \mapsto x$ becomes the object map of a 2-functor $1 // T \to P$.

It is useful to picture a typical morphism of $1 \dbslash S$ and a typical morphism of $1 // T$ as
\[ \xygraph{{\xybox{\xygraph{{1}="p0" [u(.5)r] {Sx_1}="p1" [d] {Sx_2}="p2"
"p0":"p1"^-{y_1}:@{<-}"p2"^-{Sf_1}:@{<-}"p0"^-{y_2}
"p0" [r(.5)u(.15)] :@{=>}[d(.3)]^{f_2}}}}
[r(3)]
{\xybox{\xygraph{{1}="p0" [u(.5)r] {Tx_1}="p1" [d] {Tx_2}="p2"
"p0":"p1"^-{z_1}:"p2"^-{Tf_1}:@{<-}"p0"^-{z_2}
"p0" [r(.5)u(.15)] :@{=>}[d(.3)]^{f_2}}}}} \]
respectively. The vertex of the universal lax wedge we are in the process of describing is the 2-category $(1 \dbslash S) \times_P (1 // T)$. An object of this 2-category will be denoted as $(x,y,z)$ where $(x,y) \in 1 \dbslash S$ and $(x,z) \in 1 // T$, an arrow will be denoted as $(f,f_1,f_2) : (x_1,y_1,z_1) \to (x_2,y_2,z_2)$ where $(f,f_1):(x_1,y_1) \to (x_2,y_2) \in 1 \dbslash S$ and $(f,f_2):(x_1,z_1) \to (x_2,z_2) \in 1 // T$, and a 2-cell $\alpha:(f,f_1,f_2) \to (g,g_1,g_2)$ is by definition a 2-cell $\alpha:f \to g$ in $P$ such that $(S(\alpha)y_2)f_1 = g_1$ and $f_2 = g_2(T(\alpha)z_1)$.
\begin{const}\label{const:universal-lax-wedge}
We now describe a lax wedge $\kappa$ for $d_*S \times d_*T$ with vertex $(1 \dbslash S) \times_P (1 // T)$. For $x \in P$ we define $\kappa_x : Sx \times Tx \to (1 \dbslash S) \times_P (1 // T)$ by
\[ \begin{array}{lccr} {\kappa_x(y,z) = (x,y,z)} &&& {\kappa_x(f_1,f_2) = (1_x,f_1,f_2)} \end{array} \]
where $y$ and $f_1:y_1 \to y_2$ are in $Sx$, and $z$ and $f_2:z_1 \to z_2$ are in $Tx$. For $f:x_1 \to x_2$ in $P$ we define $\kappa_f:\kappa_{x_1}(Sf \times 1) \to \kappa_{x_2}(1 \times Tf)$ as
\[ (\kappa_f)_{(y,z)} = (f,1_{Sfy},1_{Tfz}) : (x_1,Sfy,z) \to (x_2,y,Tfz) \]
where $y \in Sx_2$ and $z \in Tx_1$. For $\alpha:f \to g$ in $P$ we must give a modification as on the left
\[ \xygraph{{\xybox{\xygraph{!{0;(3,0):(0,.3333)::} {\kappa_{x_1}(Sf \times 1)}="p0" [r] {\kappa_{x_1}(Sg \times 1)}="p1" [d] {\kappa_{x_2}(1 \times Tg)}="p2" [l] {\kappa_{x_2}(1 \times Tf)}="p3" "p0":"p1"^-{\kappa_{x_1}(S\alpha \times \id)}:"p2"^-{\kappa_g}:@{<-}"p3"^-{\kappa_{x_d}(\id \times T\alpha)}:@{<-}"p0"^-{\kappa_f} "p0" [d(.55)r(.425)] :@{=>}[r(.15)]^-{\kappa_{\alpha}}}}}
[r(6)]
{\xybox{\xygraph{!{0;(3,0):(0,.3333)::} {(x_1,Sfy,z)}="p0" [r] {(x_1,Sgy,z)}="p1" [d] {(x_2,y,Tgz)}="p2" [l] {(x_2,y,Tfz)}="p3" "p0":"p1"^-{(1_{x_1},S(\alpha)_y,1_z)}:"p2"^-{(g,1_{Sgy},1_{Tgz})}:@{<-}"p3"^-{(1_{x_2},1_y,T(\alpha)_z)}:@{<-}"p0"^-{(f,1_{Sfy},1_{Tfz})} "p0" [d(.55)r(.425)] :@{=>}[r(.15)]^-{\alpha}}}}} \]
and for $y \in Sx_2$ and $z \in Tx_1$, the component $(\kappa_{\alpha})_{(y,z)}$ is as indicated on the right in the previous display. It is straight forward to verify that $\kappa_x$ is a functor, $\kappa_f$ and $\kappa_{\alpha}$ are natural in the required senses, and that the lax wedge axioms are satisfied.
\end{const}
\begin{prop}\label{prop:col-caH-ScrossT}
The lax wedge $\kappa$ of Construction \ref{const:universal-lax-wedge} exhibits
\[ \tn{col}(\ca H,d_*S \times d_*T) = (1 \dbslash S) \times_P (1 // T). \]
\end{prop}
\begin{proof}
Since $\TwoCat$ admits all cotensors as a $\TwoCat$-enriched category, it suffices to show, by Lemma \ref{lem:lax-wedge-equals-caH-cocone}, that for any 2-category $X$ and any lax wedge $\psi$ with vertex $X$, there exists a unique 2-functor $\overline{\psi}:(1 \dbslash S) \times_P (1 // T) \to X$ such that $\overline{\psi}\kappa = \psi$. For $x \in P$, the equation $\overline{\psi}\kappa_x = \psi_x$ forces $\overline{\psi}(x,y,z) = \psi_{x}(y,z)$. Observe that an arrow $(f,f_1,f_2) : (x_1,y_1,z_1) \to (x_2,y_2,z_2)$ in $(1 \dbslash S) \times_P (1 // T)$ can be factored as
\[ (x_1,y_1,z_1) \xrightarrow{(1,f_1,1)} (x_1,Sfy_2,z_1) \xrightarrow{(f,1,1,)} (x_2,y_2,Tfz_1) \xrightarrow{(1,1,f_2)} (x_2,y_2,z_2) \]
and since
\[ \begin{array}{lcccr} {(1,f_1,1) = \kappa_{x_1}(f_1,1)} && {(f,1,1) = (\kappa_f)_{(y_2,z_1)}} && {(1,1,f_2) = \kappa_{x_2}(1,f_2)} \end{array} \]
the equation $\overline{\psi}\kappa = \psi$ forces us to define
\[ \overline{\psi}(f,f_1,f_2) = \psi_{x_2}(1_{y_2},f_2)(\psi_f)_{(y_2,z_1)}\psi_{x_1}(f_1,z_1). \]
Let $\alpha:(f,f_1,f_2) \to (g,g_1,g_2)$ be a 2-cell in $(1 \dbslash S) \times_P (1 // T)$. Then this 2-cell can be decomposed in $(1 \dbslash S) \times_P (1 // T)$ in the following way
\[ \xygraph{!{0;(3,0):(0,.3333)::} {(x_1,y_1,z_1)}="p0" [r] {(x_1,Sfy_2,z_1)}="p1" [r] {(x_2,y_2,Tfz_1)}="p2" [dr] {(x_2,y_2,z_2)}="p3" [l] {(x_2,y_2,Tgz_1)}="p4" [l] {(x_1,Sgy_2,z_1)}="p5" "p0":"p1"^-{(1,f_1,1)}:"p2"^-{(f,1,1)}:@{}"p3":@{<-}"p4"^-{(1,1,g_2)}:@{<-}"p5"^-{(g,1,1)}:@{}"p0"
"p1":"p5"_-{(1,S(\alpha)_{y_2},1)} "p2":"p4"^-{(1,1,Tf(\alpha)_{z_1})}
"p1" [d(.35)r(.5)] :@{=>}[d(.3)]^{\alpha}} \]
and so the equation $\overline{\psi}\kappa = \psi$ forces $\overline{\psi}(\alpha)$ to be the composite
\[ \xygraph{!{0;(3,0):(0,.3333)::} {\psi_{x_1}(y_1,z_1)}="p0" [r] {\psi_{x_2}(Sfy_2,z_1)}="p1" [r] {\psi_{x_2}(y_2,Tfz_1)}="p2" [dr] {\psi_{x_2}(y_2,z_2).}="p3" [l] {\psi_{x_2}(y_2,Tgz_1)}="p4" [l] {\psi_{x_1}(Sgy_2,z_1)}="p5" "p0":"p1"^-{\psi_{x_1}(f_1,1)}:"p2"^-{(\psi_f)_{(y_2,z_1)}}:@{}"p3":@{<-}"p4"^-{\psi_{x_2}(1,g_2)}:@{<-}"p5"^-{(\psi_g)_{(y_2,z_1)}}:@{}"p0"
"p1":"p5"_-{\psi_{x_1}(S(\alpha)_{y_2},1)} "p2":"p4"^-{\psi_{x_2}(1,T(\alpha)_{z_1})}
"p1" [d(.35)r(.35)] :@{=>}[d(.3)]^{(\psi_{\alpha})_{(y_2,z_1)}}} \]
So if $\overline{\psi}$ exists then it is the unique 2-functor such that $\overline{\psi}\kappa = \psi$. Thus it suffices to show that if $\overline{\psi}$ is defined in this way, then it is indeed a 2-functor, and this 2-functoriality follows easily from the lax wedge axioms for $\psi$.
\end{proof}
Putting this result together with Propositions \ref{prop:H-dagger} and \ref{prop:weights-of-the-form-pi0star-2catwt} we obtain
\begin{cor}\label{cor:lax-coend-formula}
Let $P$ be a 2-category and $S:P^{\op} \to \Cat$ and $T:P \to \Cat$ be 2-functors. Then one has
\[ \tn{col}(H^{\dagger}_L,S \times T) = \pi_{0*}((1 \dbslash S) \times_P (1 // T)). \]
\end{cor}
\begin{rem}\label{rem:set-case-consequence}
Returning to the situation of Lemma \ref{lem:coend-for-exactness-in-Cat}, in which $P$ is a mere category, and $S$ and $T$ are the discrete-valued $B(q-,b)$ and $A(a,p-)$ respectively. Then in this case $1 \dbslash S = q \downarrow b$, $1 // T = a \downarrow p$, and since these are locally discrete as 2-categories we have
\[ \pi_{0*}((1 \dbslash S) \times_P (1 // T)) = (1 \dbslash S) \times_P (1 // T) = (q \downarrow b) \times_P (a \downarrow p).  \]
By Corollary \ref{cor:pi0-inverts-laxcoend->coend} the canonical comparison functor
\[ (q \downarrow b) \times_P (a \downarrow p) \longrightarrow \int^{x \in P} B(qx,b) \times A(a,px)  \]
from the lax to the strict coend is inverted by $\pi_0$, and so we recover Lemma \ref{lem:coend-for-exactness-in-Cat} since the codomain of this comparison functor is discrete.
\end{rem}

\subsection{Proof of Theorem \ref{thm:hard-exactness}}
\label{ssec:exactness-of-internalisations}
In this section we return to the situation of Lemma \ref{cor:pi0-exact-via-2-coends}, and in the light of the developments of Sections \ref{ssec:lax-coends} and \ref{ssec:lax-wedges}, give a combinatorial characterisation of $\pi_0$-exact squares in $\TwoCat$. We then reformulate this characterisation in the case where the square in question is a pullback square, in terms of the 2-functors which generate the pullback. This last characterisation is then applied to give the proof of Theorem \ref{thm:hard-exactness}.

We are denoting by $\tn{obj}:\Cat \to \Set$ the functor which forgets the arrows of a category, and so $\tn{obj}_*:\TwoCat \to \Cat$, which on objects is the process of applying $\tn{obj}$ to the homs of a 2-category, is perhaps more plainly described as the process of forgetting 2-cells. From the adjunctions $\pi_0 \ladj d \ladj \tn{obj}$, for any category $X$ one obtains the function $c_X:\tn{obj}(X) \to \pi_0(X)$ naturally in $X$, which explicitly is given by associating to any object of $X$ its connected component. Thus for any 2-category $X$, $c_{X*}:\tn{obj}_*X \to \pi_{0*}X$ is the identity on objects functor which sends 1-cells of $X$ to their connected components in the appropriate hom-category of $X$.
\begin{lem}\label{lem:pi0-inverts-objstar->pi0star}
For any 2-category $X$, $c_{X*}$ is inverted by $\pi_0$.
\end{lem}
\begin{proof}
Since $c_{X*}$ is the identity on objects, $\pi_0c_{X*}$ is clearly surjective. To say that objects $x$ and $y$ of $X$ are identified by $\pi_0c_{X*}$, is to say that $x$ and $y$ are in the same connected component of $\pi_{0*}X$. Thus one has an undirected path
\[ x \xrightarrow{p_1} z_1 \xleftarrow{} ... \xrightarrow{} z_n \xleftarrow{p_{n+1}} y \]
of equivalence classes of arrows of $X$ by the connectedness relation in the appropriate homs. Taking $q_i \in p_i$ for $1 \leq i \leq n+1$, that is choosing inhabitants of these equivalence classes of arrows, gives an undirected path in $\tn{obj}_*X$ between $x$ and $y$ whence $\pi_0c_{X*}$ is injective.
\end{proof}
Recall the setting of Lemma \ref{cor:pi0-exact-via-2-coends} is that of a lax square
\[ \xygraph{!{0;(1.5,0):(0,.6667)::} {P}="p0" [r] {B}="p1" [d] {C}="p2" [l] {A}="p3" "p0":"p1"^-{q}:"p2"^-{g}:@{<-}"p3"^-{f}:@{<-}"p0"^-{p} "p0" [d(.55)r(.4)] :@{=>}[r(.2)]^-{\phi}} \]
in $\TwoCat$. For $a \in A$ and $b \in B$ writing
\[ E = (1 \dbslash B(q-,b)) \times_P (1 // A(a,p-))  \]
one has functors
\[ \tn{obj}_*E \xrightarrow{c_E} \pi_{0*}E \xrightarrow{\tn{comp.}} {\int^{x \in P} B(qx,b) \times A(a,px)} \xrightarrow{\tilde{\phi}_{a,b}} C(fa,gb).   \]
By Lemma \ref{lem:pi0-inverts-objstar->pi0star}, Corollary \ref{cor:lax-coend-formula} and Corollary \ref{cor:pi0-inverts-laxcoend->coend}, $\pi_0$ inverts the first two of these functors, and so by Lemma \ref{cor:pi0-exact-via-2-coends} the given square is $\pi_0$-exact iff for all $a,b \in P$, the above composite functor is inverted by $\pi_0$. This is our combinatorial characterisation of $\pi_0$-exactness. It remains only to read off what this composite functor is explicitly to have a directly usable criterion.

Let us denote the above composite functor as
\[ \ca C(\phi,a,b) : \ca F(\phi,a,b) \longrightarrow C(fa,gb)  \]
so that in particular
\[ \ca F(\phi,a,b) = \tn{obj}_*((1 \dbslash B(q-,b)) \times_P (1 // A(a,p-))), \]
and when the context, that is to say $(\phi,a,b)$, is clear, we write this more simply as $\ca C : \ca F \to C(fa,gb)$. In elementary terms the category $\ca F$ is described as follows. An object is a triple $(x,y,z)$ where $x \in P$, $y:qx \to b$ and $z:a \to px$. A morphism $(x_1,y_1,z_1) \to (x_2,y_2,z_2)$ is a triple $(h,h_1,h_2)$ as in
\[ \xygraph{{\xybox{\xygraph{!{0;(1,0):(0,.5)::} {a}="p0" [ur] {px_1}="p1" [d(2)] {px_2}="p2" "p0":"p1"^-{z_1}:"p2"^-{ph}:@{<-}"p0"^-{z_2} "p0" [r(.5)u(.2)] :@{=>}[d(.4)]^-{h_2}}}}
[r(2)]
{\xybox{\xygraph{!{0;(1,0):(0,.5)::} {b}="p0" [ul] {qx_1}="p1" [d(2)] {qx_2}="p2" "p0":@{<-}"p1"_-{y_1}:"p2"_-{qh}:"p0"_-{z_2} "p0" [l(.5)u(.2)] :@{=>}[d(.4)]_-{h_1}}}}} \]
and composition is inherited in the evident way from $A$ and $B$. The functor $\ca C(\phi,a,b)$ is given on objects by $(x,y,z) \mapsto g(y)\phi_xf(z)$, and on arrows by
\[ \xygraph{{(h,h_1,h_2)} [r] {\mapsto} [r(2.5)]
{\xybox{\xygraph{!{0;(.75,0):(0,1)::} {fa}="p0" [r(1.5)u] {fpx_1}="p1" [r(2)] {gqx_1}="p2" [r(1.5)d] {gb.}="p3" [l(1.5)d] {gqx_2}="p4" [l(2)] {fpx_2}="p5" "p0":"p1"^-{fz_1}:"p2"^-{\phi_{x_1}}:"p3"^-{gy_1}:@{<-}"p4"^-{gy_2}:@{<-}"p5"^-{\phi_{x_2}}:@{<-}"p0"^-{fz_2} "p1":"p5"^-{fph} "p2":"p4"_-{gqh}
"p0" [r(.75)u(.2)] :@{=>}[d(.4)]^{fh_2} "p3" [l(.75)u(.2)] :@{=>}[d(.4)]_{gh_1} "p1" [dr] {\scriptstyle =}}}}} \]
To summarise, our combinatorial characterisation of $\pi_0$-exactness is
\begin{prop}\label{prop:pi0-exact-combinatorial-characterisation}
A lax square
\[ \xygraph{!{0;(1.5,0):(0,.6667)::} {P}="p0" [r] {B}="p1" [d] {C}="p2" [l] {A}="p3" "p0":"p1"^-{q}:"p2"^-{g}:@{<-}"p3"^-{f}:@{<-}"p0"^-{p} "p0" [d(.55)r(.4)] :@{=>}[r(.2)]^-{\phi}} \]
in $\TwoCat$ is $\pi_0$-exact iff for all $a \in A$ and $b \in B$, the functor $\ca C(\phi,a,b)$ described above in elementary terms is inverted by $\pi_0$.
\end{prop}
\begin{rem}\label{rem:inverted-by-pi0-explained}
Let $h:X \to Y$ be a functor. The condition that $\pi_0$ inverts $h$ amounts to the following conditions
\begin{enumerate}
\item For any $y \in Y$, there exists $x \in X$ and an undirected path in $Y$ between $hx$ and $y$.\label{pi0-surjective}
\item For all $x_1,x_2 \in X$, if there exists an undirected path between $hx_1$ and $hx_2$ in $Y$, then there exists an undirected path in $X$ between $x_1$ and $x_2$.\label{pi0-injective}
\end{enumerate}
Condition (\ref{pi0-surjective}) is the condition that $\pi_0h$ be surjective, and (\ref{pi0-injective}) is the condition that $\pi_0h$ be injective, unpacked in elementary terms. In the case where $h$ itself is surjective on objects, (\ref{pi0-surjective}) is automatic, and (\ref{pi0-injective}) clearly follows from
\begin{enumerate}
\setcounter{enumi}{2}
\item For all $x_1,x_2 \in X$, if there exists $hx_1 \to hx_2$ in $Y$, then there exists an undirected path in $X$ between $x_1$ and $x_2$.\label{pi0-injective-easy}
\end{enumerate}
\end{rem}
\begin{rem}\label{rem:recover-exactness-characterisation-in-Cat}
Given a lax square $\ca S$ in $\Cat$, one has a lax square $d_*\ca S$ in $\TwoCat$, and $\ca S$ is exact iff $d_*\ca S$ is $\pi_0$-exact. Thus the explicit characterisations of exact squares in $\Cat$ of Guitart \cite{Guitart-ExactSquares}, follow from Proposition \ref{prop:pi0-exact-combinatorial-characterisation} and Remark \ref{rem:inverted-by-pi0-explained}.
\end{rem}
In the case where $\phi$ is an identity and the commuting square is a pullback square, we denote $\ca C(\phi,a,b)$ and $\ca F(\phi,a,b)$ as $\ca C(f,g,a,b)$ and $\ca F(f,g,a,b)$, and as before denote these as $\ca C$ and $\ca F$ when $(f,g,a,b)$ are understood. An object of $\ca F$ in this case is a pair $(y,z)$ where $y:b_1 \to b$ is in $B$ and  $z:a \to a_1$ is in $A$, such that $fa_1 = gb_1$, and in diagramatic terms, we write such an object as
\[ fa \xrightarrow{fz} fa_1 = gb_1 \xrightarrow{gy} gb. \]
The effect of $\ca C$ on $(y,z)$ is just this composite $(gy)(fz)$. A morphism $(y_1,z_1) \to (y_2,z_2)$ consists of $(\beta,\delta,\alpha,\varepsilon)$ as depicted in
\[ \xygraph{!{0;(2,0):(0,.375)::} {fa}="p0" [ur] {fa_1=gb_1}="p1" [dr] {gb.}="p2" [dl] {fa_2=gb_2}="p3" "p0":"p1"^-{fz_1}:"p2"^-{gy_1}:@{<-}"p3"^-{gy_2}:@{<-}"p0"^-{fz_2} "p1":"p3"|-{f\alpha=g\beta}
"p0" [r(.5)u(.2)] :@{=>}[d(.4)]^{f\varepsilon}
"p2" [l(.5)u(.2)] :@{=>}[d(.4)]_{g\delta}} \]
and $\ca C(\beta,\delta,\alpha,\varepsilon)$ is the composite 2-cell $(g(\delta)f(z_1))(g(y_2)f(\varepsilon))$ in $C$. We summarise this special case in
\begin{cor}\label{cor:pi0-exact-pb}
Let $f:A \to C$ and $g:B \to C$ be 2-functors. Then the pullback square
\[ \xygraph{{P}="p0" [r] {B}="p1" [d] {C}="p2" [l] {A}="p3" "p0":"p1"^-{}:"p2"^-{g}:@{<-}"p3"^-{f}:@{<-}"p0"^-{}:@{}"p2"|-{\tn{pb}}} \]
is $\pi_0$-exact iff for all $a \in A$ and $b \in B$, the functor $\ca C(f,g,a,b)$ described above is inverted by $\pi_0$.
\end{cor}
Recall that the situation of Theorem \ref{thm:hard-exactness} is that of a pullback square
\[ \xygraph{{P}="p0" [r] {B}="p1" [d] {C}="p2" [l] {A}="p3" "p0":"p1"^-{q}:"p2"^-{g}:@{<-}"p3"^-{f}:@{<-}"p0"^-{p}:@{}"p2"|-{\tn{pb}}} \]
in $\CrIntCat {\Cat}$ in which $g$ is a discrete fibration and $f$ is an objectwise opfibration. These hypotheses were described in elementary terms at the end of Section \ref{ssec:codesc-crossed-dblcat}. By Remark \ref{rem:exactness-via-Cnr} to prove  Theorem \ref{thm:hard-exactness}, it suffices to show that $\Cnr(f)$ and $\Cnr(g)$ satisfy the conditions of Corollary \ref{cor:pi0-exact-pb}. This means that for $a \in A$ and $b \in B$, we must show that the functor
\begin{equation}\label{eq:corner-indicator}
\ca C(\Cnr(f),\Cnr(g),a,b) : \ca F(\Cnr(f),\Cnr(g),a,b) \longrightarrow \Cnr(C)(fa,gb)
\end{equation}
is inverted by $\pi_0$.

By Corollary \ref{cor:pi0-exact-pb} and the definition of the 2-categories of corners, an object of $\ca F = \ca F(\Cnr(f),\Cnr(g),a,b)$ consists of $((v_2,h_2),(v_1,h_1))$ as on the left
\[ \xygraph{{\xybox{\xygraph{!{0;(1.7,0):(0,.2)::} {fa}="p0" [r] {fa_1=gb_1}="p1" [r] {gb}="p2" "p0":"p1"^-{(fv_1,fh_1)}:"p2"^-{(gv_2,gh_2)}}}}
[r(4)]
{\xybox{\xygraph{{a}="p0" [d] {i}="p1" [r] {a_1}="p2" "p0":"p1"^-{v_1}:"p2"^-{h_1}}}}
[r(2)]
{\xybox{\xygraph{{b_1}="p0" [d] {j}="p1" [r] {b_1}="p2" "p0":"p1"^-{v_2}:"p2"^-{h_2}}}}} \]
in $\Cnr(C)$, the data of which is depicted in double categorical terms on the right. A morphism $((v_2,h_2),(v_1,h_1)) \to ((v_4,h_4),(v_3,h_3))$ in $\ca F$ consists of the data depicted in
\[ \xygraph{!{0;(4,0):(0,.18)::} {fa}="p0" [ur] {fa_1=gb_1}="p1" [dr] {gb.}="p2" [dl] {fa_2=gb_2}="p3" "p0":"p1"^-{(fv_1,fh_1)}:"p2"^-{(gv_2,gh_2)}:@{<-}"p3"^-{(gv_4,gh_4)}:@{<-}"p0"^-{(fv_3,fh_3)} "p1":"p3"|-{(fv_5,fh_5)=(gv_6,gh_6)}
"p0" [r(.6)u(.2)] :@{=>}[d(.4)]_{(f\phi_1,f\phi_2)}
"p2" [l(.6)u(.2)] :@{=>}[d(.4)]^{(g\phi_3,g\phi_4)}} \]
Recall that the effect of $\ca C$ is to send such data to the composite 1 or 2-cells they describe in $\Cnr(C)$.
\begin{lem}\label{lem:C-fibres-for-hard-exactness}
Suppose that $f:A \to C$ and $g:B \to C$ are morphisms of $\CrIntCat {\Cat}$, $g$ is a discrete fibration, $f$ is an objectwise opfibration, $a \in A$ and $b \in B$. Then any object $((v_2,h_2),(v_1,h_1))$ of $\ca F$ is in the same connected component of its fibre by $\ca C$ as one in which $v_1$ is $f_0$-opcartesian, $h_1$ is a horizontal identity and $v_2$ is a vertical identity.
\end{lem}
\begin{proof}
We begin by taking the chosen opcartesian square $\kappa_1$ as on the left in
\[ \xygraph{{\xybox{\xygraph{!{0;(1.5,0):(0,.6667)::} {fi_1}="p0" [r] {fa_1=gb_1}="p1" [d] {gj_1}="p2" [l] {k_1}="p3" "p0":"p1"^-{fh_1}:"p2"^-{gv_2}:@{<-}"p3"^-{h_3}:@{<-}"p0"^-{v_3}:@{}"p2"|-*{\kappa_1}}}}
[r(3)]
{\xybox{\xygraph{{i_1}="p0" [r] {a_1}="p1" [d] {i_2}="p2" [l] {i_3}="p3" "p0":"p1"^-{h_1}:"p2"^-{v_4}:@{<-}"p3"^-{h_4}:@{<-}"p0"^-{v_5}:@{}"p2"|-*{\kappa_2}}}}
[r(2.5)]
{\xybox{\xygraph{{j_2}="p0" [r] {b_1}="p1" [d] {j_1}="p2" [l] {j_3}="p3" "p0":"p1"^-{h_5}:"p2"^-{v_2}:@{<-}"p3"^-{h_6}:@{<-}"p0"^-{v_6}:@{}"p2"|-*{\kappa_3}}}}} \]
Take an $f_0$-opcartesian lift $v_4:a_1 \to i_2$ of $gv_2$. Then take the chosen opcartesian square $\kappa_2$ in $A$, and since $f$ is a crossed internal functor $f\kappa_2=\kappa_1$. Define $h_5:j_2 \to b_1$ as the unique horizontal arrow in $B$ such that $gh_5=fh_1$. Then take the chosen opcartesian square $\kappa_3$ in $B$ as in the previous display, and since $g$ is a crossed internal functor $g\kappa_3=\kappa_1$. Next take an $f_0$-opcartesian lift $v_7:a \to i_4$ of $f(v_5v_1)$, and so one has $v_8:i_4 \to i_3$ unique such that $v_8v_7=v_5v_1$ and $fv_8=1_{k_1}$. The diagram
\[ \xygraph{!{0;(4,0):(0,.5)::}
{fa}="p0" [ur] {fa_1=gb_1}="p1" [dr] {gb}="p2" [dl] {fi_4=gj_3}="p3" [u] {fi_2=gj_1}="p4"
"p0":"p1"^-{(fv_1,fh_1)}:"p2"^-{(gv_2,gh_2)}:@{<-}"p3"^-{(g1_{j_3},g(h_2h_6))}:@{<-}"p0"^-{(fv_7,f1_{i_4})}
"p4" (:@{<-}"p0"_-{(f(v_5v_1),fh_5)}, :@{<-}"p1"|(.6){(fv_4,f1_{i_2})=(gv_2,g1_{j_1})}, :"p2"^-{(g1_{j_1},gh_2)}, :@{<-}"p3"|(.6){(fv_8,fh_4)=(g1_{j_3},gh_6)})
"p4" ([u(.4)l(.3)] :@{=>}[d(.15)]^-{f(\id)}, [u(.4)r(.3)] :@{=>}[d(.15)]^-{g(\id)}, [d(.4)l(.3)] :@{=>}[u(.15)]_-{f(\id)}, [d(.4)r(.3)] :@{=>}[u(.15)]_-{g(\id)})} \]
exhibits the object $((v_2,h_2),(v_1,h_1))$ of $\ca F$ as being in the same connected component of its fibre as an object of the required form.
\end{proof}
\begin{proof}
(\emph{of Theorem \ref{thm:hard-exactness}}).
Our task is to show that for $a \in A$ and $b \in B$, the functor (\ref{eq:corner-indicator}) satisfies the conditions described in Remark \ref{rem:inverted-by-pi0-explained}. Recall that in terms of the double category $C$, a morphism $fa \to gb$ in $\Cnr(C)$ is a pair $(v,h)$, where $v:fa \to c$ is a vertical arrow, and $h:c \to gb$ is a horizontal arrow. Since $f_0$ is an opfibration there is a vertical arrow $u:a \to a_1$ in $A$ such that $fu = v$. Since $g$ is a discrete fibration there is a (unique) horizontal arrow $k:b_2 \to b$ such that $gk = h$. The functor (\ref{eq:corner-indicator}) sends $((1_{b_2},k),(u,1_{a_1}))$ to $(v,h)$, and so we have verified condition(\ref{pi0-surjective}) of Remark \ref{rem:inverted-by-pi0-explained}. In fact in this case the functor (\ref{eq:corner-indicator}) itself is surjective on objects.

It remains to verify condition (\ref{pi0-injective-easy}) of Remark \ref{rem:inverted-by-pi0-explained}. In light of Lemma \ref{lem:C-fibres-for-hard-exactness} it suffices to verify that given 
\[ \xygraph{*=(0,1.2)!(0,-1.2){\xybox{\xygraph{{x_1}="q0" [r(.35)] {=} [r(2.25)u(.1)]
{\xybox{\xygraph{!{0;(1.7,0):(0,.1)::} {fa}="p0" [r] {fa_1=gb_1}="p1" [r] {gb}="p2" "p0":"p1"^-{(fv_1,f1_{a_1})}:"p2"^-{(g1_{b_1},gh_1)}}}}*\frm{-}
"q0" [r(6)] {x_2} [r(.35)] {=} [r(2.25)u(.1)]
{\xybox{\xygraph{!{0;(1.7,0):(0,.1)::} {fa}="p0" [r] {fa_2=gb_2}="p1" [r] {gb}="p2" "p0":"p1"^-{(fv_2,f1_{a_2})}:"p2"^-{(g1_{b_2},gh_2)}}}}*\frm{-}}}}} \]
in $\ca F$ where $v_1$ and $v_2$ are $f_0$-opcartesian, and
\[ \xygraph{!{0;(2,0):(0,.25)::}
{fa}="p0" [ur] {fa_1=gb_1}="p1" [dr] {gb}="p2" [dl] {fa_2=gb_2}="p3" "p0":"p1"^-{(fv_1,f1_{a_1})}:"p2"^-{(g1_{b_1},gh_1)}:@{<-}"p3"^-{(g1_{b_2},gh_2)}:@{<-}"p0"^-{(fv_2,f1_{a_2})} "p0" [u(.3)r(.9)] :@{=>}[d(.6)]^-{(\phi_1,\phi_2)}} \]
in $\Cnr(C)$, then $x_1$ and $x_2$ are in the same connected component of $\ca F$.

In double categorical terms we have $v_1$, $v_2$, $h_1$, $h_2$, $\phi_1$ and $\phi_2$ as in
\[ \xygraph{{a}="p0" :[d] {a_1}_-{v_1} "p0" [r] {a}="p1" :[d] {a_2}_-{v_2} "p1" [r] {b_1} :[r] {b}^-{h_1} [dl] {b_2} :[r] {b}^-{h_2} [ur] {gb_1}="p2" :[r] {gb}^-{gh_1} :[d] {gb}="p3"^-{g1_b} :@{<-}[l] {gb_2}^-{gh_2} :@{<-}"p2"^-{\phi_1}:@{}"p3"|-*{\phi_2} [ur] {b_1}="p4" :[r] {b}^-{h_1} :[d] {b}="p5"^-{1_b} :@{<-}[l] {b_2}^-{h_2} :@{<-}"p4"^-{\phi_3}:@{}"p5"|-*{\phi_4}} \]
such that $\phi_1f(v_1) = fv_2$. Since $g$ is a discrete fibration there is a unique square $\phi_4$ as on the right in the previous display, and by the uniqueness of lifts of horizontal arrows, the source and target horizontal arrows of $\phi_4$ must be $h_1$ and $h_2$ respectively as indicated. Since $\phi_1f(v_1)=fv_2$ and $v_1$ is $f_0$-opcartesian, there is a unique vertical arrow $v_3:a_1 \to a_2$ such that $fv_3=\phi_1$ and $v_3v_1=v_2$. Using all this information, the diagram
\[ \xygraph{!{0;(4,0):(0,.18)::} {fa}="p0" [ur] {fa_1=gb_1}="p1" [dr] {gb}="p2" [dl] {fa_2=gb_2}="p3" "p0":"p1"^-{(fv_1,f1_{a_1})}:"p2"^-{(g1_{b_1},gh_1)}:@{<-}"p3"^-{(g1_{b_2},gh_2)}:@{<-}"p0"^-{(fv_2,f1_{a_2})} "p1":"p3"|-{(fv_3,f1_{a_2})=(g\phi_3,g1_{b_2})}
"p0" [r(.6)u(.2)] :@{=>}[d(.4)]_{(f(\id),f(\id))}
"p2" [l(.6)u(.2)] :@{=>}[d(.4)]^{(g\phi_3,g\phi_4)}} \]
exhibits an arrow in $\ca F$ between $x_1$ and $x_2$.
\end{proof}

\subsection{$T^G$ via codescent}
\label{ssec:int-alg-class-codesc}
Given a 2-monad $(\ca K,T)$ and a strict $T$-algebra $(X,x)$, as in Section 4.3 of \cite{Weber-CodescCrIntCat}, we denote by $\ca R_TX$ the standard simplicial $T$-algebra whose 2-truncation is
\[ \xygraph{!{0;(2,0):(0,1)::} {T^3X}="p0" [r] {T^2X}="p1" [r] {TX.}="p2"
"p2":"p1"|-{T\eta^T_X} "p1":@<1.5ex>"p2"^-{\mu^T_X} "p1":@<-1.5ex>"p2"_-{Tx} "p0":@<1.5ex>"p1"^-{\mu^T_{TX}} "p0":"p1"|-{T\mu^T_X} "p0":@<-1.5ex>"p1"_-{T^2x}} \]
From the data of the adjunction $F : (\ca L,S) \to (\ca K,T)$ of 2-monads one has, for each strict $S$-algebra $(X,x)$, a simplicial object $\ca R_FX : \Delta^{\op} \to \Algs T$ defined as in Construction 4.3.1 \cite{Weber-CodescCrIntCat} by the equations on the left for $n \in \N$, and the isomorphisms on the right
\[ \begin{array}{lccr} {(\ca R_FX)_n = TF_!S^nX} &&& {\Algs S(\ca R_SX,\overline{F}Y) \iso \Algs T(\ca R_FX,Y)} \end{array} \]
which are 2-natural in $X$ and $Y$. From this abstract definition one may deduce, as in
Lemma 4.3.2 of \cite{Weber-CodescCrIntCat}, that the face and degeneracy maps of $\ca R_FX$ are given by the formulae
\[ \begin{array}{lccr}
{d_i^{n+1} = \left\{\begin{array}{lll}
{TF_!S^nx} && {i = 0} \\
{TF_!S^{n-i}\mu^S_{S^{i-1}X}} && {1 \leq i \leq n} \\
{\mu^T_{F_!S^nX}T(F^c_{S^nX})} && {i = n+1}
\end{array}\right.}
&&& {s_i^{n+1} = TF_!S^{n-i}\eta^S_{S^iX}.}
\end{array} \]
In particular $\ca R_T = \ca R_{1_T}$. Proposition 4.3.3 of \cite{Weber-CodescCrIntCat} says that when $\Algs T$ has codescent objects, the left adjoint $(-)^{\dagger}_F$ to $J_F$ exists and is given on objects by the formula on the left
\begin{equation}\label{eq:dagger-F-codesc-formula}
\begin{array}{lccr}
{X^{\dagger}_F = \tn{CoDesc}(\ca R_FX)} &&& {T^S = \tn{CoDesc}(\ca R_F1)}
\end{array}
\end{equation}
which in particular gives the formula on the right when $\ca L$ has a terminal object $1$.

We now extend this to exhibit $T^G$ as the result of applying
\[ \tn{CoDesc} : [\Delta^{\op},\Algs T] \longrightarrow \Algs T \]
to a morphism of simplicial $T$-algebras. As recalled in Remark \ref{rem:T^G-via-units}, the unit of the adjunction $(-)^{\dagger}_F \ladj J_F$ is denoted as $g^F$, and in the context of a morphism $G : H \to F$ of internalisable adjunctions, $g^F_1$ and $g^H_1$ are the universal lax morphisms
\[ \begin{array}{lccr} {g^S_T : 1 \longrightarrow \overline{F}T^S} &&& {g^R_T : 1 \longrightarrow \overline{H}T^R.} \end{array} \]
Given $(X,x) \in \Algs S$, by the universal property of the unit $g^H$, one has a unique strict $T$-morphism $G^{\dagger}_X$ making
\[ \xygraph{!{0;(2,0):(0,.5)::} {\overline{G}X}="p0" [r] {\overline{H}(\overline{G}X)^{\dagger}_H}="p1" [d] {\overline{H}X^{\dagger}_F}="p2" [l] {\overline{G}\overline{F}X^{\dagger}_F}="p3" "p0":"p1"^-{g^H_{\overline{G}X}}:"p2"^-{\overline{H}G^{\dagger}_X}:@{<-}"p3"^-{\overline{\gamma}_{X^{\dagger}_F}}:@{<-}"p0"^-{\overline{G}g^F_X}} \]
commute in $\Algl R$. When $X = 1$ one has $G^{\dagger}_X = T^G$ by Remark \ref{rem:T^G-via-units}. We shall now explain how $G^{\dagger}_X$ may be obtained as the result of applying $\tn{CoDesc}$ to a morphism of simplicial $T$-algebras. This morphism of simplicial $T$-algebras is provided by
\begin{const}\label{const:simplicial-morphism-underlying-TG}
Let $H : (\ca M,R) \to (\ca K,T)$ and $F : (\ca L,S) \to (\ca K,T)$ be internalisable adjunctions and $G:H \to F$ be a morphism thereof. Then we define
\[ (\overline{\ca R}_G)_X : \ca R_H\overline{G}X \longrightarrow \ca R_FX \]
2-naturally in $X \in \Algs S$, as unique such that
\[ \xygraph{!{0;(5,0):(0,.2)::} {\Algs T(\ca R_FX,\overline{F}Y)}="p0" [r(2)] {\Algs T(\ca R_H\overline{G}X,Y)}="p1" [d] {\Algs R(\ca R_R\overline{G}X,\overline{H}Y)}="p2" [l] {\Algs R(\overline{G}\ca R_SX,\overline{G}\overline{F}Y)}="p3" [l] {\Algs S(\ca R_SX,\overline{F}Y)}="p4"
"p0":"p1"^-{\Algs T(\ca R_FX,Y)}:"p2"^-{}:@{<-}"p3"^-{\Algs R(G^l_X,\overline{\gamma}_Y)}:@{<-}"p4"^-{\overline{G}_{\ca R_SX,\overline{F}Y}}:@{<-}"p0"^-{}} \]
commutes for all $Y \in \Algs T$, in which the vertical arrows are the isomorphisms recalled above.
\end{const}
Next we give an explicit description of the components of $(\overline{\ca R}_G)_{X,n}$. First we require some preliminary notation. Given an adjunction of 2-monads $F : (\ca L,S) \to (\ca K,T)$ one has 2-natural transformations $F^c : F_!S \to TF_!$ and $F^l : SF^* \to F^*T$ making $(F_!,F^c)$ and $(F^*,F^l)$ colax and lax monad morphisms respectively. In fact for each $n \in \N$ one has $F^c_n : F_!S^n \to T^nF_!$ and $F^l_n : S^nF^* \to F^*T^n$ defined inductively by the formulae
\[ \begin{array}{ccccccc} {F^c_0 = \id_{F_!}} && {F^c_{n+1} = (T^nF^c)(F^c_nS)} && {F^l_0 = \id_{F^*}} && {F^l_{n+1} = (F^lT^n)(SF^l_n)} \end{array} \]
so that in particular $F^c_1 = F^c$ and $F^l_n = F^l$.
\begin{lem}\label{lem:explicit-RG}
The component $(\overline{\ca R}_G)_{X,n}$ defined by Construction \ref{const:simplicial-morphism-underlying-TG} is given by the composite
\[ {TH_!R^nG^*X} \xrightarrow{TF_!((G^c_n)_{G^*X})} {TF_!S^nG_!G^*X} \xrightarrow{TF_!S^n\varepsilon^G_X} {TF_!S^nX.} \]
\end{lem}
\begin{proof}
We denote by
\[ \varphi_{X,Y,n}^F : \Algs S(S^{n+1}X,\overline{F}Y) \longrightarrow \Algs T(TF_!S^nX,Y) \]
the components of the defining isomorphism of $\ca R_F$. In the proof of Lemma 4.3.2 of \cite{Weber-CodescCrIntCat}, this was denoted as $\varphi_{X,Y,n}$ and both $\varphi_{X,Y,n}^F$ and its inverse were described explicitly. We shall use these details here. By definition an explicit description of $(\overline{\ca R}_G)_{X,n}$ is obtained by tracing through the effect of the composite
\[ \varphi^H_{G^*X,Y,n} \comp \Algs R(G^l_{n+1,X},\overline{\gamma}_Y) \comp \overline{G}_{S^{n+1}X,\overline{F}Y} \comp (\varphi^F_{X,Y,n})^{-1} \]
on $1_{TF_!S^nX}$ (in the case $Y = TF_!S^nX$). By the formula for $(\varphi^F_{X,Y,n})^{-1}$, the morphism $(\varphi^F_{X,Y,n})^{-1}(1_{TF_!S^nX})$ is the component at $X$ of the composite
\[ (F^*\mu^TF_!S^n)(F^lTF_!S^n)(SF^*\eta^TF_!S^n)(S\eta^FS^n) \]
which in terms of the string diagrams{\footnotemark{\footnotetext{The string diagrams in this work go from top to bottom.}}} of \cite{JoyalStreet-GeometryTensorCalculus} can be written as on the left of
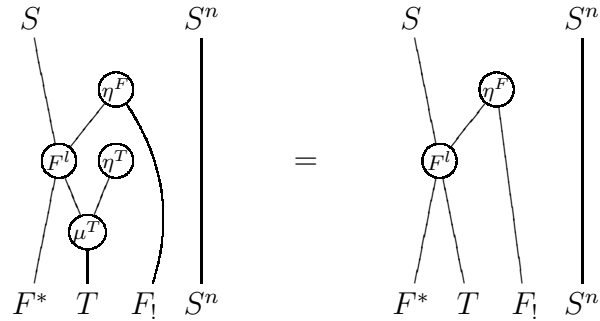
\[ \xygraph{!{0;(2,0):(0,1)::} 
{\xybox{\xygraph{!{0;(.3,0):(0,2.5)::} 
{\scriptstyle{\eta^F}} *\xycircle<6.5pt,6.5pt>{-}="p0" [d]
{\scriptstyle{\eta^T}} *\xycircle<6.5pt,6.5pt>{-}="p1" [l(2)]
{\scriptstyle{F^l}} *\xycircle<6.5pt,6.5pt>{-}="p2" [rd]
{\scriptstyle{\mu^T}} *\xycircle<7pt,6.5pt>{-}="p3"
"p0" (-"p2",-@/^{1pc}/[d(3.03)r] {F_!})
"p1"-"p3"
"p2" (-[u(2)l] {S},-"p3",-[d(2)l] {F^*})
"p3"-[d] {T}
"p0" [r(3)u] {S^n} -[d(4)] {S^n}}}}
[r] {=} [r]
{\xybox{\xygraph{!{0;(.3,0):(0,2.5)::} 
{\scriptstyle{\eta^F}} *\xycircle<6.5pt,6.5pt>{-}="p0" [l(2)d]
{\scriptstyle{F^l}} *\xycircle<6.5pt,6.5pt>{-}="p2"
"p0" (-"p2",-[d(3.03)r] {F_!})
"p2" (-[u(2)l] {S},-[d(2)r] {T},-[d(2)l] {F^*})
"p0" [r(3)u] {S^n} -[d(4)] {S^n}}}}} \]
and the equation follows by one of $T$'s unit laws. Applying
\[ \varphi^H_{G^*X,Y,n} \comp \Algs R(G^l_{n+1,X},\overline{\gamma}_Y) \comp \overline{G}_{S^{n+1}X,\overline{F}Y} \]
to this component, given that $G^l_{n+1} = (G^lR^n)(RG^l_n)$, reveals $(\overline{\ca R}_G)_{X,n}$ as the component at $X$ of the first term of
\[ \xygraph{!{0;(1.4,0):} 
{\xybox{\xygraph{!{0;(.25,0):(0,2.5)::} 
{\scriptstyle{\eta^R}} *\xycircle<6.5pt,6.5pt>{-}="p0" [r(2)]
{\scriptstyle{G^l_n}} *\xycircle<7pt,7pt>{-}="p1" [dl]
{\scriptstyle{G^l}} *\xycircle<6.5pt,6.5pt>{-}="p2" [r(2)]
{\scriptstyle{\eta^F}} *\xycircle<6.5pt,6.5pt>{-}="p3" [dl]
{\scriptstyle{F^l}} *\xycircle<6.5pt,6.5pt>{-}="p4" [dl]
{\scriptstyle{\gamma}} *\xycircle<6.5pt,6.5pt>{-}="p5" [dl]
{\scriptstyle{\varepsilon^H}} *\xycircle<7pt,7pt>{-}="p6" [d]
{\scriptstyle{\mu^T}} *\xycircle<7pt,6.5pt>{-}="p7"
"p0"-"p2"
"p1" (-[ul] {R^n},-[ur] {G^*},-"p2",-@`{"p1"+(2,0),"p3"+(1,0)}[r(3)d(6)] {S^n})
"p2" (-"p4",-@/_{.5pc}/"p5")
"p3" (-"p4",-@/^{.5pc}/[d(5.03)l(.5)] {F_!})
"p4" (-"p5",-@`{"p4"+(2,-1)}"p7")
"p5"-"p6"-@/^{.5pc}/[u(4.98)l] {H_!}
"p7" (-@/^{1pc}/[u(6)l(3)] {T},-[d] {T})}}}
[r] {=} [r]
{\xybox{\xygraph{!{0;(.25,0):(0,2.5)::} 
{\scriptstyle{G^l_n}} *\xycircle<7pt,7pt>{-}="p1" [dl]
{\scriptstyle{\eta^S}} *\xycircle<6.5pt,6.5pt>{-}="p2" [r(2)]
{\scriptstyle{\eta^F}} *\xycircle<6.5pt,6.5pt>{-}="p3" [dl]
{\scriptstyle{F^l}} *\xycircle<6.5pt,6.5pt>{-}="p4" [d(2)l(2)]
{\scriptstyle{\varepsilon^F}} *\xycircle<6.5pt,6.5pt>{-}="p6" [d]
{\scriptstyle{\mu^T}} *\xycircle<7pt,6.5pt>{-}="p7"
"p2" [u(.5)l] {\scriptstyle{\varepsilon^G}} *\xycircle<6.5pt,6.5pt>{-}="p8"
"p1" (-[ul] {R^n},-[ur] {G^*},-"p8",-@`{"p1"+(2,0),"p3"+(1,0)}[r(3)d(6)] {S^n})
"p2"-"p4"
"p3" (-"p4",-@/^{.5pc}/[d(5.03)l(.5)] {F_!})
"p4" (-"p6",-@`{"p4"+(2,-1)}"p7")
"p6"-@/^{.5pc}/[u(4.98)l(1.75)] {F_!}
"p7" (-@/^{1pc}/[u(6)l(3)] {T},-[d] {T})
"p8"-[u(1.48)l(.5)] {G_!}}}}
[r] {=} [r]
{\xybox{\xygraph{!{0;(.25,0):(0,2.5)::} 
{\scriptstyle{G^c_n}} *\xycircle<6.5pt,6.5pt>{-}="p0" [dr]
{\scriptstyle{\varepsilon^G}} *\xycircle<6.5pt,6.5pt>{-}="p1" [l(2)]
{\scriptstyle{\eta^F}} *\xycircle<6.5pt,6.5pt>{-}="p2" [dl]
{\scriptstyle{\varepsilon^F}} *\xycircle<6.5pt,6.5pt>{-}="p3" [dr]
{\scriptstyle{\eta^T}} *\xycircle<6.5pt,6.5pt>{-}="p4" [dl]
{\scriptstyle{\mu^T}} *\xycircle<6.5pt,6.5pt>{-}="p5"
"p0" (-[u(.98)l] {G_!},-[r(.75)u] {R^n},-"p1"-[u(2)r(1.5)] {G^*},
-@`{"p2"+(-2,1),"p1"+(3,-1.5)}[d(5)r(2)] {S^n})
"p2" (-"p3"-[u(2.98)l(.5)] {F_!},-@/^{.5pc}/[d(4.04)r] {F_!})
"p4"-"p5" (-[d] {T},-@/^{.5pc}/[u(5)l(2)] {T})
}}}
[r] {=} [r]
{\xybox{\xygraph{!{0;(.35,0):(0,2.5)::}
{\scriptstyle{G^c_n}} *\xycircle<6.5pt,6.5pt>{-}="p0" [dr]
{\scriptstyle{\varepsilon^G}} *\xycircle<6.5pt,6.5pt>{-}="p1"
"p0" (-[u(.97)l] {G_!}, -[r(.5)u] {R^n}, -"p1"-[u(2)r] {G^*}, -[d(2)l] {S^n}, [l(2)u(.97)] {F_!} -[d(3)] {F_!}, [l(3)u] {T} -[d(3)] {T})}}}} \]
and the above computation follows by the definition of $\gamma$, the compatibility of $F^l$ and $G^l$ with units, the mateship of $G^l_n$ and $G^c_n$, a triangle equation of $F_! \ladj F^*$, and a unit law of $T$. The component at $X$ of the last term in this computation is the composite of the statement.
\end{proof}
Having identified the simplicial morphism $(\ca R_G)_X$ in Construction \ref{const:simplicial-morphism-underlying-TG} and given an explicit description of its components in Lemma \ref{lem:explicit-RG}, we come to the main result of this section.
\begin{prop}\label{prop:internalisation-as-codescent}
Suppose that $H : (\ca M,R) \to (\ca K,T)$ and $F : (\ca L,S) \to (\ca K,T)$ are internalisable adjunctions of 2-monads, and that $G : H \to F$ is a morphism between them. Then
\[ G^{\dagger}_X = \tn{\CoDesc}((\overline {\ca R}_G)_X).  \]
\end{prop}
\begin{rem}\label{rem:}
The case $X = 1$ of Proposition \ref{prop:internalisation-as-codescent} is our promised explicit description of $T^G : T^R \to T^S$. The diagram
\[ \xygraph{!{0;(3.5,0):(0,.5)::} {TH_!R^21}="pt0" [r] {TH_!R1}="pt1" [r] {TH_!1}="pt2"
"pt2":"pt1"|-{TH_!\eta^R_{1}} "pt1":@<1.5ex>"pt2"^-{\mu^T_{H_!1}T(H^c_{1})} "pt1":@<-1.5ex>"pt2"_-{TH_!(!)} "pt0":@<1.5ex>"pt1"^-{\mu^T_{H_!R1}T(H^c_{R1})} "pt0":"pt1"|-{TH_!\mu^R_{1}} "pt0":@<-1.5ex>"pt1"_-{TH_!R(!)}
"pt0" [d] {TF_!S^21}="pb0" [r] {TF_!S1}="pb1" [r] {TF_!1}="pb2"
"pb2":"pb1"|-{TF_!\eta^S_1} "pb1":@<1.5ex>"pb2"^-{\mu^T_{F_!1}T(F^c_1)} "pb1":@<-1.5ex>"pb2"_-{TF_!(!)} "pb0":@<1.5ex>"pb1"^-{\mu^T_{F_!S1}T(F^c_{S1})} "pb0":"pb1"|-{TF_!\mu^S_1} "pb0":@<-1.5ex>"pb1"_-{TF_!S(!)}
"pt0":"pb0"|-{TF_!(S^2(\varepsilon^G_1)(G^c_2)_{1})} "pt1":"pb1"|-{TF_!(S(\varepsilon^G_1)G^c_{1})} "pt2":"pb2"|-{TF_!(\varepsilon^G_1)}} \]
contains that part of $(\overline{\ca R}_G)_1$ which influences this explicit description of $T^G$.
\end{rem}
\begin{proof}
(\emph{of Proposition \ref{prop:internalisation-as-codescent}})
As recalled above the unit of the adjunction $(-)^{\dagger}_F \ladj J_F$ is denoted $g^F$, its components being lax morphisms between strict $T$-algebras. Suppose that for $(X,x) \in \Algs S$, that a codescent cocone for $\ca R_FX$
\[ \begin{array}{lccr} {q^F_{X,0} : TF_!X \to X^{\dagger}_F} &&& {q^F_{X,1} : q^F_{X,0}\mu^T_{F_!X}T(F^c_X) \to q^F_{X,0}TF_!(x)} \end{array} \]
is given. Then by Lemma 4.3.4 of \cite{Weber-CodescCrIntCat} the equations
\[ \begin{array}{lccr} {g^F_X = F^*(q^F_{X,0})F^*(\eta^T_{F_!X})\eta^F_X} &&& {\overline{g^F_X} = F^*(q^F_{X,1})F^*(\eta^T_{F_!SX})\eta^F_{SX}} \end{array} \]
describe the underlying arrow and lax morphism coherence 2-cell of
\[ g^F_X:X \longrightarrow \overline{F}X^{\dagger}_F \]
in terms of the given codescent data.

By the definition of $G^{\dagger}_X$ it suffices to show that
\[ \xygraph{!{0;(2,0):(0,.5)::} {\overline{G}X}="p0" [r] {\overline{H}(\overline{G}X)^{\dagger}_H}="p1" [d] {\overline{H}X^{\dagger}_F}="p2" [l] {\overline{G}\overline{F}X^{\dagger}_F}="p3" "p0":"p1"^-{g^H_{\overline{G}X}}:"p2"^-{\overline{H}\tn{\CoDesc}((\overline {\ca R}_G)_X)}:@{<-}"p3"^-{\overline{\gamma}_{X^{\dagger}_F}}:@{<-}"p0"^-{\overline{G}g^F_X}} \]
commutes in $\Algl R$. By definition the strict $T$-morphism $\tn{\CoDesc}((\overline {\ca R}_G)_X)$ is defined uniquely by
\[ \begin{array}{lccr} {\tn{\CoDesc}((\overline {\ca R}_G)_X)q^H_{\overline{G}X,0} = q^F_{X,0}TF_!(\varepsilon^G_X)} &&&
{\tn{\CoDesc}((\overline {\ca R}_G)_X)q^H_{\overline{G}X,1} = q^F_{X,1}TF_!(S(\varepsilon^G_X)G^c_{G^*X}).} \end{array} \]
Thus the lax $R$-algebra morphism $\overline{H}(\tn{CoDesc}(\overline{\ca R}_G)_X)g^H_{\overline{G}X}$ has, by Lemma 4.3.4 of \cite{Weber-CodescCrIntCat}, underlying $1$ and $2$-cell data given by 
\[ \begin{array}{lccr} {H^*(q^F_{X,0})H^*TF_!(\varepsilon^G_X)H^*(\eta^T_{H_!G^*X})\eta^H_{G^*X}}
&&&
{H^*(q^F_{X,1})H^*TF_!(S(\varepsilon^G_X)G^c_{G^*X})H^*(\eta^T_{H_!RG^*X})\eta^H_{G^*RX}.} \end{array} \]
The underlying $1$-cell of the lax $R$-morphism $\overline{\gamma}_{X^{\dagger}_F}\overline{G}(g^F_X)$ is the composite on the left hand side of
\[ \begin{array}{lcr} {\overline{\gamma}_{X^{\dagger}_F}G^*F^*(q^F_{X,0})G^*F^*(\eta^T_{F_!X})G^*(\eta^F_X)} & = &
{H^*(q^F_{X,0})H^*(\eta^T_{F_!X})\overline{\gamma}_{F_!X}G^*(\eta^F_X)} \end{array} \]
which because of the naturality of $\overline{\gamma}$ equals the expression on the right. Similarly the 2-cell data for $\overline{\gamma}_{X^{\dagger}_F}\overline{G}(g^F_X)$ is $H^*(q^F_{X,1})H^*(\eta^T_{F_!SX})\overline{\gamma}_{F_!SX}G^*(\eta^F_{SX})$.
To reconcile these two lax $R$-morphisms it suffices to show that the outside of
\[ \xygraph{!{0;(2,0):(0,.5)::} {G^*}="p0" [r] {G^*F^*F_!}="p1" [r] {H^*F_!}="p2" [d] {H^*TF_!}="p3" [l] {H^*TH_!G^*}="p4" [l] {H^*H_!G^*}="p5" "p0":"p1"^-{G^*\eta^F}:"p2"^-{\overline{\gamma}F_!}:"p3"^-{H^*\eta^TF_!}:@{<-}"p4"^-{H^*TF_!\varepsilon^G}:@{<-}"p5"^-{H^*\eta^TG^*}:@{<-}"p0"^-{\eta^HG^*} "p5":"p2"|-{H^*F_!\varepsilon^G}} \]
commutes, and we note that the bottom inner region of this diagram commutes by naturality. To establish commutativity of the top inner region, we note that by the universal property of $\eta^G$ it suffices to show
\begin{equation}\label{eq:last-eq-needed-for-codescent-lemma}
(H^*F_!\varepsilon^GG_!)(\eta^HG^*G_!)\eta^G = (\overline{\gamma}H_!)(G^*\eta^FG_!)\eta^G.
\end{equation}
Recall that the canonical isomorphism $\overline{\gamma}:G^*F^* \to H^*$ is the witness to the fact that both $G^*F^*$ and $H^*$ are right adjoints to $F_!G_! = H_!$. Thus one formula which determines this canonical isomorphism uniquely says that the right hand side of (\ref{eq:last-eq-needed-for-codescent-lemma}) equals $\eta^H$. The calculation
\[ {(H^*F_!\varepsilon^GG_!)(\eta^HG^*G_!)\eta^G} = {(H^*F_!\varepsilon^GG_!)(H^*H_!\eta^G)\eta^H} = {\eta^H} \]
in which the first step follows by naturality, and the second by a triangle equation for $G_! \ladj G^*$, establishes the same for the left hand side of (\ref{eq:last-eq-needed-for-codescent-lemma}).
\end{proof}

\subsection{Proof of Theorem \ref{thm:main}}
\label{ssec:alg-square-before-codescent}
Given an internalisable adjunction $F : (\ca L,S) \to (\ca K,T)$ of 2-monads as in Proposition \ref{prop:internalisation-as-codescent}, if $T$ preserves codescent objects, then by Corollary 4.3.6 of \cite{Weber-CodescCrIntCat} one can write the underlying object and $T$-algebra action of $T^S$ as
\[ \begin{array}{lccr} {T^S = \tn{CoDesc}(U^T\ca R_F1)} &&& {a^S = \tn{CoDesc}(\sigma_T\ca R_F1)} \end{array} \]
and similarly for $T^R$. Putting this together with Proposition \ref{prop:internalisation-as-codescent} gives
\begin{prop}\label{prop:int-algsquare-via-codescent}
Suppose that $H : (\ca M,R) \to (\ca K,T)$ and $F : (\ca L,S) \to (\ca K,T)$ are internalisable adjunctions of 2-monads, and that $G : H \to F$ is a morphism between them. Suppose moreover that $T$ preserves codescent objects. Then the algebra square for $T^G$ is obtained by applying $\tn{CoDesc}$ to the commutative square
\[ \xygraph{!{0;(3,0):(0,.3333)::} {TU^T\ca R_H1}="p0" [r] {TU^T\ca R_F1}="p1" [d] {U^T\ca R_F1}="p2" [l] {U^T\ca R_H1}="p3" "p0":"p1"^-{TU^T(\overline{\ca R}_G)_1}:"p2"^-{\sigma_T\ca R_F1}:@{<-}"p3"^-{U^T(\overline{\ca R}_G)_1}:@{<-}"p0"^-{\sigma_T\ca R_H1}} \]
of simplicial objects in $\ca K$.
\end{prop}
Before proceeding to the proof of Theorem \ref{thm:main}, we point out that there is a more elementary situation in which the exactness of $T^G$ is guaranteed, but which does not require the developments of Sections \ref{ssec:codesc-crossed-dblcat}-\ref{ssec:exactness-of-internalisations}. In particular this case suffices for \cite{BataninBerger-HtyThyOfAlgOfPolyMnd}, and justifies Theorems 5.14 and 5.15 therein.

We denote by $\Cat_{\tn{pb}}$ the 2-category of categories with pullbacks, pullback preserving functors and cartesian natural transformations, and by $\Cat(-) : \Cat_{\tn{pb}} \to \TwoCat$ the 2-functor which sends any category $\ca E$ with pullbacks to the 2-category $\Cat(\ca E)$ of category objects in $\ca E$. Recall that adjunctions of monads and morphisms thereof can be defined in any 2-category. Suppose we are given adjunctions of monads in $\Cat_{\tn{pb}}$ and morphisms thereof as on the left in
\[ \xygraph{
{(\ca C,R')}="p0" [r(2)] {(\ca D,S')}="p1" [dl] {(\ca E,T')}="p2" "p0":"p1"^-{G'}:"p2"^-{F'}:@{<-}"p0"^-{H'}
"p1" [r(3)]
{(\ca M,R)}="p0" [r(2)] {(\ca L,S)}="p1" [dl] {(\ca K,T)}="p2" "p0":"p1"^-{G}:"p2"^-{F}:@{<-}"p0"^-{H}} \]
and we denote the result of applying $\Cat(-)$ to this as on the right. So in particular, the 2-monads $R$, $S$ and $T$ are $\Cat(R')$, $\Cat(S')$ and $\Cat(T')$ respectively.
\begin{thm}\label{thm:exact-internalisations-easy}
In the context just described $T^G:T^R \to T^S$ is exact.
\end{thm}
\begin{proof}
Since in this context the simplicial objects appearing in the square of Proposition \ref{prop:int-algsquare-via-codescent} are componentwise discrete category objects, and so taking codescent in this case is simply a matter of interpretting this as a square in $\ca K = \Cat(\ca E)$. Recall that an internal functor $f:X \to Y$ is a discrete fibration iff the square
\[ \xygraph{!{0;(1.5,0):(0,.6667)::} {X_1}="p0" [r] {X_0}="p1" [d] {Y_0}="p2" [l] {Y_1}="p3" "p0":"p1"^-{d_0}:"p2"^-{f_0}:@{<-}"p3"^-{d_0}:@{<-}"p0"^-{f_1}} \]
is a pullback. This is so for $\sigma_T\ca R_F1$ as an internal functor since $\mu^{T'}$ is cartesian natural, and so $\sigma_T\ca R_F1$ is a discrete fibration. Moreover the cartesianness of $\mu^{T'}$ also ensures that the square of Proposition \ref{prop:int-algsquare-via-codescent} is a pullback square in $\ca K$. Thus the result follows by Proposition \ref{prop:exact-pullbacks}.
\end{proof}
Recall in the context of Theorem \ref{thm:main}, one has a commutative triangle of polynomial adjunctions of 2-monads
\[ \xygraph{!{0;(1.5,0):(0,.6667)::} {(\Cat/I,R)}="p0" [r(2)] {(\Cat/J,S)}="p1" [dl] {(\Cat/K,T)}="p2" "p0":"p1"^-{G}:"p2"^-{F}:@{<-}"p0"^-{H}} \]
which gives, by Construction \ref{const:T^G}, the strict morphism of strict $T$-algebras $T^G : T^R \to T^S$. Moreover one assumes that $I$, $J$ and $K$ are discrete and $p_T$ is a discrete opfibration with finite fibres. Theorem \ref{thm:main} says that under these conditions $T^G$ is exact.
\begin{proof}
(\emph{of Theorem \ref{thm:main}}).
The square of Proposition \ref{prop:int-algsquare-via-codescent} lives in $[\Delta^{\op},\Cat/K]$, in the present context, and we denote this square as $\ca S$. The algebra square of $T^G$ is $\CoDesc(\ca S)$ by Proposition \ref{prop:int-algsquare-via-codescent}. Pointwise left Kan extensions, comma objects and codescent objects in $\Cat/K$ are formed fibrewise, and so by Theorem \ref{thm:hard-exactness}, it suffices to show
\begin{enumerate}
\item $\ca S$ is a pullback square in $\CrIntCat {\Cat/K}$,
\label{pfstep:S-pb-CrIntCat}
\item $\sigma_T\ca R_F1$ is a discrete fibration, and
\label{pfstep:dfib}
\item $U^T(\overline{\ca R}_G)_1$ is an objectwise opfibration.
\label{pfstep:objwise-opfib}
\end{enumerate}

(\ref{pfstep:S-pb-CrIntCat}): Since $p_T$ is a discrete opfibration with finite fibres and such functors are pullback stable, $p_S$ and $p_R$ are also discrete opfibrations with finite fibres. By Theorem 4.4.5 of \cite{Weber-PolynomialFunctors} $R$, $S$ and $T$ are opfamilial 2-monads. Since $H_!$ and $F_!$ are of the form $\Sigma_h$ and $\Sigma_f$, they are also opfamilial 2-functors. By Theorem 4.5.1 of \cite{Weber-PolynomialFunctors} $T$ preserves all sifted colimits, and thus in particular codescent objects. The object $G_!1$ of $\Cat/J$ is just $g : I \to J$, which is discrete as an object of $\Cat/J$ since $I$ is a discrete category. We will now see that as a consequence of these formal properties, $\ca S$ lives in $\CrIntCat {\Cat/K}$.

As recalled in Section \ref{ssec:exact-nat} opfamilial 2-functors preserve split opfibrations and morphisms thereof. By Proposition 4.4.1 of \cite{Weber-CodescCrIntCat} the simplicial objects appearing in the square of Proposition \ref{prop:int-algsquare-via-codescent} are category objects. The codescent-relevant part of $\sigma_T\ca R_F1$ is
\[ \xygraph{!{0;(3.5,0):(0,.5)::}
{T^2F_!S^21}="pt0" [r] {T^2F_!S1}="pt1" [r] {T^2F_!1}="pt2"
"pt2":"pt1"|-{T^2F_!\eta^S_1} "pt1":@<1.5ex>"pt2"^-{T\mu^T_{F_!1}T^2(F^c_1)} "pt1":@<-1.5ex>"pt2"_-{T^2F_!(!)} "pt0":@<1.5ex>"pt1"^-{T\mu^T_{F_!S1}T^2(F^c_{S1})} "pt0":"pt1"|-{T^2F_!\mu^S_1} "pt0":@<-1.5ex>"pt1"_-{T^2F_!S(!)}
"pt0" [d] {TF_!S^21}="pb0" [r] {TF_!S1}="pb1" [r] {TF_!1}="pb2"
"pb2":"pb1"|-{TF_!\eta^S_1} "pb1":@<1.5ex>"pb2"^-{\mu^T_{F_!1}T(F^c_1)} "pb1":@<-1.5ex>"pb2"_-{TF_!(!)} "pb0":@<1.5ex>"pb1"^-{\mu^T_{F_!S1}T(F^c_{S1})} "pb0":"pb1"|-{TF_!\mu^S_1} "pb0":@<-1.5ex>"pb1"_-{TF_!S(!)}
"pt0":"pb0"|-{\mu^T_{F_!S^21}} "pt1":"pb1"|-{\mu^T_{F_!S1}} "pt2":"pb2"|-{\mu^T_{F_!1}}} \]
and since every map into a discrete object is a split opfibration with chosen opcartesians exactly the identity 2-cells, $\eta^S_1$ and $\mu^S_1$ are morphisms of split opfibrations over $1$. Since $TF_!$ and $T^2F_!$ are opfamilial they send these to morphisms of split opfibrations over $TF_!1$ and $T^2F_!1$ respectively, thus exhibiting $U^T\ca R_F1$ and $TU^T\ca R_F1$ as crossed internal categories. Since $\mu^T$ is an opfamilial natural transformation, $(\mu^T_{F_!S1},\mu^T_{F_!1}):T^2F_!t_{S1} \to TF_!t_{S1}$ is a morphism of split opfibrations by Proposition 4.3.6 of \cite{Weber-PolynomialFunctors}, and so $\sigma_T\ca R_F1$ is a crossed internal functor. Similarly $\sigma_T\ca R_H1$ is a crossed internal functor.

The codescent-relevant part of $U^T(\overline{\ca R}_G)_1$ is
\[ \xygraph{!{0;(3.5,0):(0,.5)::} {TH_!R^21}="pt0" [r] {TH_!R1}="pt1" [r] {TH_!1}="pt2"
"pt2":"pt1"|-{TH_!\eta^R_{1}} "pt1":@<1.5ex>"pt2"^-{\mu^T_{H_!1}T(H^c_{1})} "pt1":@<-1.5ex>"pt2"_-{TH_!(!)} "pt0":@<1.5ex>"pt1"^-{\mu^T_{H_!R1}T(H^c_{R1})} "pt0":"pt1"|-{TH_!\mu^R_{1}} "pt0":@<-1.5ex>"pt1"_-{TH_!R(!)}
"pt0" [d] {TF_!S^21}="pb0" [r] {TF_!S1}="pb1" [r] {TF_!1}="pb2"
"pb2":"pb1"|-{TF_!\eta^S_1} "pb1":@<1.5ex>"pb2"^-{\mu^T_{F_!1}T(F^c_1)} "pb1":@<-1.5ex>"pb2"_-{TF_!(!)} "pb0":@<1.5ex>"pb1"^-{\mu^T_{F_!S1}T(F^c_{S1})} "pb0":"pb1"|-{TF_!\mu^S_1} "pb0":@<-1.5ex>"pb1"_-{TF_!S(!)}
"pt0":"pb0"|-{TF_!(S^2(\varepsilon^G_1)(G^c_2)_{1})} "pt1":"pb1"|-{TF_!(S(\varepsilon^G_1)G^c_{1})} "pt2":"pb2"|-{TF_!(\varepsilon^G_1)}} \]
and to say that this underlies a crossed internal functor is to say that the square on the left in
\[ \xygraph{{\xybox{\xygraph{!{0;(2.5,0):(0,.4)::} {TH_!R1}="p0" [r] {TF_!S1}="p1" [d] {TF_!1}="p2" [l] {TH_!1}="p3" "p0":"p1"^-{TF_!(S(\varepsilon^G_1)G^c_{1})}:"p2"^-{TF_!(t_{S1})}:@{<-}"p3"^-{TF_!(\varepsilon^G_1)}:@{<-}"p0"^-{TH_!(t_{R1})}}}}
[r(5)]
{\xybox{\xygraph{!{0;(2.5,0):(0,.4)::} {G_!R1}="p0" [r] {S1}="p1" [d] {1}="p2" [l] {G_!1}="p3" "p0":"p1"^-{S(\varepsilon^G_1)G^c_{1}}:"p2"^-{t_{S1}}:@{<-}"p3"^-{\varepsilon^G_1}:@{<-}"p0"^-{G_!(t_{R1})}}}}} \]
underlies a morphism of split opfibrations $TH_!(t_{R1}) \to TF_!(t_{S1})$. Now this square is the result of applying $TF_!$ to the square on the right. Since $G_!1$ is discrete, the chosen $G_!t_{R1}$-opcartesian 2-cells exactly the identities, and so the square on the right in the previous display is a morphism of split opfibrations $G_!(t_{R1}) \to t_{S1}$. Thus since $TF_!$ and $T^2F_!$ are opfamilial, the square on the left in the previous display, and also $T$ of that square, are morphisms of split opfibrations. Thus $U^T(\overline{\ca R}_G)_1$ and $TU^T(\overline{\ca R}_G)_1$ are crossed internal functors, and so $\ca S$ does indeed live in $\CrIntCat {\Cat/K}$. It is a pullback by Lemma \ref{lem:pbs-CrIntCat} and since $\mu^T$ is cartesian.

(\ref{pfstep:dfib}): Follows immediately from the cartesianness of $\mu^T$.

(\ref{pfstep:objwise-opfib}): To see that $(U^T(\overline{\ca R}_G)_1)_0$ is an opfibration, note that it is the result of applying $TF_!$ to the unique morphism $t_{G_!1}:G_!1 \to 1$. Like any map into a discrete object, $t_{G_!1}$ is an opfibration. Since $TF_!$ is opfamilial, and opfamilial 2-functors preserve opfibrations, the result follows.
\end{proof}
%


\end{document}